\documentclass[12pt]{article}
\usepackage[margin=35mm]{geometry}
\usepackage{amssymb, amsmath, amsthm, amsfonts}
\usepackage{stackengine,scalerel,calc}
\stackMath
\newcommand\Rwh[1]{
	\savestack{\tmpbox}{\stretchto{\scaleto{\scalerel*[\widthof{\ensuremath{#1}}]{\kern-.6pt\bigwedge\kern-.6pt}{\rule[-\textheight]{1ex}{\textheight}}}{\textheight}}{.5ex}}
	\stackon[1pt]{#1}{\tmpbox}
}

\DeclareMathOperator\supp{\mathrm{supp}}
\DeclareMathOperator{\graph}{graph}

\newtheorem{theorem}{Theorem}[section]
\newtheorem{lemma}[theorem]{Lemma}
\newtheorem{proposition}[theorem]{Proposition}
\newtheorem{corollary}[theorem]{Corollary}
\theoremstyle{definition}
\newtheorem{definition}[theorem]{Definition}

\newtheorem{example}[theorem]{Example}
\theoremstyle{remark}
\newtheorem{remark}[theorem]{Remark}
\usepackage{tikz-cd}
\tikzcdset{
	every label/.append style={font=\small}
}
\newcommand\tto{\rightrightarrows}
\newcommand\nnh{}
\usepackage{enumitem}
\setlist[itemize]{nosep}
\setlist[enumerate]{nosep}
\newcommand\shH{\mathscr{H}}
\newcommand\fg{\mathfrak{g}}

\DeclareMathOperator\GL{GL}

\usepackage{mathrsfs}

\newcommand\pdt{\partial_t}
\newcommand\pds{\partial_s}
\newcommand\drt{\delta^R_t}
\newcommand\drs{\delta^R_s}

\newcommand\shA{\mathscr{A}}
\newcommand\shcinf{\mathscr{C}^\infty}
\newcommand\shTM{{\mathscr{T}_M}}

\newcommand\shE{\mathscr{E}}
\DeclareMathOperator{\Hom}{Hom}
\newcommand\ddR{d_{\mathrm{dR}}}
\newcommand\ddRs{d_{\mathrm{s,dR}}}
\newcommand\wgrisv{{\Omega^\bullet_{\G,\mathrm{risv}}}}
\newcommand\wgrisvk{{\Omega^k_{\G,\mathrm{risv}}}}
\newcommand\ora[1]{\overrightarrow{#1}}
\newcommand\shR{\mathscr{R}_{\G}}
\newcommand\rld{\delta^R}

\newcommand\wt{\widetilde}
\newcommand\G{{\mathcal{G}}}

\newcommand\R{\mathbb{R}}
\newcommand\N{\mathbb{N}}

\newcommand\src{\mathbf{s}}
\newcommand\tgt{\mathbf{t}}
\newcommand\one{\mathbf{1}}
\newcommand\anch{{\mathbf{a}}}
\newcommand\inv{\mathbf{i}}
\newcommand\mult{\mathbf{m}}
\newcommand\mset[1]{{#1}^{(2)}}
\DeclareMathOperator{\id}{id}

\DeclareMathOperator\ev{ev}

\newcommand\sM{\mathsf{M}}
\newcommand\sN{\mathsf{N}}

\newcommand\sS{\mathsf{S}}

\newcommand\sC{\mathsf{C}}

\newcommand\sE{\mathsf{E}}
\newcommand\sF{\mathsf{F}}
\newcommand\sH{\mathsf{H}}

\newcommand{\wh}[1]{\widehat{#1}}
\newcommand{\cH}{\mathcal{H}}

\DeclareMathOperator\pr{pr}
\newcommand{\Lie}{\operatorname{Lie}}
\newcommand\LG{{\Lie(\G)}}
\newcommand{\LB}[1][\cdot \hspace{1pt} , \cdot]{\left[\hspace{1pt} #1 \hspace{1pt} \right]}

\usepackage{titlesec}
\titleformat{\section}
  {\normalfont\large\bfseries}{\thesection}{1em}{}
\titleformat{\subsection}
  {\normalfont\normalsize\bfseries}{\thesubsection}{1em}{}
\titleformat{\subsubsection}
  {\normalfont\normalsize\itshape}{\thesubsubsection}{1em}{}

\usepackage{newtxtext}
\usepackage[scr=rsfso]{mathalpha}

\usepackage{xurl}

\usepackage[
backend=biber,
style=alphabetic,
maxbibnames=99,
maxcitenames=99,
giveninits=false
]{biblatex}

\AtBeginBibliography{%
\sloppy
\setlength{\emergencystretch}{3em}%
}

\renewbibmacro{in:}{%
\ifentrytype{article}
{}
{\printtext{\bibstring{in}\intitlepunct}}}

\usepackage[colorlinks=true, linkcolor=blue, urlcolor=blue, citecolor=blue]{hyperref}
\hypersetup{breaklinks=true}

\newcommand\Emph[1]{{\emph{#1}}}

\begin{document}
\title{A category of locally convex Lie algebroids}
\author{A. G. Shaltut\footnote{
    {\tt ahmed.shaltut@queensu.ca}\\Dept. of Mathematics and Statistics, Queen's University at Kingston.}}
\date{}
\maketitle

\begin{abstract}
We study first-order locally convex Lie algebroids in the setting of Bastiani calculus. The first-order condition is automatic in finite rank, but is an additional regularity hypothesis for general locally convex vector bundles. We define sheaves of scalar-valued and vector-valued Lie algebroid forms as fiberwise continuous alternating maps with smooth local representatives. Morphisms are defined by requiring the induced pullback on inverse-image sheaves of scalar-valued forms to commute with the Lie algebroid differentials. We prove that these algebroids form a category, study representations and cohomology sheaves, and show that the Lie algebroids and morphisms of locally convex Lie groupoids belong to this category. We also give two applications. For a compact manifold $K$, the current algebroid $C^\infty(K,A)\to C^\infty(K,M)$ of a first-order Banach Lie algebroid $A\to M$ is a first-order Fr\'echet Lie algebroid, provided 
$M$ admits local addition. Finally, we prove a Banach Lie II theorem: a morphism from  the Lie algebroid of
a locally convex Lie groupoid  to that of  a Banach-Lie groupoid, both with Hausdorff source fibers, integrates uniquely to a groupoid morphism when the source groupoid is source-connected and source-simply connected.
\end{abstract}

\section*{Introduction}

Lie algebroids were introduced by Pradines as infinitesimal objects associated with differentiable groupoids \cite{Pradines1967Lie}. They generalize Lie algebras and tangent bundles and the theory is now standard; see \cite{Mackenzie2005GeneralTheory}. In finite dimensions, the space of $k$-forms on a Lie algebroid $A\to M$ is $\Omega_A^k(M)=\Gamma(\wedge^k A^*)$, and the anchor and bracket determine the Chevalley-Eilenberg differential $d_A$. This gives the de Rham complex for $A=TM$ and Lie algebra cohomology when $M$ is a point. 
The associated differential calculus is classical \cite{Mackenzie2005GeneralTheory,Marle2008Calculus}. Equivalently, a finite-dimensional Lie algebroid structure can be described by a degree-one differential on $\Gamma(\wedge^\bullet A^*)$ or by a homological vector field on $A[1]$ \cite{Vaintrob1997Homological}. Several definitions of Lie algebroid morphisms are available in finite dimensions \cite{HigginsMackenzie1990Algebraic,Mackenzie2005GeneralTheory,Vaintrob1997Homological}; the pullback-of-forms definition has also been extended to Banach Lie algebroids in \cite{Anastasiei2011BanachLieAlgebroids}.

Representations of Lie algebroids, namely flat $A$-connections on vector bundles, give cohomology complexes with coefficients. In finite dimensions these are related to characteristic classes, modular classes, van Est maps and representations up to homotopy; see \cite{Kubarski1991ChernWeil,Fernandes2002Holonomy,EvensLuWeinstein1999Transverse,Crainic2003Cohomology,AbadCrainic2012Representations}.

The locally convex setting is more delicate. A locally convex vector bundle need not admit a finite local frame, and bundles of continuous alternating multilinear forms need not possess a canonical smooth vector bundle structure. We therefore define Lie algebroid forms as sheaves of fiberwise continuous alternating maps with smooth local representatives, without choosing topologies on exterior dual bundles. The construction is carried out in Bastiani calculus.

The main hypothesis is the first-order condition of Definition~\ref{def:liealgebroid}. It holds for finite-dimensional Lie algebroids and, by Corollary~\ref{cor:finite-dim-fiber}, for finite-rank locally convex Lie algebroids with a bracket on local sections satisfying the stated conditions. In a local trivialization, it requires an explicit first-order bracket formula with a Bastiani smooth, fiberwise continuous bilinear term. For general locally convex bundles, we impose this regularity rather than deducing it from the Leibniz rule. The distinction is pertinent in infinite-dimensional geometry: queer Poisson brackets can depend on higher derivatives \cite{BeltitaGolinskiTumpach2018QueerPoisson}, and related questions for Banach Lie algebroids have been discussed in \cite{GolinskiJakimowicz2025Predual}. Our first-order condition excludes higher-order dependence and also requires smooth dependence of the local bilinear term.
In the Banach setting, related locality and
first-order conditions were studied in
\cite{CabauPelletier2012AlmostLie}.
We develop a sheaf-based differential calculus and study morphisms  for first-order locally convex Lie algebroids in Bastiani calculus, with applications to current algebroids and a Lie II theorem with Banach target.

All base manifolds are Hausdorff, but the arrow manifold of a Lie groupoid need not be Hausdorff. The local constructions concerning Lie groupoids and their algebroids do not require Hausdorff source fibers. The latter hypothesis is imposed in Section~\ref{sec:Lie-2nd-Banach}, where global uniqueness of reconstructed paths is needed.

Current groupoids provide infinite-dimensional examples. Lie groupoids of smooth mappings into a Lie groupoid were studied in \cite{AmiriGloecknerSchmeding2020CurrentGroupoids}. We establish a corresponding first-order result for current algebroids of first-order Banach Lie algebroids. For compact $K$ and a local addition on the base, this yields Fr\'echet Lie algebroids $C^\infty(K,A)\to C^\infty(K,M)$.

Our second application is a Lie II theorem with Banach target. In finite dimensions, the integrability problem was solved by Crainic and Fernandes \cite{CrainicFernandes2003Integrability}; their path-space methods also provide a model for Lie II. Integration of algebroid comorphisms is considered in \cite{CattaneoDherinWeinstein2013Comorphisms}. Relevant infinite-dimensional work includes Banach Lie algebroids \cite{Anastasiei2011BanachLieAlgebroids}, almost Lie structures on anchored Banach bundles \cite{CabauPelletier2012AlmostLie}, and Banach-Lie groupoids \cite{BeltitaGolinskiJakimowiczPelletier2019}. We prove that a first-order morphism from the Lie algebroid of a source-connected and source-simply connected  locally convex Lie groupoid to that of a Banach-Lie groupoid integrates uniquely, provided both have Hausdorff source fibers. The theorem concerns Lie algebroids of {given} groupoids and does not show integrability of arbitrary locally convex Lie algebroids. The proof uses right-logarithmic Banach ODEs in the target, with smooth dependence on parameters, and smoothing of source-fiber homotopies. Hausdorffness is required of the source fibers, not of the arrow manifolds.

The paper is organized as follows. Section~\ref{sec:prelim} fixes the conventions for Bastiani calculus, bundles, and mapping spaces. Section~\ref{sec:firstorderLiealgbd} develops first-order Lie algebroids and their differential forms, including the current construction. Section~\ref{sec:liealgbdmorphism} treats morphisms, representations, coefficient-valued cohomology, and current functoriality. Section~\ref{sec:Lie-2nd-Banach} proves the Banach-target Lie II theorem and discusses its covering-groupoid consequences.

\section{Preliminaries}\label{sec:prelim}
We use Bastiani calculus throughout this paper. All locally convex model spaces are Hausdorff. Unless stated otherwise, manifolds are Hausdorff. The only non-Hausdorff manifolds considered here are arrow spaces of Lie groupoids. In Section~\ref{sec:Lie-2nd-Banach},
we assume further that Lie groupoids have Hausdorff source fibers. Indeed,
all definitions of vector fields, tangent bundles, submanifolds, and forms are local and therefore remain meaningful for non-Hausdorff manifolds.
For background on infinite-dimensional differential geometry, Bastiani calculus, submersions and immersions, and mapping manifolds, we refer to \cite{GloecknerNeeb2026InfiniteDimensional}.
In Section~\ref{sec:Lie-2nd-Banach}, we employ local existence and uniqueness for Banach ODEs with parameters in a locally convex space. We refer to  \cite[\S2.4]{GloecknerNeeb2026InfiniteDimensional}, particularly, Theorem~2.4.29 for the relevant ODE results.

For a topological space $M$, a chart  about $x \in M$ is a homeomorphism $\phi\colon U_\phi\to V_\phi$, where  $U_\phi$
is an open neighborhood of $x$ and  $V_\phi$ is an open subset of a locally convex space $\sM_\phi$. We shall always adopt this notation for charts unless defined otherwise. A smooth atlas for $M$ is a collection of charts covering $M$ whose transition maps are smooth in the Bastiani sense. 
A \Emph{smooth manifold} is a topological space equipped with a maximal smooth atlas. On connected components, the modeling spaces are topologically isomorphic, so we may speak of a single model space up to a topological isomorphism.

Suppose that $M$ is a  smooth manifold.
Consider a vector field $X\in\Gamma(TM)$ and 
a map $f\in C^\infty(U,\sF) $ where $U\subseteq M$ is open and $\sF$ is a locally convex space. 
We denote $X.f:= \pr_2\circ  Tf\circ X\in C^\infty(U,\sF)$.
Suppose that we have chart $\phi$ with $U_\phi=U$. We denote $X^\phi:=\pr_2\circ T\phi\circ X\circ \phi^{-1}\colon V_\phi\to \sM_{\phi}$. In this case, the map $X.f$ is given locally by $y\mapsto d (f\circ \phi^{-1})(y;X^\phi(y))$.

A subset $S \subseteq M$ is a  \Emph{submanifold} if it is locally modeled on a closed subspace of the modeling space: for every $x \in S$, there is a chart $\phi$ such that $\phi(U_\phi \cap S) = V_\phi \cap \sS_\phi$ for some closed subspace $\sS_\phi \subseteq \sM_\phi$. If for every $x\in S$, there exists a submanifold chart $\phi$ for which  $\sS_\phi$ is complemented in $\sM_\phi$, then $S$ is a \Emph{split submanifold}. These charts define the manifold structure on $S$ with the subspace topology.

A smooth map $f \colon M \to N$ is a \Emph{submersion} if, for every $x \in M$, there exist charts $\phi$ about $x$ and $\psi$ about $f(x)$ such that
$f(U_\phi)\subseteq U_\psi$ and 
 $\psi \circ f \circ \phi^{-1}$ restricts to a projection $p \colon \sM_\phi \cong \sN_\psi \times \sC \to \sN_\psi$. Equivalently, $\psi\circ f\circ \phi^{-1} = \pi|_{V_{\phi}}$ where $\pi\colon \sM_{\phi}\to \sN_{\psi}$ is a continuous linear surjective map  which admits  a continuous linear right inverse.
Similarly, $f$ is an \Emph{immersion} if it locally looks like the inclusion of a complemented subspace.
That is, if $\psi\circ f\circ\phi^{-1} = \iota|_{V_{\phi}}$,
where $\iota \colon \sM_{\phi}\to \sN_{\psi}$ is a topological embedding 
onto a complemented subspace of $\sN_\psi$. In particular, an immersion restricts to an embedding locally about each point.
 An immersion that is also a topological embedding is called an \Emph{embedding}.

A \Emph{trivial vector bundle} is a projection $\pr_1\colon M\times\sE\to M$, 
where
 $M$ is a smooth manifold and $\sE$ is a locally convex space.
Let $N\times \sF$ be another trivial  vector bundle. A \Emph{trivial vector bundle morphism} is a smooth map $\Psi\colon M\times \sE\to N\times \sF$ which is given in terms of a pair of  maps $f\colon M\to N$ and  $\wh \Psi\colon M\times \sE\to \sF$  such that $\Psi(x,u) = (f(x), \wh \Psi(x,u))$  and $\Psi_x:=\wh \Psi(x,\cdot)$ is linear.%

The composition of two vector bundle morphisms is defined in
the obvious way. Let
$\Psi_i\colon M_i\times \sE_i\to M_{i+1}\times \sE_{i+1}$, $i=1,2$, be  vector bundle morphisms. The composition is given by
$$(\Psi_2\circ \Psi_1)(x,u) = ((f_2\circ f_1)(x), \wh{\Psi_2}(f_1(x), \wh{ \Psi_1}(x,u)))
$$
 for all $(x,u)\in M_1 \times \sE_1$. That is, $\Rwh{\Psi_2\circ \Psi_1} = \wh{\Psi_2} \circ ((f_1\circ \pr_1), \wh{\Psi_1})$ which is smooth and linear in the second argument (so that, fiberwise, $(\Psi_2\circ \Psi_1)_x = ({\Psi_2})_{f_1(x)}\circ (\Psi_1)_x$ is linear). 
 
 We say that $\Psi$ is a \Emph{vector bundle isomorphism} if 
there exists a  vector bundle morphism $\Psi^{-1}\colon N\times \sF\to M\times \sE$ 
which is a smooth inverse for $\Psi$.
In particular, $f$ is a  diffeomorphism, 
$\wh{\Psi^{-1}}(f(x), \wh{\Psi}(x,u))  = u$ for all $ 
(x,u)\in M\times \sE$, 
and $\wh \Psi(f^{-1}(y), \wh{\Psi^{-1}}(y,w) ) = w$  for all $(y,w)\in N\times \sF$. 
That is, $\Psi^{-1}$  can be written as
$$\Psi^{-1}(y, w) =  (f^{-1}(y), (\Psi_{f^{-1}(y)})^{-1}(w))\quad \forall (y,w)\in N\times\sF.$$ 

	A \Emph{ vector bundle} is a  manifold $E$ together with
	a  surjective map $p\colon E\to M$ onto a  manifold $M$ such that,
	for each $x\in M$, the fiber $E_x:=p^{-1}(x)$
	is a locally convex space, and such that there
	exists a family of   diffeomorphisms (\Emph{local trivializations}) $\{\phi\colon p^{-1}(U_\phi)\to U_\phi\times \sF_\phi\}$
	such that $\{\sF_\phi\}$ are locally convex spaces and $\{U_\phi\}$ is
	an open cover of $M$,
	 satisfying the following conditions: 
	\begin{enumerate}\itemsep0em
		\item
		 $\pr_2\circ \phi|_{p^{-1}(x)}\colon p^{-1}(x)\to \sF_\phi$ is a
		 topological isomorphism for all $x\in U_\phi$;
		\item $\pr_1\circ \phi = p|_{p^{-1}(U_\phi)}$, where $\pr_1\colon U_\phi\times \sF_\phi\to U_\phi$ is the projection onto $U_\phi$;
		\item \label{cond:overlap} for any two local trivializations $\phi$ and $\psi$ with $U_\phi\cap U_\psi\neq \emptyset$,
		the \Emph{transition map} $\psi\circ\phi^{-1}\colon (U_\phi\cap U_\psi)\times \sF_\phi
		\to (U_\phi\cap U_\psi)\times \sF_\psi$ is an   isomorphism between trivial vector bundles (hence $\sF_{\psi}\cong \sF_\phi$).
	\end{enumerate}

For $S\subseteq M,$ we denote $E|_S:= p^{-1}(S)$.

Let $\wt p\colon \wt E\to \wt M$ be another vector bundle. A \Emph{vector bundle morphism} from $E\to \wt E$ is given by a pair of smooth maps $(\Phi,f)$ where $\Phi\colon E\to \wt E$ and $f\colon M\to \wt M$ satisfy: $\wt p\circ  \Phi = f\circ p$  and $\Phi_x\colon E_x\to \wt E_{f(x)}$ is linear. 
Composition of  vector bundle morphisms is a  vector bundle morphism.

Similarly to the finite-dimensional case, the local triviality condition ensures that $p$ is a submersion, the model spaces $\sF_\phi$ are (topologically) isomorphic on connected components, and the set of transition maps forms a \v{C}ech cocycle. One can alternatively define
a vector bundle as a set and use the manifold structure defined by the
system of local trivializations.

To relate the above definition to the standard one in the Banach setting, we use the following elementary lemma. See ~\cite[Lemma~1.5.9,Remark~1.5.10]
{GloecknerNeeb2026InfiniteDimensional} for the proof in the 
locally convex setting where the space of multilinear maps 
has the bounded-convergence topology.

We first recall a simple fact. Let $P$ be a topological space and
let
$B\colon P\times \sE\times \sF\to \sH$ be continuous, where $\sE$ is a locally
convex space and $\sF,\sH$ are Banach spaces, and suppose that
$B(p,\cdot,\cdot)$ is bilinear for every $p\in P$. 
Then, for every
$p_0\in P$ and every $\varepsilon>0$, there are a neighborhood $P_0$
of $p_0$ and a $0$-neighborhood $W\subseteq \sE$ such that
$\|B(p,h,\cdot)\|_{\mathrm{op}}<\varepsilon$ for all $p\in P_0$ and
$h\in W$. %
Indeed, by continuity at $(p_0,0,0)$, there are a
neighborhood $P_1$ of $p_0$, a $0$-neighborhood $W_1\subseteq \sE$,
and $\delta>0$ such that $\|B(p,h,v)\|_{\sH}<1$ whenever $p\in P_1$,
$h\in W_1$, and $\|v\|_{\sF}<\delta$. Choose $0<\lambda\leq 1$ so
that $2\lambda/\delta<\varepsilon$, and put $W=\lambda W_1$. If
$h=\lambda h_1\in W$ and $\|v\|_{\sF}\leq 1$, then, with
$w=(\delta/2)v$, bilinearity gives
$\|B(p,h,v)\|_{\sH}=(2\lambda/\delta)\|B(p,h_1,w)\|_{\sH}<\varepsilon$.
Hence $\|B(p,h,\cdot)\|_{\mathrm{op}}<\varepsilon$, as claimed.

\begin{lemma}\label{lem:smooth-linear-expo}
Let $\sE$ be a locally convex space, let $\sF,\sH$ be Banach spaces, and let
$U\subseteq\sE$ be open. Let $f\colon U\times\sF\to\sH$ be a smooth map in the
Bastiani sense, and suppose that $f(x,\cdot)$ is linear for every $x\in U$.
Then $f^\vee\colon U\to\mathcal L(\sF,\sH)$,
$f^\vee(x)(v):=f(x,v)$, is smooth, where $\mathcal L(\sF,\sH)$ is endowed with the
operator norm topology. Moreover,
$d(f^\vee)(x;h)(v)=d_1f(x,v;h)$, and more generally,
$$d^k(f^\vee)(x;h_1,\ldots,h_k)(v)
=d_1^kf(x,v;h_1,\ldots,h_k).$$
\end{lemma}

\begin{proof}
Since $f$ is  smooth and $f(x,\cdot)$ is linear for every
$x\in U$, the map
$d_1^k f\colon U\times\sF\times\sE^k\to\sH$ is continuous, linear in
$v\in\sF$, and $k$-linear in $(h_1,\ldots,h_k)\in\sE^k$, for every
$k\geq0$. Thus
$T_k(x;h_1,\ldots,h_k)(v):=d_1^k f(x,v;h_1,\ldots,h_k)$ defines a map
$T_k\colon U\times\sE^k\to\mathcal L(\sF,\sH)$. Moreover, whenever the
segment from $x$ to $x+u$ is contained in $U$, the fundamental theorem
of calculus gives, for every $v\in\sF$,
\begin{align*}
T_k(x+u;h_1,\ldots,h_k)(v)&-T_k(x;h_1,\ldots,h_k)(v)\\
&\quad=\int_0^1T_{k+1}(x+su;h_1,\ldots,h_k,u)(v)\,ds. \tag{*}   
\end{align*}

We first show that $T_k$ is continuous. Fix
$(x_0;h_1^0,\ldots,h_k^0)\in U\times\sE^k$. By multilinearity,
$T_k(x;h_1,\ldots,h_k)-T_k(x;h_1^0,\ldots,h_k^0)$ is a telescoping
sum whose $j$-th term contains the factor $h_j-h_j^0$. Applying the
preceding remark to each term, with $x$ and the remaining $h_i$ as
parameters, shows that this difference tends to $0$ in operator norm
as $(h_1,\ldots,h_k)\to(h_1^0,\ldots,h_k^0)$, uniformly for $x$ in a
neighborhood of $x_0$.

For the dependence on $x$, apply the preceding remark to
$d_1^{k+1}f(y,v;h_1^0,\ldots,h_k^0,u)$ with $h^0_i$ fixed. Taking $x$ in a
sufficiently small convex neighborhood of $x_0$, the segment from
$x_0$ to $x$ lies in the neighborhood supplied by the remark, and
$x-x_0$ belongs to the corresponding $0$-neighborhood in $\sE$.
Hence, by $(*)$,
\begin{align*}
&\|T_k(x;h_1^0,\ldots,h_k^0)-T_k(x_0;h_1^0,\ldots,h_k^0)\|_{\mathrm{op}}\\
&\qquad\qquad\leq
\sup_{0\leq s\leq1}
\|T_{k+1}(x_0+s(x-x_0);h_1^0,\ldots,h_k^0,x-x_0)\|_{\mathrm{op}},
\end{align*}
and the right-hand side tends to $0$ as $x\to x_0$. Together with the
preceding paragraph and the triangle inequality, this proves that
$T_k\colon U\times\sE^k\to\mathcal L(\sF,\sH)$ is continuous.

In particular, $f^\vee=T_0$ is continuous. Now take $u=th$ in $(*)$.
For sufficiently small $t\neq0$,
\[
\begin{aligned}
&\left\|
\frac{T_k(x+th;h_1,\ldots,h_k)-T_k(x;h_1,\ldots,h_k)}{t}
-T_{k+1}(x;h_1,\ldots,h_k,h)
\right\|_{\mathrm{op}}\\
&\qquad\leq
\sup_{0\leq s\leq1}
\|T_{k+1}(x+sth;h_1,\ldots,h_k,h)
-T_{k+1}(x;h_1,\ldots,h_k,h)\|_{\mathrm{op}},
\end{aligned}
\]
which tends to $0$ as $t\to0$ by the continuity of $T_{k+1}$. Thus
$d_1T_k(x;h_1,\ldots,h_k;h)=T_{k+1}(x;h_1,\ldots,h_k,h)$. Induction
gives $d^k(f^\vee)(x;h_1,\ldots,h_k)=T_k(x;h_1,\ldots,h_k)$, and hence
$d^k(f^\vee)(x;h_1,\ldots,h_k)(v)
=d_1^k f(x,v;h_1,\ldots,h_k)$. Thus $f^\vee$ is smooth.
\end{proof}

\begin{remark}\label{remark:lcvb-bvb}
\begin{enumerate}
\item In the Banach category, the vector bundle definition used here is equivalent to
the standard one. Indeed, suppose
 that
$\Psi\colon U\times \sF\to U\times \sF$,
$\Psi(x,u)=(x,\widehat\Psi(x,u))$,
is a smooth trivial vector bundle isomorphism
representing the transition map between two local trivializations over $U$ with typical fiber $\sF$. Then $\widehat\Psi(x,\cdot)$ is
a continuous linear automorphism of $\sF$ for every $x$. By
Lemma~\ref{lem:smooth-linear-expo}, the map
$A\colon U\to \mathcal L(\sF)$,
$A(x):=\wh\Psi(x,\cdot)$,
is smooth. Since each $A(x)$ is invertible, $A$ takes values in
$\GL(\sF)\subseteq \mathcal L(\sF)$. Thus
$\Psi(x,u)=(x,A(x)u)$
with $A\colon U\to \GL(\sF)$ smooth.
Conversely, if $A\colon U\to \GL(\sF)$ is smooth, then
$\Psi(x,u):=(x,A(x)u)$
is a smooth trivial vector bundle isomorphism, because the evaluation map
$\operatorname{ev}\colon \mathcal L(\sF)\times \sF\to \sF$
is smooth. Its inverse is given by
$\Psi^{-1}(x,u)=(x,A(x)^{-1}u)$,
and this is smooth since inversion in the Banach-Lie group $\GL(\sF)$ is smooth.

\item The same argument applies to multilinear maps in the Banach case. If $\sE$ is a locally convex space, $\sF$ and
$\sH$ are Banach spaces, $U\subseteq \sE$ is open, and
$B\colon U\times \sF^k\to \sH$
is smooth and $B(x,\cdot)$ is $k$-linear for every $x\in U$,
then
$b\colon U\to \mathcal L^k(\sF,\sH)$,
$b(x)(v_1,\ldots,v_k):=B(x,v_1,\ldots,v_k)$,
is smooth, where $\mathcal L^k(\sF,\sH)$ carries its usual operator norm topology.
If $B(x,\cdot)$ is alternating, then $b$ takes values in the closed subspace
$\mathcal L^k_{\mathrm{alt}}(\sF,\sH)$.
\end{enumerate}
\end{remark}

 A  subset $L\subseteq E$ is 
 a \Emph{(split) subbundle} if,
 for every $x\in M$, there exists
a local trivialization $\phi\colon p^{-1}(U_\phi)\to U_\phi\times \sF_\phi$
and a (complemented) closed subspace $\sH_\phi\subseteq \sF_\phi$ such that the restriction $\phi|_{p^{-1}(U_\phi)\cap L}\colon p^{-1}(U_\phi)\cap L\to U_\phi\times\sH_{\phi}$ is a  diffeomorphism, thus defining a local trivialization for $L$ about $x$. 
In particular, $L$ is a  (split) submanifold of $E$.

Let $f\colon N\to M$ be a smooth map. The \Emph{pullback bundle} 
 $ f^*E \to N$
 is constructed in the usual way; we set
 $f^*E = N\times_{f,p} E $.  This has a smooth manifold structure for which 
 $f^*p:=\pr_1|_{f^*E}\colon f^*E\to N$ is a submersion.
 Let $y\in N$ and let $\psi$ be a local trivialization about $f(y)$. Let $U_{\wt\psi}=f^{-1}(U_\psi)$.
 Note that $(f^*p)^{-1}(U_{\wt\psi}) = (f^*E)\cap (U_{\wt\psi}\times p^{-1}(U_\psi))$. Construct a diffeomorphism $\wt\psi\colon (f^*p)^{-1}(U_{\wt\psi})\to U_{\wt\psi}\times \sF_\psi$ by restricting 
 $\id_{U_{\wt\psi}}\times (\pr_2\circ \psi)$ to $(f^*E)\cap (U_{\wt\psi}\times p^{-1}(U_\psi))$. The inverse is given by $\wt\psi^{-1}(x, u)=
 (x, \psi^{-1}(f(x), u))$. For another local trivialization $\phi$ defining a diffeomorphism $\wt\phi$ in the same manner, the overlap map $\wt\phi\circ{\wt\psi}^{-1}$ is given by the smooth map
 $$
 (U_{\wt\psi}\cap U_{\wt\phi})\times \sF_\psi \ni (x,u)\mapsto (x, \pr_2((\phi\circ\psi^{-1})(f(x), u))) \in (U_{\wt\psi}\cap U_{\wt\phi})\times \sF_{\phi}.
 $$
 This shows that $f^*p\colon f^*E\to N$ is a smooth vector bundle.

\paragraph{Mapping spaces, spaces of sections, and exponential laws.}

We recall several facts concerning spaces of smooth maps,
spaces of sections of locally convex vector bundles, and
canonical mapping manifolds. These will be used particularly
in the construction of current Lie algebroids in
Theorem~\ref{lemma:current-first-order}.
We refer to
\cite[Chapters~1 and~4]{GloecknerNeeb2026InfiniteDimensional}.
Unless otherwise stated, all function spaces are endowed
with the smooth compact-open topology. Throughout the listed facts below, $K$ denotes a compact smooth manifold.

\begin{enumerate}

\item %
Let $\sF$ be a locally convex space. Then
$C^\infty(K,\sF)$, endowed with the smooth compact-open
topology, is a locally convex space. This topology is
determined by uniform convergence on compact subsets
of all derivatives in local charts of $K$.
If $\sF$ is a Fr\'echet space, then
$C^\infty(K,\sF)$ is a Fr\'echet space. In particular,
this holds whenever $\sF$ is Banach.

\item %
Let $P$ be a locally convex manifold, and let $\sF$
be a locally convex space. A map
$f\colon P\to C^\infty(K,\sF)$ is smooth if and only
if its adjoint
$f^\wedge\colon P\times K\to\sF$,
$(p,k)\mapsto f(p)(k)$, is smooth.
More generally, for $r\in\mathbb N_0\cup\{\infty\}$,
the map $f$ is of class $C^r$ if and only if
$f^\wedge$ is of class $C^{r,\infty}$, where the
latter denotes $r$-fold differentiability in the
$P$-variable and smoothness in the $K$-variable,
including the corresponding mixed derivatives.
In particular, the evaluation map
$\ev\colon C^\infty(K,\sF)\times K\to\sF$,
$(f,k)\mapsto f(k)$, is smooth.

\item %
Let $\sE$ and $\sF$ be locally convex spaces,
let $V\subseteq\sE$ be open, and let
$h\colon K\times V\to\sF$ be smooth.
Then the induced superposition map
$h_*\colon C^\infty(K,V)\to C^\infty(K,\sF)$,
defined by $h_*(\gamma)(k):=h(k,\gamma(k))$,
is smooth.
The set $C^\infty(K,V)$ is open in
$C^\infty(K,\sE)$. In particular, if
$h(k,v)=h_0(v)$ is independent of $k$, then
the induced map
$C^\infty(K,h_0)\colon C^\infty(K,V)
\to C^\infty(K,\sF)$,
$\gamma\mapsto h_0\circ\gamma$, is smooth.

\item %
Let $q\colon E\to K$ be a smooth locally convex
vector bundle with typical fiber $\sF$, and let
$\Gamma(E)$ denote its vector space of smooth sections.
We endow $\Gamma(E)$ with the subspace topology
induced by the smooth compact-open topology on
$C^\infty(K,E)$. This makes $\Gamma(E)$ into a
locally convex space.
More explicitly, let
$\theta_i\colon E|_{U_i}\to U_i\times\sF$,
$i\in I$, be smooth local trivializations whose
domains cover $K$. For $\sigma\in\Gamma(E)$, put
$\sigma_i:=\pr_2\circ\theta_i\circ\sigma|_{U_i}
\in C^\infty(U_i,\sF)$.
Then the linear map
$\Gamma(E)\to\prod_{i\in I}C^\infty(U_i,\sF)$,
$\sigma\mapsto(\sigma_i)_{i\in I}$,
is a topological embedding with closed image.
Its image consists precisely of the families of
local representatives satisfying the compatibility
conditions imposed by the bundle transition maps.
Consequently, the subspace topology on $\Gamma(E)$
coincides with the initial locally convex topology
defined by its local representatives, and this
topology is independent of the chosen trivializing
cover.

\item %
Let $q\colon E\to K$ be as above.
If the typical fiber $\sF$ is complete, then
$\Gamma(E)$ is complete. Similarly, sequential
completeness, quasi-completeness, and Mackey
completeness of $\sF$ pass to $\Gamma(E)$.
If $\sF$ is metrizable, then $\Gamma(E)$ is
metrizable, since $K$ is compact and hence
$\sigma$-compact.
In particular, if $\sF$ is a Banach space, then
$\Gamma(E)$ is a Fr\'echet space.

\item %
Let $p\colon A\to M$ be a smooth locally convex
vector bundle, and let $\gamma\colon K\to M$ be smooth.
Define $\Gamma_\gamma(A):=
\{\sigma\in C^\infty(K,A):p\circ\sigma=\gamma\}$,
equipped with the subspace topology inherited
from $C^\infty(K,A)$.
The map
$\Gamma_\gamma(A)\to\Gamma(\gamma^*A)$,
$\sigma\mapsto(k\mapsto(k,\sigma(k)))$,
is an isomorphism of locally convex topological
vector spaces.
Thus the space of smooth sections of $A$ along
$\gamma$ can be identified with the section space
of the pullback bundle $\gamma^*A\to K$,
including its locally convex topology.

\item %
Let $q\colon E\to K$ be a smooth locally convex
vector bundle and let $P$ be a locally convex
manifold.
A map $f\colon P\to\Gamma(E)$ is smooth if and
only if its adjoint $f^\wedge\colon P\times K\to E$,
$(p,k)\mapsto f(p)(k)$, is smooth and satisfies
$q\circ f^\wedge=\pr_K$.
More generally, $f$ is of class $C^r$ if and
only if $f^\wedge$ is of class $C^{r,\infty}$
and satisfies the same bundle condition.
In particular, the evaluation map
$\ev\colon\Gamma(E)\times K\to E$,
$(\sigma,k)\mapsto\sigma(k)$, is smooth.
The same conclusions hold for sections along
a fixed smooth map $\gamma\colon K\to M$,
using the identification
$\Gamma_\gamma(A)\cong\Gamma(\gamma^*A)$.

\item %
Let $E_1,E_2,E_3\to K$ be smooth locally convex
vector bundles.
If $\Theta\colon E_1\to E_2$ is a smooth vector
bundle morphism over $\id_K$, then the induced map
$\Gamma(\Theta)\colon\Gamma(E_1)\to\Gamma(E_2)$,
$\sigma\mapsto\Theta\circ\sigma$,
is continuous linear.
More generally, if
$B\colon E_1\times_K E_2\to E_3$
is a smooth fiberwise bilinear map, then
the induced map
$\Gamma(E_1)\times\Gamma(E_2)\to\Gamma(E_3)$,
$(\sigma,\tau)\mapsto
(k\mapsto B(\sigma(k),\tau(k)))$,
is smooth and bilinear.
The latter assertion follows by applying
the exponential law for sections to the jointly
smooth adjoint map.
We also use the corresponding continuity
of multiplication of sections by smooth
scalar-valued functions.

\item %
Let $P$ be a locally convex manifold and
$q\colon E\to K$ a smooth locally convex
vector bundle.
To prove that a map $f\colon P\to\Gamma(E)$
is smooth, it suffices to show that its adjoint
$f^\wedge\colon P\times K\to E$ is smooth and
takes values in the appropriate fibers.
Equivalently, one may check smoothness of all
local representatives in bundle trivializations.
Moreover, differentiation commutes with continuous
linear point evaluation. For example, if
$V\subseteq\sE$ is open and
$f\colon V\to C^\infty(K,\sF)$ is smooth, then
$(df(x;h))(k)=d_1f^\wedge(x,k;h)$
for all $x\in V$, $h\in\sE$, and $k\in K$.
The analogous identity holds for section-valued
maps in local trivializations, as well as for
higher derivatives.
These are consequences of the exponential laws
above and the Bastiani chain rule.

\item %
Let $M$ be a locally convex manifold admitting
a smooth local addition.
Then $C^\infty(K,M)$ admits a canonical smooth
manifold structure, modeled near
$\gamma\in C^\infty(K,M)$ on open subsets of
the locally convex space $\Gamma(\gamma^*TM)$.
Its underlying topology is the smooth
compact-open topology.
The canonical structure is characterized by
the following exponential law: for every
locally convex manifold $P$, a map
$f\colon P\to C^\infty(K,M)$ is smooth if
and only if its adjoint
$f^\wedge\colon P\times K\to M$,
$(p,k)\mapsto f(p)(k)$, is smooth.
The corresponding assertion also holds for
$C^r$-maps and $C^{r,\infty}$-adjoints.
In particular, the evaluation map
$\ev\colon C^\infty(K,M)\times K\to M$,
$(\gamma,k)\mapsto\gamma(k)$, is smooth.

\item %
In the setting of the preceding item, for every
$\gamma\in C^\infty(K,M)$ there is a canonical
isomorphism of locally convex
spaces
$T_\gamma C^\infty(K,M)\cong\Gamma(\gamma^*TM)$,
given explicitly by
$v\mapsto(k\mapsto T_\gamma\ev_k(v))$,
where $\ev_k\colon C^\infty(K,M)\to M$
denotes evaluation at $k$.
Under the natural identification of
$\Gamma(\gamma^*TM)$ with smooth maps
$K\to TM$ covering $\gamma$, these
fiberwise identifications assemble into
the natural smooth vector bundle identification
$TC^\infty(K,M)\cong C^\infty(K,TM)$,
where the latter is understood with its natural
mapping-bundle structure over $C^\infty(K,M)$.
The smooth bundle assertion follows by checking
the preceding identification in the canonical
charts induced by a local addition. We shall
use this identification when defining the anchor
of a current Lie algebroid.

\item %
Let $M$ and $N$ be locally convex manifolds
admitting smooth local additions, and let
$h\colon M\to N$ be smooth. Then the induced map
$h_*:=C^\infty(K,h)\colon C^\infty(K,M)\to C^\infty(K,N)$,
$\gamma\mapsto h\circ\gamma$, is smooth.
More generally, smooth parameter-dependent
superposition maps between canonical mapping
manifolds are smooth whenever the hypotheses
of the corresponding exponential law are
satisfied.
In addition, if $M$ and $N$ admit local
additions, there is a canonical diffeomorphism
$C^\infty(K,M\times N)\cong
C^\infty(K,M)\times C^\infty(K,N)$.

\end{enumerate}

In what follows, we use these facts without
further mention. In particular, for a smooth
Banach vector bundle $A\to M$ and
$\gamma\in C^\infty(K,M)$, the space
$\Gamma(\gamma^*A)$ is Fr\'echet, smooth
families of sections are characterized by
their adjoint maps, and smooth fiberwise
operations induce smooth maps between the
corresponding section spaces. 

A groupoid $\G\tto M$ with structure maps $(\src,\tgt,\mult,\one,\inv)$ is 
a \Emph{Lie groupoid} 
if $M$ is a locally convex manifold which is Hausdorff, $\G$
is a not necessarily Hausdorff manifold,
$\src$ is a  submersion,  $\inv$ is a diffeomorphism, hence $\tgt$ is also a submersion, $\mult$ is a smooth map, where $\mset{\G}=(\src\times\tgt)^{-1}(\Delta_M)$ is a split submanifold of $\G\times\G$, and $\one$ is an embedding. 
For all $g\in \G$,
right-multiplication $R_g(h) = hg$ is a smooth diffeomorphism on source fibers $R_g\colon \G_{\tgt(g)}\to \G_{\src(g)}$.
We set  $\G_x : =\src^{-1}(x)$ for the source fiber at $x$.
We say that $\G$ is source connected (resp. simply connected)  
if each $\G_x$ is connected (resp. simply connected).
We note that $T\src$ is a submersion and that $\ker(T\src)$ is a split subbundle of
$T\G$. Indeed, let $\phi\colon U_\phi\to V_\phi\subseteq \sM_\psi\times \sC$ and
$\psi\colon U_\psi\to V_\psi\subseteq \sM_\psi$ be submersion charts for $\src$, so that
$
\psi\circ s\circ \phi^{-1}=p|_{V_\phi},$
for a continuous projection
$
p\colon M_\psi\times C\to M_\psi.$
Under the usual tangent identifications,
$
T(\psi\circ \src\circ \phi^{-1})=Tp=p\times p.
$
Thus $T\src$ is locally a projection onto a complemented space. Moreover,
$ T\phi\bigl(\ker(T\src)\cap TU_\phi\bigr) =
V_\phi\times(\{0\}\times \sC),
$
which is a split subbundle of
$TV_\phi\cong V_\phi\times(\sM_\psi\times \sC)$. 
The vector bundle $\Lie(\G):=\one^{*}\ker(T\src)\to M$ is anchored with
the anchor defined by restricting
$T\tgt$. A Lie bracket is defined as the Lie bracket of right-invariant extensions over the unit map. That is, for $\xi,\eta\in\Gamma(\Lie(\G)|_U)$, we set $[\xi,\eta] = [\ora\xi,\ora\eta]\circ\one$, where $\ora\xi (g)=T_{1_{\tgt(g)}}R_g(\xi(\tgt(g)))$, for $g\in\tgt^{-1}(U)$, and similarly $\ora\eta$ is defined. 

\section{First order locally convex Lie algebroids}\label{sec:firstorderLiealgbd}
\begin{definition}\label{def:liealgebroid}
A \Emph{first-order 
Lie algebroid}
is a vector bundle $p\colon A\to M$
together with a vector bundle morphism  $\anch\colon A\to TM$ over $\id_M$,
called the \Emph{anchor},
and a morphism of sheaves of $\R$-modules 
$\LB_{\shA}\colon \shA\otimes_\R \shA\to \shA$ 
on the sheaf of sections $\shA$ of $A$ such that:
\begin{enumerate}\itemsep0em

\item \label{item:LA1} for every open set $U\subseteq M$,
the bracket $\LB_{\shA(U)}$ turns $\shA(U)$ into a Lie algebra;
\item \label{item:LA2} for every open set $U\subseteq M$,
every $\xi,\eta\in \shA(U)$, and every $f\in \shcinf_M(U)$,
the Leibniz rule $$[\xi,f\eta]_{\shA(U)}=f[\xi,\eta]_{\shA(U)}+(\anch(\xi).f)\eta$$ holds;

\item \label{item:LA3} for every local trivialization $\phi\colon A|_U\to U\times \sF_\phi$,
there exists a smooth map $C_\phi\colon U\times \sF_\phi\times \sF_\phi\to \sF_\phi$
such that, for each $x\in U$, the map $(u,v)\mapsto C_\phi(x,u,v)$ is alternating bilinear
and, for all $\xi,\eta\in \shA(U)$, one has
\begin{equation}\label{eq:first-order}
[\xi,\eta]^\phi(x)=\anch(\xi).\eta^\phi(x)-\anch(\eta).\xi^\phi(x)+C_\phi(x,\xi^\phi(x),\eta^\phi(x)),     
\end{equation}
where $\xi^\phi,\eta^\phi,[\xi,\eta]^\phi\in C^\infty(U,\sF_{\phi})$ are the local representatives in $\phi.$
\end{enumerate}
\end{definition}

The smooth maps $C_\phi$ are uniquely determined by the value of the Lie bracket at 
constant local sections; see Lemma~\ref{lemma:characterization}.
An anchored vector bundle satisfying Conditions~\ref{item:LA1} and
\ref{item:LA2} is called a \Emph{locally convex Lie algebroid}.

For a first-order Banach Lie algebroid, the local structure map
$C_\phi\colon U_\phi\times \sF_\phi\times \sF_\phi\to \sF_\phi$
may equivalently be regarded as a smooth map
$c_\phi\colon U_\phi\to \mathcal L^2_{\mathrm{alt}}(\sF_\phi)$, see Remark~\ref{remark:lcvb-bvb}.
In this case, the regularity condition is related 
to the 
one
appearing in the theory of
almost Lie algebroids on anchored Banach bundles~\cite[Sections~3.1 and~3.3]
{CabauPelletier2012AlmostLie}, where localization of
brackets and local structure expressions are considered.
Here we work with general locally convex bundles in
Bastiani calculus and explicitly require the local
bilinear term $C_\phi$ to be smooth, thereby avoiding the need to topologize the space of continuous bilinear maps on the fiber. This hypothesis
will be used to construct the sheaves of Lie algebroid
forms and their differentials without assuming the
existence of smooth exterior dual bundles.

\begin{remark}[Coordinate change for $C_\phi$]\label{remark:coordinate-change-C}
Let $\phi \colon A|_U \to U \times \sF_\phi$ and $\wt\phi \colon A|_{\wt U} \to \wt U \times \sF_{\wt\phi}$ be two local trivializations. Set
$
(\wt\phi \circ \phi^{-1})(x,u) = (x,\wh K(x,u))
$
for $(x,u) \in (U \cap \wt U) \times \sF_\phi$, where $\wh K$ is smooth and linear in the second variable. Let $\anch^\phi$ be the local representative of the anchor in the trivialization $\phi$. Then, on $U \cap \wt U$ with a slight abuse of notation, the local structure maps $C_\phi$ and $C_{\wt\phi}$ satisfy
\[\begin{split}
C_{\wt\phi}(x,\wh K(x,u),\wh K(x,v))
&=
\wh K(x,C_{\phi}(x,u,v))
- d_1\wh K(x,v;\anch^\phi(x,u))
+ d_1\wh K(x,u;\anch^\phi(x,v)).
\end{split}
\]
Indeed, if $\xi^{\wt\phi}(x)=\wh K(x,\xi^\phi(x))$ and $\eta^{\wt\phi}(x)=\wh K(x,\eta^\phi(x))$, then one writes the bracket in the two trivializations and compares the two expressions.
In particular, Condition~\ref{item:LA3} of Definition~\ref{def:liealgebroid}
may be verified on any local trivializing cover.
\end{remark}

We shall write $\LB$ and suppress the subscript
when the Lie algebroid is clear from the context.
The tangent bundle $TM\to M$,
with $\anch=\id_{TM}$ and the usual Lie bracket of vector fields,
is a Lie algebroid.
Obviously, $\anch$ extends to a morphism of sheaves of $\shcinf_M$-modules $\anch\colon \shA\to \shTM$.

\begin{lemma}\label{lemma:anchsheaf}
The anchor $\anch\colon A\to TM$
induces a morphism of sheaves of Lie algebras
$\anch\colon (\shA,\LB_{\shA})\to (\shTM,\LB_{\shTM})$.
\end{lemma}

\begin{proof}
Let $U\subseteq M$ be open and let $\xi,\eta\in \shA(U)$.
Set $X:=\anch([\xi,\eta])-[\anch(\xi),\anch(\eta)]\in \shTM(U)$.
We show that $X=0$.
Let $f\in \shcinf_M(U)$ and $\zeta\in \shA(U)$.
Using the Jacobi identity in $\shA(U)$ and the Leibniz rule, we compute
\[ 0=[[ \xi,\eta],f\zeta]+[[\eta,f\zeta],\xi]+[[f\zeta,\xi],\eta]=(X.f)\zeta. \]
Fix $x\in U$. If $A_x= \{0\}$, then by~\eqref{eq:first-order}, $\xi(x)=\eta(x)=[\xi,\eta](x)=0$. Hence $X(x)=0$. Suppose otherwise and choose a local trivialization of $A$ about $x$ so that,
possibly after shrinking, $U$ is the trivializing neighborhood.
Choose a constant local section $\zeta$ with $\zeta(x)\neq 0$.
Evaluating at $x$ gives $(X.f)(x)\zeta(x)=0$, hence $(X.f)(x)=0$.
Since this holds for every $f$, we get $X(x)=0$. Indeed,
possibly after shrinking $U$ again, if $\kappa$ is a chart 
about
$x$ with domain $U$ and $X(x)\neq 0$ then $T_x\kappa(X(x))\neq 0$. 
By the Hahn-Banach theorem, there exists $\lambda\in \sM_\kappa'$ such that $X.(\lambda\circ \kappa)(x) \neq 0$, a contradiction.
As $x$ was arbitrary, $X=0$ on $U$.
\end{proof}

The following lemma 
is elementary. It provides a characterization of~\ref{item:LA3} in Definition~\ref{def:liealgebroid}.

\begin{lemma}\label{lemma:characterization}
Let $p\colon A\to M$ be  
a locally convex vector bundle with an anchor $\anch$ and a bracket $\LB_{\shA}$ satisfying Conditions~\ref{item:LA1} and \ref{item:LA2}. 
Then the following are equivalent:
\begin{enumerate}[label=(\roman*)]
    \itemsep0em
    \item \label{item:Characterization1} $A\to M$ is a first-order locally convex Lie algebroid;
    \item \label{item:Characterization2} In each local trivialization $\phi\colon A|_U\to U\times \sF_\phi$, set $C_\phi(x,u,v) := [\xi_u,\xi_v]^\phi(x)$, where $\xi_u,\xi_v\in\shA(U)$ are the constant sections associated with $u,v\in\sF_\phi$.
    Then
    \begin{enumerate}
        \itemsep0em 
        \item $C_\phi\colon U\times \sF_\phi\times \sF_\phi \to \sF_\phi$  is smooth;
        \item for every $x\in U$, $u\in\sF_\phi$,
        $\zeta\in \shA(U)$ for which $\zeta^\phi(x)=0$,
        we have
        $[\xi_u,\zeta]^\phi(x) = \anch(\xi_u).\zeta^\phi(x)$;
        \item if $\zeta_1,\zeta_2\in \shA(U)$ satisfy $
        \zeta_1^\phi(x)= \zeta_2^\phi(x) = 0$, then 
        $[\zeta_1,\zeta_2]^\phi(x)=0$.
    \end{enumerate}
\end{enumerate}
\end{lemma}
\begin{proof}
    By definition \ref{item:Characterization1} implies \ref{item:Characterization2}. Now consider a local trivialization $\phi$.
    Let $\xi, \eta\in \shA(U)$, and let $x\in U$. 
    Fix $u=\xi^\phi(x)$ and $v=\eta^\phi(x)$. Write
    $\xi = \wt \xi +\xi_u$ and $\eta = \wt \eta + \eta_v$, where $\xi_u(y) := \phi^{-1}(y, u)$ and $\eta_v(y) := \phi^{-1}(y, v)$.
    Expanding 
    $[\xi,\eta]^\phi(x) = [\wt \xi+\xi_u,\wt \eta +  \eta_v]^\phi(x)$, we see that
    $$
    [\xi,\eta](x) = [\xi_u,\eta_v](x)+[\xi_u,\wt\eta](x)+
    [\wt\xi,\eta_v](x) + [\wt\xi,\wt\eta](x).
    $$
    Hence antisymmetry and the conditions in \ref{item:Characterization2}
    give \eqref{eq:first-order}.
\end{proof}

\begin{corollary}\label{cor:finite-dim-fiber}
Let $M$ be a locally convex manifold and 
    let $p\colon A\to M$ be  
a  vector bundle with a finite-dimensional typical fiber and with an anchor $\anch$ and a bracket $\LB_{\shA}$ satisfying Conditions~\ref{item:LA1} and \ref{item:LA2} in Definition~\ref{def:liealgebroid}.
Then $A\to M$ is a first-order Lie algebroid.
\end{corollary}
\begin{proof}
Let $\phi\colon A|_U\to U\times \R^n$ be a local trivialization.
    Let $\{e_1,\ldots,e_n\}$ be a basis for the fiber, and define $\underline e_i = \phi^{-1}(\cdot, e_i)$. 
    Now $[\underline e_i,\underline e_j]= \sum_k c^k_{ij} \underline e_k$, where $c^k_{ij}\in C^\infty(U)$.
    Set $C_\phi(x,u,v) := \sum_{i,j,k}u^iv^jc^k_{ij}(x)e_k$. Then $C_\phi$ is smooth and alternating bilinear in $u$ and $v$. Write $\zeta \in \shA(U)$ as $\zeta = \sum \alpha^i\underline e_i$. If 
   $\zeta^\phi(x) =0$, then $\alpha^j(x) = 0$. By the Leibniz rule, 
   $[\underline e_i,\zeta]^\phi(x)=\anch(\underline e_i).\zeta^\phi(x)$.
   The final statement in \ref{item:Characterization2} follows similarly. Thus,
   by Lemma~\ref{lemma:characterization}, $A\to M$ is first-order.
\end{proof}    

\begin{corollary}
    Every finite-dimensional Lie algebroid is first-order.
\end{corollary}

More generally, suppose that $M$ admits smooth cutoff functions equal to $1$ near a given point and supported inside any prescribed neighborhood (smoothly regular). Then a Lie bracket on global sections of a finite-rank anchored bundle satisfying the Leibniz rule defines a bracket on local sections: multiply a local section by such a cutoff to extend it globally, and use the Leibniz rule to check that the bracket germ is independent of the extensions. The resulting sheaf bracket satisfies Conditions~\ref{item:LA1} and \ref{item:LA2}, so Corollary~\ref{cor:finite-dim-fiber} applies. This is the localization hypothesis relevant to smooth regularity~\cite[\S 14]{kriegl1997convenient}; compare~\cite[Remark~4.6]{BeltitaGolinskiJakimowiczPelletier2019}.

\begin{example}[Locally convex Lie algebras]
Let $M=\{*\}$ and let $A=\mathfrak g\to \{*\}$, where $\mathfrak g$ is a locally convex Lie algebra (with continuous bracket), for example the Lie algebra of a locally convex Lie group~\cite{GloecknerNeeb2026InfiniteDimensional}.
Set $\anch=0$, and define the bracket on sections by the bracket of $\mathfrak g$.
Then $A$ is a first-order Lie algebroid.
Indeed, in the standard local trivialization, the first-order formula reduces to $[u,v]=C(u,v)$, where $C(u,v):=[u,v]_{\mathfrak g}$.
\end{example}

\begin{remark}[An algebraic bracket without first-order regularity]
The Lie and Leibniz identities alone do not imply the continuity required by Condition~\ref{item:LA3}. In particular, an algebraic Lie bracket on a locally convex space need not be continuous~\cite{GloecknerNeeb2026InfiniteDimensional}.
A simple standard example can be constructed as follows.
Let $\sE$ be an infinite dimensional Banach space. Let $\lambda$ be a non-continuous linear functional. Define $[u,v]=\lambda(u)v-\lambda(v)u$
for $u,v\in \sE$. This defines a Lie bracket. Indeed, it is clearly bilinear, and the Jacobi identity follows by noting that $\lambda([u,v])=0$.
Take $e\in \sE$ such that $\lambda(e)=1$. Note that
$f\colon u\mapsto [u,e]+u=\lambda(u)e$. Now take $\mu\in \sE'$ such that $\mu(e)=1$. Then $\lambda = \mu\circ f$.
Thus, if the bracket were continuous in $u$, $\lambda$ would be continuous which is a contradiction. 
This gives a bracket on the sheaf of sections over the one-point base satisfying Conditions~\ref{item:LA1} and \ref{item:LA2}, but not Condition~\ref{item:LA3}. It is not a locally convex Lie algebra in the usual topological sense.
\end{remark}

\begin{example}[Action Lie algebroids]\label{ex:action-lie-algbd}
Let $\mathfrak g$ be a locally convex Lie algebra (hence with smooth bracket), and let $\mathfrak g$ act smoothly on a locally convex manifold $M$.
Thus we have a smooth map $\rho\colon M\times \mathfrak g\to TM$ over $M$, linear in the second variable, such that $\fg\ni u\mapsto \rho(\cdot,u)\in \Gamma(TM)$ is a Lie algebra morphism.
Set $A:=M\times \mathfrak g$ and define the anchor by $\anch(x,u):=\rho(x,u)$.
For local sections $\xi,\eta\in C^\infty(U,\mathfrak g)$, define
\[
[\xi,\eta](x):=d\eta(x;\anch(x,\xi(x)))-d\xi(x;\anch(x,\eta(x)))+[\xi(x),\eta(x)]_{\mathfrak g}.
\]
Then $A$ is a first-order Lie algebroid.
The Leibniz rule is immediate from the formula.
The Jacobi identity is the standard one for the action Lie algebroid bracket.
Indeed, if $X_\xi:=\rho(\cdot,\xi(\cdot))$, then the assumption that
$u\mapsto \rho(\cdot,u)$ is a Lie algebra morphism implies, by the usual
chain-rule computation for $\fg$-valued functions, that $X_{[\xi,\eta]}=[X_\xi,X_\eta].$
Obviously, the first-order term is simply $C(x,u,v)=[u,v]_{\mathfrak g}$.
\end{example}

Now we consider first-order current algebroids.
Current groupoids and algebroids were studied in~\cite{AmiriGloecknerSchmeding2020CurrentGroupoids}.
In Theorem~\ref{lemma:current-first-order}, we show that the current algebroid associated with a first-order Banach-Lie algebroid is first-order when the base admits local addition.

\begin{theorem}[First-order current algebroids]\label{lemma:current-first-order}
Let $K$ be a compact smooth manifold and let $p\colon A\to M$ be a first-order Banach Lie algebroid with typical fiber $\sF$. Assume that $M$ admits a local addition. Then $C^\infty(K,A)$ admits a canonical smooth manifold structure for which $p_*\colon C^\infty(K,A)\to C^\infty(K,M)$ is a first-order Fr\'echet-Lie algebroid.
\end{theorem}
\begin{proof}
Set $\underline{M}:=C^\infty(K,M)$ and $\underline{A}:=C^\infty(K,A)$. Endow $\underline{M}$ with its canonical smooth manifold structure. For $\gamma\in\underline{M}$, we identify $p_*^{-1}(\gamma)$ with the Fr\'echet space $\Gamma(\gamma^*A)$.

We construct local trivializations of $p_*$. Let $\gamma\in\underline{M}$. Since $K$ is compact, choose finitely many local trivializations $\tau_i\colon A|_{U_i}\to U_i\times\sF$, open sets $R_i\subseteq K$ with $\gamma(\overline{R}_i)\subseteq U_i$, and a smooth partition of unity $\{\chi_i\}$ subordinate to $\{R_i\}$. Set $\Omega_i:=((K\setminus\supp(\chi_i))\times M)\cup(\gamma^{-1}(U_i)\times U_i)$. Then $\Omega_i$ is an open neighborhood of $\graph(\gamma)$.

On $\Omega_i$, define a smooth vector bundle map $\Theta_i\colon\pr_M^*A|_{\Omega_i}\to(\gamma\circ\pr_K)^*A|_{\Omega_i}$ as follows. On $\gamma^{-1}(U_i)\times U_i$, set $T_i(k,x):=\tau_{i,\gamma(k)}^{-1}\circ\tau_{i,x}$ and define $\Theta_i(k,x,e_x):=(k,x,\chi_i(k)T_i(k,x)e_x)$. On $(K\setminus\supp(\chi_i))\times M$, set $\Theta_i(k,x,e_x):=(k,x,0)$. The two definitions agree on the overlap because $\chi_i$ vanishes there. Hence $\Theta_i$ is smooth. On $\Omega:=\bigcap_i\Omega_i$, set $\Theta:=\sum_i\Theta_i$. Then $\Theta$ is a smooth vector bundle map and $\Theta_{k,\gamma(k)}=\id_{A_{\gamma(k)}}$ for all $k\in K$.

For each $k_0\in K$, choose a local trivialization $\tau\colon A|_U\to U\times\sF$ about $\gamma(k_0)$ and an open neighborhood $O\subseteq K$ of $k_0$ such that $O\times U\subseteq\Omega$. In this trivialization, $(k,x)\mapsto \tau_{\gamma(k)}\circ\Theta_{k,x}\circ\tau_x^{-1}$ is smooth as a map to $\mathcal{L}(\sF)$ by Lemma~\ref{lem:smooth-linear-expo}. It is the identity at $(k,\gamma(k))$. Since $\GL(\sF)$ is open in $\mathcal{L}(\sF)$, after shrinking $O$ and $U$, the map $\Theta_{k,x}$ is a linear isomorphism for all $(k,x)\in O\times U$. Thus there is an open neighborhood $W_\gamma\subseteq K\times M$ of $\graph(\gamma)$ such that $\Theta$ restricts to a smooth vector bundle isomorphism over $W_\gamma$. Write $\Theta^\gamma$ for this restriction.

Set $U_\gamma:=\{\eta\in \underline{M}:\graph(\eta)\subseteq W_\gamma\}$. We claim that $U_\gamma$ is open. Let $\eta\in U_\gamma$. Since $\graph(\eta)$ is compact and contained in $W_\gamma$, there are compact sets $C_j\subseteq K$, open sets $O_j\subseteq K$, and open sets $V_j\subseteq M$ such that $C_j\subseteq O_j$, the sets $C_j$ cover $K$, $\eta(C_j)\subseteq V_j$, and $O_j\times V_j\subseteq W_\gamma$ for all $j$. Then
\[
D:=\bigcap_j\{\wt\eta\in C^\infty(K,M):\wt\eta(C_j)\subseteq V_j\}
\]
is an open neighborhood of $\eta$ in the smooth compact-open topology, and $D\subseteq U_\gamma$. Hence $U_\gamma$ is open.

Define $\Psi_\gamma\colon p_*^{-1}(U_\gamma)\to U_\gamma\times\Gamma(\gamma^*A)$ by $\Psi_\gamma(\sigma)=(p\circ\sigma,\wh\sigma)$, where $\wh\sigma(k)=\Theta^\gamma_{k,p(\sigma(k))}(\sigma(k))$. Its inverse is $\Psi_\gamma^{-1}(\eta,u)(k)=(\Theta^\gamma_{k,\eta(k)})^{-1}u(k)$.

Let $\Psi_\delta$ be another such map. On $U_\gamma\cap U_\delta$, the transition map is $(\eta,u)\mapsto(\eta,S_{\delta\gamma}(\eta,u))$, where
\[
S_{\delta\gamma}(\eta,u)(k)=\Theta^\delta_{k,\eta(k)}(\Theta^\gamma_{k,\eta(k)})^{-1}u(k).
\]
By the exponential law, $S_{\delta\gamma}$ is smooth and linear in the second variable. Thus the maps $\Psi_\gamma$ define a smooth manifold structure on $\underline{A}$ for which $p_*$ is a smooth vector bundle.

We verify that the manifold topology agrees with the smooth
compact-open topology. 
Since $p_*\colon \underline A\to\underline M$ is continuous for the
smooth compact-open topologies and $U_\gamma$ is open,
$p_*^{-1}(U_\gamma)$ is open in $\underline A$ for all $\gamma$.
Consider the smooth maps
$(k,a)\mapsto
\bigl(k,\Theta^\gamma_{k,p(a)}a\bigr),$
and
$(k,x,u)\mapsto
(\Theta^\gamma_{k,x})^{-1}u$ which are 
defined whenever $(k,p(a))\in W_\gamma$ and
$(k,x)\in W_\gamma$, respectively, with
$u\in A_{\gamma(k)}$.
Their associated superposition maps are continuous
for the smooth compact-open topologies. 
 Hence $\Psi_\gamma$
is a homeomorphism for the smooth compact-open
topology. Thus, the
resulting manifold topology coincides with the
smooth compact-open topology.

Let $r\in\mathbb{N}_0\cup\{\infty\}$, let $P$ be a %
$C^r$-manifold, and let $h\colon P\to\underline{A}$ be a map. Write $h^\wedge(z,k):=h(z)(k)$. Suppose first that $h$ is $C^r$. Then $h_M:=p_*\circ h\colon P\to\underline{M}$ is $C^r$. On $h_M^{-1}(U_\gamma)$, write $\Psi_\gamma\circ h=(h_M,h_\gamma)$. %
By the exponential law, the map $(z,k)\mapsto h_\gamma(z)(k)$ is $C^{r,\infty}$. Since
$h^\wedge(z,k)=(\Theta^\gamma_{k,h_M(z)(k)})^{-1}h_\gamma(z)(k)$, and $h_M^\wedge$ is $C^{r,\infty}$ by canonicity of $\underline{M}$, it follows that $h^\wedge$ is $C^{r,\infty}$.

Conversely, suppose that $h^\wedge$ is $C^{r,\infty}$. Then $h_M^\wedge=p\circ h^\wedge$ is $C^{r,\infty}$, so $h_M$ is $C^r$ by canonicity of $\underline{M}$. For $z_0\in P$, put $\gamma:=h_M(z_0)$ and $P_0:=h_M^{-1}(U_\gamma)$. Then $P_0$ is an open neighborhood of $z_0$, and the second component $h_\gamma$ of $\Psi_\gamma\circ h|_{P_0}$ satisfies
$h_\gamma(z)(k)=\Theta^\gamma_{k,h_M(z)(k)}h^\wedge(z,k)$.
The right-hand side is $C^{r,\infty}$. 
By the exponential law,
$h_\gamma\colon P_0\to\Gamma(\gamma^*A)$ is $C^r$. Hence $h$ is $C^r$. This proves canonicity of the manifold structure on $\underline{A}$.

Identify $T\underline{M}$ with $C^\infty(K,TM)$ and denote $\underline{\anch}:=C^\infty(K,\anch)$.  
This defines a smooth vector bundle morphism $\underline{A}\to T\underline{M}$ over $\id_{\underline{M}}$.

For fixed $k\in K$, set $W_{\gamma,k}:=\{x\in M:(k,x)\in W_\gamma\}$. Define $\Psi_\gamma^k\colon A|_{W_{\gamma,k}}\to W_{\gamma,k}\times A_{\gamma(k)}$ by $\Psi_\gamma^k(a)=(p(a),\Theta^\gamma_{k,p(a)}a)$. This is a local trivialization of $A$ about $\gamma(k)$. Since $A$ is first-order, it has a structure map $C_\gamma^k$ in this trivialization. Thus $C_\gamma^k(x,u,v)$ is the bracket coefficient for the two $\Psi_\gamma^k$-constant sections with values $u$ and $v$.

Define $C_\gamma\colon U_\gamma\times\Gamma(\gamma^*A)\times\Gamma(\gamma^*A)\to\Gamma(\gamma^*A)$ by
\[
C_\gamma(\eta,u,v)(k):=C_\gamma^k(\eta(k),u(k),v(k)).
\]
This is alternating and bilinear in $u$ and $v$. We check that it is smooth. This is local on $K$. Fix $k_0\in K$ and choose a local trivialization $\tau\colon A|_U\to U\times\sF$ with $\gamma(k_0)\in U$. After shrinking an open neighborhood $O$ of $k_0$, assume $O\times U\subseteq W_\gamma$. Set $T(k,x):=\tau_{\gamma(k)}\circ\Theta^\gamma_{k,x}\circ\tau_x^{-1}$. Then $T(k,x)\in\GL(\sF)$ and $(k,x)\mapsto T(k,x)$ is smooth. In the trivialization induced by $\tau$, Remark~\ref{remark:coordinate-change-C} gives
\[
\begin{aligned}
C_\gamma^k(x,u,v)
=&\,T(k,x)C_\tau(x,T(k,x)^{-1}u,T(k,x)^{-1}v)\\
&-d_2T(k,x;\anch^\tau(x,T(k,x)^{-1}u))T(k,x)^{-1}v\\
&+d_2T(k,x;\anch^\tau(x,T(k,x)^{-1}v))T(k,x)^{-1}u.
\end{aligned}
\]
All terms are smooth in $(k,x,u,v)$. Hence the adjoint map $(\eta,u,v,k)\mapsto C_\gamma(\eta,u,v)(k)$ is smooth. %
By the exponential law, $C_\gamma$ is smooth.

Since evaluation at each $k\in K$ is continuous linear on $\Gamma(\gamma^*A)$, differentiation in the base variable commutes with evaluation. In particular, when $u$ and $v$ are  fixed,
\[
(d_1C_\gamma(\eta,u,v;h))(k)
=d_1C_{\gamma,k}(\eta(k),u(k),v(k);h(k)).
\]

Let $Q\subseteq\underline{M}$ be open, and let $X,Y\in\underline{\shA}(Q)$, where $\underline{\shA}$ is the sheaf of sections of $p_*\colon\underline{A}\to\underline{M}$. On $Q\cap U_\gamma$, write $X^\gamma=\pr_2\circ\Psi_\gamma\circ X$ and $Y^\gamma=\pr_2\circ\Psi_\gamma\circ Y$. Define
\[
B_\gamma(X,Y)(\eta):=\underline{\anch}(X).Y^\gamma(\eta)-\underline{\anch}(Y).X^\gamma(\eta)+C_\gamma(\eta,X^\gamma(\eta),Y^\gamma(\eta)).
\]
Then $\eta\mapsto\Psi_\gamma^{-1}(\eta,B_\gamma(X,Y)(\eta))$ is a smooth section over $Q\cap U_\gamma$.

We check that these local sections agree on overlaps. Let $\Psi_\delta$ be another local trivialization. Then $X^\delta(\eta)=S_{\delta\gamma}(\eta,X^\gamma(\eta))$ and $Y^\delta(\eta)=S_{\delta\gamma}(\eta,Y^\gamma(\eta))$. For fixed $k$, set $\wh{K}^k_{\delta\gamma}(x,w):=\Theta^\delta_{k,x}(\Theta^\gamma_{k,x})^{-1}w$. Applying Remark~\ref{remark:coordinate-change-C} pointwise in $k$ gives
\[
\begin{aligned}
&C_\delta(\eta,S_{\delta\gamma}(\eta,u),S_{\delta\gamma}(\eta,v))\\
&=S_{\delta\gamma}(\eta,C_\gamma(\eta,u,v))
-d_1S_{\delta\gamma}(\eta,v;\underline{\anch}^\gamma(\eta,u))
+d_1S_{\delta\gamma}(\eta,u;\underline{\anch}^\gamma(\eta,v)).
\end{aligned}
\]
Here $\underline{\anch}^\gamma$ is the local representative of $\underline{\anch}$. The chain rule gives
\[
\underline{\anch}(X).Y^\delta(\eta)
=d_1S_{\delta\gamma}(\eta,Y^\gamma(\eta);\underline{\anch}^\gamma(\eta,X^\gamma(\eta)))
+S_{\delta\gamma}(\eta,\underline{\anch}(X).Y^\gamma(\eta)).
\]
The same formula holds with $X$ and $Y$ interchanged. The derivative terms cancel. Hence
\[
B_\delta(X,Y)(\eta)=S_{\delta\gamma}(\eta,B_\gamma(X,Y)(\eta)).
\]
Thus the local sections glue. We denote the glued section by $[X,Y]_{\underline{\shA}(Q)}$.

The construction is compatible with restrictions. Hence $[\cdot,\cdot]_{\underline{\shA}}\colon\underline{\shA}\otimes_\mathbb{R}\underline{\shA}\to\underline{\shA}$ is a morphism of sheaves of $\mathbb{R}$-modules. It is $\mathbb{R}$-bilinear and alternating by construction.

The Leibniz rule is local. Let $f\in C^\infty(Q)$. Since $(fY)^\gamma(\eta)=f(\eta)Y^\gamma(\eta)$, we get
\[
B_\gamma(X,fY)(\eta)=(\underline{\anch}(X).f)(\eta)Y^\gamma(\eta)+f(\eta)B_\gamma(X,Y)(\eta).
\]
Thus $[X,fY]_{\underline{\shA}(Q)}=f[X,Y]_{\underline{\shA}(Q)}+(\underline{\anch}(X).f)Y$.

It remains to check the Jacobi identity. This is local. First, Lemma~\ref{lemma:anchsheaf} applied in each trivialization $\Psi_\gamma^k$ gives
\begin{equation}\label{eq:anch-comp-thm2.10}
d_1\underline{\anch}^\gamma(\eta,v;\underline{\anch}^\gamma(\eta,u))
-d_1\underline{\anch}^\gamma(\eta,u;\underline{\anch}^\gamma(\eta,v))
=\underline{\anch}^\gamma(\eta,C_\gamma(\eta,u,v)).
\end{equation}
Hence $\underline{\anch}[X,Y]_{\underline{\shA}(Q)}=[\underline{\anch}(X),\underline{\anch}(Y)]$. Let $X,Y,Z\in\underline{\shA}(Q)$. In the trivialization $\Psi_\gamma$,
the terms containing second derivatives of $X^\gamma,Y^\gamma,Z^\gamma$ vanish
by symmetry of the second derivatives. 
After using~\eqref{eq:anch-comp-thm2.10}, the remaining terms in
the Jacobiator reduce to
\[
\begin{aligned}
J_\gamma(X,Y,Z)(\eta)=\sum_{\mathrm{cyc}}\Big(
&d_1C_\gamma(\eta,Y^\gamma(\eta),Z^\gamma(\eta);\underline{\anch}^\gamma(\eta,X^\gamma(\eta)))\\
&+C_\gamma(\eta,X^\gamma(\eta),C_\gamma(\eta,Y^\gamma(\eta),Z^\gamma(\eta)))\Big).
\end{aligned}
\]
Evaluating at $k\in K$, this is the Jacobi identity for the original bracket on $A$, written in the local trivialization $\Psi_\gamma^k$, and tested on the constant local sections with values $X^\gamma(\eta)(k)$, $Y^\gamma(\eta)(k)$, and $Z^\gamma(\eta)(k)$. Hence $J_\gamma(X,Y,Z)(\eta)(k)=0$ for all $k$. Therefore $J_\gamma(X,Y,Z)=0$. Thus $[\cdot,\cdot]_{\underline{\shA}}$ defines a Lie bracket on every $\underline{\shA}(Q)$.

Finally, by construction, in the local trivialization $\Psi_\gamma$ one has
\[
[X,Y]^\gamma(\eta)=\underline{\anch}(X).Y^\gamma(\eta)-\underline{\anch}(Y).X^\gamma(\eta)+C_\gamma(\eta,X^\gamma(\eta),Y^\gamma(\eta)).
\]
The map $C_\gamma$ is smooth and fiberwise alternating bilinear. Hence the current Lie algebroid is first-order in the trivializations $\Psi_\gamma$. Since the $\Psi_\gamma$ form a local trivializing cover, the coordinate-change formula of Remark~\ref{remark:coordinate-change-C} gives the first-order condition in every local trivialization. Therefore $p_*\colon C^\infty(K,A)\to C^\infty(K,M)$ is a first-order Fr\'echet-Lie algebroid.
\end{proof}

\begin{corollary}[Current action algebroid]
Let $K$ be a compact smooth manifold and let $\mathfrak{g}$ be a Banach-Lie algebra with smooth bracket acting smoothly on a Banach manifold $M$. Assume that $M$ admits a local addition.
Then $C^\infty(K,M)\times C^\infty(K,\mathfrak{g})\to C^\infty(K,M)$ is a first-order Fr\'echet-Lie algebroid.
\end{corollary}

\begin{proof}
By Example~\ref{ex:action-lie-algbd}, the action Lie algebroid $M\times\mathfrak{g}\to M$ is first-order. Since $M$ admits a local addition, we can use Theorem~\ref{lemma:current-first-order}. Using the canonical identification $C^\infty(K,M\times\mathfrak{g})\cong C^\infty(K,M)\times C^\infty(K,\mathfrak{g})$, we get the result.
\end{proof}

Finally we show that the Lie algebroid associated with a Lie groupoid is first-order.

\begin{theorem}\label{thm:groupoid-algebroid-first-order}
Let $\mathcal{G}\tto M$ be a  Lie groupoid.
Then $\Lie(\mathcal{G})$ 
is a first-order Lie algebroid.
\end{theorem}

\begin{proof}
For an open set $U\subseteq M$,
let $\shR(U)$ be the space of smooth sections of $\ker(T\src)$ over $\tgt^{-1}(U)$
which are right-invariant.
The $\shcinf_M(U)$-module structure is given by $(f\cdot X)(g):=f(\tgt(g))X(g)$.
If $\xi\in \Gamma(\Lie(\mathcal{G})|_U)$,
define its right-invariant extension by
$\ora{\xi}(g):=T_{\one(\tgt(g))}R_g(\xi(\tgt(g)))$, for $g\in \tgt^{-1}(U)$.
This gives an isomorphism of sheaves of $\shcinf_M$-modules
$\Gamma(\Lie(\mathcal{G})|-)\to \shR$, with inverse $X\mapsto X\circ \one$.
Because each right translation is a diffeomorphism preserving $\ker(T\src)$,
the bracket of two right-invariant $\src$-vertical vector fields
is again right-invariant and $\src$-vertical.
Transporting the bracket by the above isomorphism
yields a sheaf morphism $\LB\colon \Gamma(\Lie(\mathcal{G})|-)\times \Gamma(\Lie(\mathcal{G})|-)\to \Gamma(\Lie(\mathcal{G})|-)$.
The Jacobi identity and antisymmetry come from the     usual Lie bracket of vector fields.
The Leibniz rule is also immediate.
Indeed, $\ora{f\xi}=(\tgt^*f)\ora{\xi}$,
and hence $$[\ora{\xi},\ora{f\eta}]=(\tgt^*f)[\ora{\xi},\ora{\eta}]+(\ora{\xi}.\tgt^*f)\,\ora{\eta}.$$
Restricting along the unit section gives $[\xi,f\eta]=f[\xi,\eta]+(\anch(\xi).f)\eta$.
It remains to verify the local first-order formula.
Fix $x\in M$.
Since $\src\colon \mathcal{G}\to M$ is a submersion
and $\one\colon M\to \mathcal{G}$ is a section,
after shrinking we may choose a source-adapted chart $\chi\colon W\to U\times V$ about $\one(x)$
such that $\chi(\one(y))=(y,0)$ and $\src\circ \chi^{-1}(y,v)=y$, where $V$ is an open subset of a locally convex space $\sF$.
This chart identifies $\Lie(\mathcal{G})|_U$ with the trivial bundle $U\times \sF$;
let $\phi$ denote the corresponding trivialization.
Write $\tau:=\tgt\circ \chi^{-1}\colon U\times V\to M$.
Then $\tau(y,0)=y$,
and in the trivialization $\phi$ the anchor is $\anch^\phi(y,u)=d_2\tau(y,0;u)\in T_yU$.
Locally we can write the multiplication as $(y,\nu(y,v,w)):=\chi(\chi^{-1}(\tau(y,v),w)\chi^{-1}(y,v))$.
Then locally the derivative of right multiplication along
the $\one$ is given by $R_\phi(y,v,u):=d_3\nu(y,v,0;u)$.
Immediately, $R_\phi(y,0,u)=u$.
This map has the property that, 
for every local section $\xi\in \Gamma(\Lie(\mathcal{G})|_U)$,
the corresponding right-invariant vector field has local expression
$\ora{\xi}^\chi(y,v)=(0,R_\phi(y,v,\xi^\phi(\tau(y,v))))$.
Therefore, at a point $(y,0)$,
the vertical part of the bracket is obtained by differentiating only in the second variable.
Using the chain rule and the identity $R_\phi(y,0,u)=u$, one gets
\[ [\xi,\eta]^\phi(y)=\anch(\xi).\eta^\phi(y)-\anch(\eta).\xi^\phi(y)+C_\phi(y,\xi^\phi(y),\eta^\phi(y)), \]
where $C_\phi(y,u,w):=d_2R_\phi(y,0,w;u)-d_2R_\phi(y,0,u;w)$.
Since $R_\phi$ is smooth and linear in the third variable,
$C_\phi$ is smooth and $(u,w)\mapsto C_\phi(y,u,w)$ is alternating bilinear.
Thus Definition~\ref{def:liealgebroid} is satisfied.
\end{proof}

\subsection*{The sheaf $\Omega^\bullet_A(-,E)$}
Before defining morphisms of Lie algebroids, we first consider sheaves of forms with values in a locally convex vector bundle. Associated with these sheaves is a well-defined differential, which can be used to generalize the definition of Lie algebroid morphisms in terms of pullbacks of forms, without referring to any topology on spaces of linear maps.

\begin{definition}
Let $p\colon A\to M$ be a first-order Lie algebroid.
An \Emph{$A$-connection} on a vector bundle $q\colon E\to M$,
with a sheaf of section $\shE$, 
is a morphism of sheaves of $\R$-modules $\nabla\colon \shA\otimes_{\R} \shE\to \shE$
which is $\shcinf_M$-linear in the first argument,
additive in the second,
and satisfies $\nabla_\xi(f\sigma)=f\nabla_\xi\sigma+(\anch(\xi).f)\sigma$
for all open $U\subseteq M$, $\xi\in \shA(U)$, $\sigma\in \shE(U)$, and $f\in \shcinf_M(U)$.
We moreover require the following local first-order condition.
If $\phi\colon A|_U\to U\times \sF_\phi$ and $\psi\colon E|_U\to U\times \sF_\psi$
are local trivializations over the same open set $U$,
then there exists a smooth map $B_{\phi\psi}\colon U\times \sF_\phi\times \sF_\psi\to \sF_\psi$
such that $(u,v)\mapsto B_{\phi\psi}(x,u,v)$ is bilinear for each $x\in U$, and
\[ (\nabla_\xi\sigma)^{\phi\psi}(x)=\anch(\xi).\sigma^\psi(x)+B_{\phi\psi}(x,\xi^\phi(x),\sigma^\psi(x)) \]
for all $\xi\in \shA(U)$ and $\sigma\in \shE(U)$.
The curvature of $\nabla$ is the morphism $R_\nabla\colon \shA\otimes_{\R} \shA\otimes_{\R} \shE\to \shE$
given by $R_\nabla(\xi,\eta)\sigma:=\nabla_\xi\nabla_\eta\sigma-\nabla_\eta\nabla_\xi\sigma-\nabla_{[\xi,\eta]}\sigma$.
If $R_\nabla=0$, then $\nabla$ is called \Emph{flat}.
\end{definition}

\begin{remark}
Let $\phi \colon A|_U \to U \times \sF_\phi$ and $\wt\phi \colon A|_{\wt U} \to \wt U \times \sF_{\wt\phi}$ be local trivializations of $A$. Let $\psi \colon E|_U \to U \times \sF_\psi$ and $\wt\psi \colon E|_{\wt U} \to {\wt U} \times \sF_{\wt\psi}$ be local trivializations of $E$. Set
$
(\wt\phi \circ \phi^{-1})(x,u) = (x,\wh K(x,u))$, and $
(\wt\psi \circ \psi^{-1})(x,e) = (x,\wh G(x,e)),
$
where the maps $\wh K$ and $\wh G$ are smooth and fiberwise linear. Let $\anch^{\phi}$ be the local representative of the anchor in the trivialization $\phi$. Then, on $U \cap \wt U$,
 we have
\[
B_{\wt\phi\wt\psi}(x,\wh K(x,u),\wh G(x,e))
=
\wh G(x,B_{\phi,\psi}(x,u,e))
- d_1\wh G(x,e;\anch^\phi(x,u)).
\]
This follows by writing $(\nabla_\xi \sigma)^{\wt\phi\wt\psi}$ in the two systems of trivializations, using $\sigma^{\wt\psi}(x)=\wh G(x,\sigma^\psi(x))$, and comparing the two formulas.
\end{remark}

Let $A\to M$ be a first-order Lie algebroid and 
let $E\to M$ be a vector bundle with an $A$-connection $\nabla$.
For $k\in \N_0$ and an open set $U\subseteq M$,
define $\Omega_A^k(U,E)$ to be the space of maps $x\mapsto \alpha_x$,
where each $\alpha_x$ 
is a continuous alternating $k$-linear map
$A_x^k\to E_x$,
such that for every $x\in U$, there exist
 local trivializations $\phi\colon A|_{U_\phi}\to U_\phi\times\sF_\phi$ and $\psi\colon E|_{U_\psi}\to U_\psi\times \sF_\psi$
with $x\in U_\phi=U_\psi\subseteq U$ for which the local representative
$ \alpha_{\phi\psi}\colon  U_\phi\times \sF_\phi^k\to \sF_\psi$ 
 given by
\[ (y,v_1,\ldots,v_k)\mapsto \pr_2\Bigl(\psi\bigl(\alpha_y(\phi^{-1}(y,v_1),\ldots,\phi^{-1}(y,v_k))\bigr)\Bigr), \]
is smooth. 
Equivalently, if this holds for one pair of local trivializations around x, it holds for every such pair, by the smoothness and fiberwise linearity of the transition maps.
Here we set $\Omega^0_A(-,E)=\Gamma(E|-).$
The usual restriction maps make $\Omega_A^k(-,E)$ into a sheaf of $\shcinf_M$-modules.
For $\alpha\in \Omega_A^k(U,E)$ and $\xi_j\in \shA(U)$,
define $\alpha(\xi_1,\ldots,\xi_k)\in \shE(U)$ by $x\mapsto \alpha_x(\xi_1(x),\ldots,\xi_k(x))$.
This yields the evaluation morphism $\operatorname{Ev}\colon \Omega_A^k(-,E)\times \shA^k\to \shE$.
For $\alpha\in \Omega_A^k(U,E)$,
define $d_{A,\nabla}\alpha\in \Omega_A^{k+1}(U,E)$ by setting
\begin{equation}\label{eq:dAnabla} \begin{split} 
(d_{A,\nabla}\alpha)_x(a_0,\ldots,a_k) &= \sum_{i=0}^k (-1)^i \nabla_{\xi_i}\bigl(\alpha(\xi_0,\ldots,\wh\xi_i,\ldots,\xi_k)\bigr) (x) \\
&+\sum_{0\leq i<j\leq k} (-1)^{i+j}\alpha\bigl([\xi_i,\xi_j],\xi_0,\ldots,\wh\xi_i,\ldots,\wh\xi_j,\ldots,\xi_k\bigr)(x),
\end{split}
\end{equation}
for $x\in U$,
$a_i\in A_x$, and for any local sections 
$\xi_i  \in \shA(V)$ such that $\xi_i(x) = a_i$ where
$V\subseteq U$ is open. This makes sense since one can take $V=U\cap U_\phi$ where $\phi$ is a local trivialization, consider constant sections $\xi_i(y)=\phi^{-1}(y,\phi_x(a_i))$ and use the fact that 
$\nabla$ is a sheaf morphism.

\begin{lemma}
The operator $d_{A,\nabla}\colon \Omega_A^k(-,E)\to \Omega_A^{k+1}(-,E)$ is well-defined.
If $\nabla$ is flat, then $d_{A,\nabla}^2=0$.
\end{lemma}

\begin{proof}
	Let $x\in M$ and
	choose local trivializations $\phi\colon A|U\to U\times \sF_{\phi}$ and $\psi\colon E|_U\to U\times \sF_{\psi}$ 
    over
	the same open neighborhood of $x$. For any $a_j\in A_x$,
	set $u_j := \pr_2(\phi(a_j))$ and take $\xi_i\in \shA(U)$ 
    (these can be the constant sections) 
    such that 
	$\xi_i(x) = a_i$.
	In~\eqref{eq:dAnabla}, the first sum is over terms of the form:
	\[\begin{split} 
		(\nabla_{\xi_i}\alpha &(\xi_0,\ldots, \wh \xi_i, \ldots, \xi_k))^{\phi\psi}(x) \\&= \anch(\xi_i).(\alpha_{\phi\psi}(\cdot, u_0,\ldots, \wh {u_i},\ldots, u_k ))(x)\\ &  
		 + B_{\phi\psi}(x, u_i, 
		\alpha_{\phi\psi}(x, u_0,\ldots, \wh u_i, \ldots, u_k)
		)\\
		&+ \sum_{j< i}^k (-1)^j \alpha_{\phi\psi}(x,\anch(\xi_i).\xi_j^\phi(x),u_0,\ldots, \wh u_i, \ldots, \wh u_j, \ldots, u_k),
        \\
		&- \sum_{j> i}^k (-1)^j \alpha_{\phi\psi}(x,\anch(\xi_i).\xi_j^\phi(x),u_0,\ldots, \wh u_i, \ldots, \wh u_j, \ldots, u_k),
	\end{split}\]
	while the second sum involves terms of the form:
	\[\begin{split} 
		\alpha_{\phi\psi}(x, 
		\anch(\xi_i).\xi_j^\phi - \anch(\xi_j).\xi_i^\phi + 
		C_{\phi}(x, u_i, u_j), u_0, \ldots, \wh u_i, \ldots, \wh u_j,\ldots, u_k)
		).
	\end{split} 
	\]
	Thus
    \eqref{eq:dAnabla} becomes:
	\begin{equation}\label{eq:dAnablalocal}
		\begin{split}
			&(d_{A,\nabla}\alpha)_{\phi\psi} (x,u_0,\ldots, u_k) \\
			&=\sum_{i=0}^k
            (-1)^i  \anch(\xi_i).(\alpha_{\phi\psi}(\cdot, u_0,\ldots, \wh {u_i},\ldots, u_k ))(x)  \\ &\quad  
			+   \sum_{i=0}^k (-1)^i
			B_{\phi\psi}(x, u_i,
			\alpha_{\phi\psi}(x, u_0,\ldots, \wh u_i, \ldots, u_k))
             \\
			&\quad +\sum_{ i< j }  (-1)^{i+j} \alpha_{\phi\psi}(x, 
			C_{\phi}(x, u_i, u_j), u_0, \ldots, \wh u_i, \ldots, \wh u_j,\ldots, u_k)
			).
	\end{split} \end{equation}
    For each $i$, the term $\anch(\xi_i).(\alpha_{\phi\psi}(\cdot, u_0,\ldots, \wh {u_i},\ldots, u_k ))(x)$ depends on $\xi_i(x)=a_i$. Hence
	the right hand side of~(\ref{eq:dAnablalocal})  is independent of the sections $\xi_i$ and is well-behaved with
	respect to restricting to an open set $V\subseteq U$.
	Additionally, we immediately see that  $(d_{A,\nabla}\alpha)_{\phi\psi}$ is
	smooth and alternating-multilinear in $u_0,\ldots, u_k$.
	Hence, $d_{A,\nabla}\alpha$ is an element of
    $\Omega^{k+1}_A(U,E)$.

	For the second statement, note that, if we replace $\alpha$ with $d_{A,\nabla}\alpha$ in~(\ref{eq:dAnabla}) and simplify,
    we get
	\begin{equation*}
		(d^2_{A,\nabla}\alpha)(\xi_0,\ldots, \xi_{k+1}) =
		\sum_{i<j}(-1)^{i+j+1} R_\nabla(\xi_i,\xi_j)
		\alpha(\xi_0, \ldots, \wh\xi_i, \ldots, \wh\xi_j, \ldots, \xi_{k+1}).
	\end{equation*}
    Hence $d_{A,\nabla}^2=0$ if $R_{\nabla}=0$.
\end{proof} 

A vector bundle with a flat $A$-connection is called a \Emph{representation} of $A$.
For every representation $(E,\nabla)$ of $A$,
the graded sheaf $\Omega^\bullet_A(-,E):=\bigoplus_{k\in \N_0}\Omega_A^k(-,E)$,
together with $d_{A,\nabla}$, is a sheaf cochain complex.
If one takes $E=M\times \R$ with the trivial flat $A$-connection,
one obtains the sheaf of scalar-valued $A$-forms.
We write $\Omega_A^k$ for this sheaf.
Thus $\omega\in \Omega_A^k(U)$ means that each $\omega_x$
is a continuous alternating $k$-linear form on $A_x$
and, in any local trivialization $\phi$ of $A$, the map
\[ \omega_\phi\colon (U\cap U_\phi)\times \sF_\phi^k\to \R,\quad (x,v_1,\ldots,v_k)\mapsto \omega_x(\phi^{-1}(x,v_1),\ldots,\phi^{-1}(x,v_k)), \]
is smooth.
For $\omega\in \Omega_A^k(U)$, the differential is
\begin{equation} \label{eq:dA}
\begin{split}
(d_A\omega)(\xi_0,\ldots,\xi_k) &= \sum_{j=0}^k (-1)^j \anch(\xi_j)\bigl(\omega(\xi_0,\ldots,\wh\xi_j,\ldots,\xi_k)\bigr) \\
&+\sum_{0\leq i<j\leq k} (-1)^{i+j}\omega\bigl([\xi_i,\xi_j],\xi_0,\ldots,\wh\xi_i,\ldots,\wh\xi_j,\ldots,\xi_k\bigr). 
\end{split}
\end{equation}
If $\alpha\in \Omega_A^k(U)$ and $\beta\in \Omega_A^l(U)$,
define their wedge product by
\[ (\alpha\wedge \beta)_x(a_1,\ldots,a_{k+l}) := \sum_{s\in \mathfrak{S}_{k,l}} \mathrm{sgn}(s)\, \alpha_x(a_{s(1)},\ldots,a_{s(k)}) \beta_x(a_{s(k+1)},\ldots,a_{s(k+l)}), \]
where $\mathfrak{S}_{k,l}$ is the set of $(k,l)$-shuffles.
Then $\alpha\wedge\beta\in \Omega_A^{k+l}(U)$,
one has $\alpha\wedge\beta=(-1)^{kl}\beta\wedge\alpha$,
and 
$$d_A(\alpha\wedge\beta)=(d_A\alpha)\wedge\beta+(-1)^k\alpha\wedge d_A\beta.$$
Therefore, $\Omega_A:=\bigoplus_{k\in \N_0}\Omega_A^k$ is a sheaf of dg-$\R$-algebras.
In the special case $A=TM$ and $\anch=\id_{TM}$,
we write $\Omega_M^k:=\Omega_{TM}^k$,
and the differential is the usual de~Rham differential.
More generally, if $(E,\nabla)$ is a representation
and $\alpha\in \Omega_A^k(U)$, $\beta\in \Omega_A^l(U,E)$,
then $$d_{A,\nabla}(\alpha\wedge\beta)=d_A\alpha\wedge\beta+(-1)^k\alpha\wedge d_{A,\nabla}\beta.$$
Hence $\Omega^\bullet_A(-,E)$ is a sheaf of dg-modules over the sheaf of dg-algebras $\Omega^\bullet_A$.
Finally, let $\sE$ be a locally convex space
equipped with a continuous bilinear associative commutative  multiplication $\Theta\colon \sE\times \sE\to \sE$.
Then the same shuffle formula, with the pointwise product induced by $\Theta$,
makes $\Omega^\bullet_A(-,\sE):=\Omega^\bullet_A(-,M\times\sE)$ (with the trivial flat $A$-connection) into a sheaf of dg-algebras.
More generally, suppose that $E\to M$ is a vector bundle
with a smooth fiberwise bilinear commutative associative multiplication $\Theta\colon E\oplus E\to E$,
and let $\nabla$ be a flat $A$-connection on $E$
such that $$\nabla_\xi(\Theta(\sigma,\tau))=\Theta(\nabla_\xi\sigma,\tau)+\Theta(\sigma,\nabla_\xi\tau)$$
for all local sections $\xi\in \shA(U)$ and $\sigma,\tau\in \shE(U)$.
Then for $\alpha\in \Omega_A^k(U,E)$ and $\beta\in \Omega_A^l(U,E)$, one defines
\begin{multline*} 
(\alpha\wedge\beta)_x(a_1,\ldots,a_{k+l})  :=  \sum_{s\in \mathfrak{S}_{k,l}} \mathrm{sgn}(s)\, \Theta_x\bigl( \alpha_x(a_{s(1)},\ldots,a_{s(k)}), \beta_x(a_{s(k+1)},\ldots,a_{s(k+l)}) \bigr). 
\end{multline*}
This gives a well-defined element of $\Omega_A^{k+l}(U,E)$,
and with this product $\Omega^\bullet_A(-,E)$ becomes a sheaf of dg-algebras.

The cohomology of $A$ with coefficients in $E$ is the cohomology of the global complex $\Omega_A(M,E)$:
\[
H^k_A(M,E):=
\frac{\ker(d_{A,\nabla}\colon \Omega_A^k(M,E)\to \Omega_A^{k+1}(M,E))}
{\operatorname{im}(d_{A,\nabla}\colon \Omega_A^{k-1}(M,E)\to \Omega_A^k(M,E))}.
\]
We also denote by $\shH_A^k(-,E)$ the $k$-th cohomology sheaf of  $\Omega_A(-,E)$.
In degree zero one has
$$H^0_A(M,E)=\{\sigma\in \Gamma(E):\nabla_\xi(\sigma|_U)=0\text{ for every open } U\subseteq M \text{ and every } \xi\in\shA(U)\}.$$
Indeed, for $\sigma\in \Omega_A^0(M,E)=\Gamma(E)$, the differential is given by
$(d_{A,\nabla}\sigma)(\xi)=\nabla_\xi\sigma$.
If $E=M\times\mathbb R$ is the trivial representation, we write $H^\bullet_A$.
The wedge product on scalar-valued forms reduces to a graded-commutative algebra structure on $H^\bullet_A$.
More generally, if $(E,\nabla)$ is a representation with a map $\Theta$ as above,
then $\shH^\bullet_A(-,E)$ is a sheaf of graded-commutative $\R$-algebras
and a graded module over $\shH^\bullet_A$.
If, moreover, $\Theta$ has a $\nabla$-flat unit, then $\shH^\bullet_A(-,E)$ is a sheaf of graded-commutative $\shH^\bullet_A$-algebras.

\begin{example}
Let $A=TM$ with $\anch=\id_{TM}$, and let $E=M\times\sE$ for a locally convex space $\sE$.
Equip $E$ with the trivial representation $\nabla_X\sigma=X.\sigma$.
Then $\Omega_A^k(M,E)$ is the usual space of $\sE$-valued smooth $k$-forms on $M$, in the Bastiani sense, and $d_{A,\nabla}$ is the ordinary de~Rham differential~\cite{GloecknerNeeb2026InfiniteDimensional}.
Thus $H^\bullet_{TM}(M,M\times\sE)=H^\bullet_{\mathrm{dR}}(M,\sE)$.
If $h\colon M\to N$ is smooth, then the morphism $Th\colon TM\to TN$ induces the usual pullback in de~Rham cohomology.
\end{example}

\begin{example}
Let $M=\{*\}$ and let $A=\mathfrak g\to\{*\}$ be a locally convex Lie algebra with smooth bracket.
Let $\sE$ be a locally convex space, set $E:=\{*\}\times\sE$, and let $\pi\colon \mathfrak g\times\sE\to\sE$ be a smooth representation.
Then an $E$-valued $k$-form is exactly a continuous alternating $k$-linear map $c\colon \mathfrak g^k\to\sE$.
The differential is
\begin{multline*}
    (dc)(u_0,\ldots,u_k)
=
\sum_i(-1)^i\pi(u_i)c(u_0,\ldots,\widehat u_i,\ldots,u_k)
\\+
\sum_{i<j}(-1)^{i+j}c([u_i,u_j],u_0,\ldots,\widehat u_i,\ldots,\widehat u_j,\ldots,u_k).
\end{multline*}
Hence $H^\bullet_A(\{*\},E)$ is the continuous Chevalley-Eilenberg cohomology of $\mathfrak g$ with coefficients in $\sE$.
\end{example}

\begin{example}
In the setting of Example~\ref{ex:action-lie-algbd}, let $\sE$ be a locally convex space, let $\pi\colon \fg\times\sE\to\sE$ be a smooth representation of $\fg$, and denote $\pi(u)e=\pi(u,e)$.
That is, $\pi$ is bilinear, $\pi(u)\in \mathcal L(\sE)$, and $$\pi([u,v])e=\pi(u)\pi(v)e-\pi(v)\pi(u)e.$$
Set $E:=M\times\sE$.
For $\xi\in C^\infty(U,\fg)$ and $\sigma\in C^\infty(U,\sE)$, define
$$
(\nabla_\xi\sigma)(x)
:=
d\sigma\bigl(x;\anch(x,\xi(x))\bigr)
+\pi(\xi(x))\sigma(x).
$$
Then $\nabla$ is a flat $A$-connection on $E$.
Indeed, $C^\infty_U$-linearity in $\xi$ follows from the linearity of $\anch$ and $\pi$.
For $f\in C^\infty(U)$ one has
$\nabla_\xi(f\sigma)=f\nabla_\xi\sigma+(\anch(\xi).f)\sigma$.
In the global trivialization $E=M\times\sE$, the local connection term is $B(x,u,e)=\pi(u)e$, which is smooth and bilinear.
It remains to show that $R_\nabla=0$.
Let $\xi,\eta\in C^\infty(U,\fg)$ and let $\sigma\in C^\infty(U,\sE)$.
For a smooth map $\tau$ on $U$ with values in a locally convex space, put
$D_\xi\tau(x):=d\tau\bigl(x;\anch(x,\xi(x))\bigr)$.
The bracket in the action Lie algebroid is
$[\xi,\eta]=D_\xi\eta-D_\eta\xi+[\xi,\eta]_\fg$,
where $[\xi,\eta]_\fg(x):=[\xi(x),\eta(x)]_\fg$.
Using the product rule and the representation property of $\pi$, we obtain
\begin{align*}
(\nabla_\xi\nabla_\eta-\nabla_\eta\nabla_\xi)\sigma
&= [D_\xi,D_\eta]\sigma
+\pi(D_\xi\eta-D_\eta\xi)\sigma
+\pi([\xi,\eta]_\fg)\sigma \\
&= D_{[\xi,\eta]}\sigma+\pi([\xi,\eta])\sigma \\
&= \nabla_{[\xi,\eta]}\sigma .
\end{align*}
Hence $R_\nabla(\xi,\eta)\sigma=0$, and $\nabla$ is flat.
The complex $\Omega^\bullet_A(M,E)$ has the following explicit form.
A $k$-form $\alpha\in \Omega_A^k(M,E)$ is a smooth map $\alpha\colon M\times\fg^k\to\sE$ which is continuous alternating $k$-linear in the $\fg$-variables.
For $u_i\in\fg$, the differential is
\begin{align*}
(d_{A,\nabla}\alpha)(x;u_0,\ldots,u_k)
={}&
\sum_{i=0}^k(-1)^i
d\bigl(\alpha(\,\cdot\,,u_0,\ldots,\widehat u_i,\ldots,u_k)\bigr)
\bigl(x;\anch(x,u_i)\bigr) \\
&+\sum_{i=0}^k(-1)^i
\pi(u_i)\alpha(x,u_0,\ldots,\widehat u_i,\ldots,u_k) \\
&+\sum_{0\leq i<j\leq k}(-1)^{i+j}
\alpha(x,[u_i,u_j],u_0,\ldots,\widehat u_i,\ldots,\widehat u_j,\ldots,u_k).
\end{align*}
If $M$ is compact, then, using the exponential law, $\alpha$ corresponds to an element
$\alpha^\vee \in \mathcal L^k_{\mathrm{alt}}(\fg, C^\infty(M,\sE))$.
Thus this complex reduces to the continuous Chevalley-Eilenberg complex for $\fg$ with coefficients in $C^\infty(M,\sE)$ where the
smooth action $\fg\times C^\infty(M,\sE)\to C^\infty(M,\sE)$ 
is given by
$$(u,\sigma)\mapsto 
(u\cdot\sigma)(x)
:=
d\sigma\bigl(x;\anch(x,u)\bigr)+\pi(u)\sigma(x).
$$
That is, we have an isomorphism $\big(\Omega^\bullet_A(M,E),d_{A,\nabla} \big)\cong C^\bullet_{\mathrm{CE,cont}}\big(\fg, C^\infty(M,\sE) \big)$ and hence
$H^\bullet_A(M,E)\cong
H^\bullet_{\mathrm{CE,cont}}\bigl(\fg,C^\infty(M,\sE)\bigr)$.
\end{example}

\section{Morphisms between first-order Lie algebroids}\label{sec:liealgbdmorphism}

Let $p\colon A\to M$ and $\wt p\colon \wt A\to \wt M$ be
first-order Lie algebroids.
For each smooth map $f\colon M\to \wt M$,
there is a canonical morphism of sheaves of $\R$-algebras $f^\sharp\colon f^{-1}\shcinf_{\wt M}\to \shcinf_M$.
Recall that $f^{-1}\shcinf_{\wt M}$ is the sheafification of the presheaf
$$U\mapsto \varinjlim_{\wt U\supseteq f(U)} \shcinf_{\wt M}(\wt U)$$
and that, on representatives, one has $f^\sharp_U([\,\wt U,g\,])=g\circ f|_U$.
Now let $F\colon A\to \wt A$ be a smooth vector bundle morphism over $f$.
For $k\geq 1$, we use the universal property of sheaves to define a morphism of sheaves $F^\sharp\colon f^{-1}\Omega^k_{\wt A}\to \Omega_A^k$
where the underlying morphism of presheaves is given 
by $F^\sharp_U([\,\wt U,\alpha\,])=F^*\alpha$,
where $$(F^*\alpha)_x(a_1,\ldots,a_k)=\alpha_{f(x)}(F(a_1),\ldots,F(a_k)),\quad 
\text{for } x\in U \text{ and } a_i\in A_x.$$
For $k=0$ we set $F^\sharp=f^\sharp$.

In the results below, identities involving inverse-image sheaves are verified on the corresponding inverse-image presheaves; the asserted morphisms of sheaves then follow from the universal property of sheafification.

\begin{lemma}\label{lemma:graded-commutative-morphism}
The morphism $F^\sharp\colon f^{-1}\Omega^\bullet_{\wt A}\to \Omega^\bullet_A$ is a well-defined
 morphism of sheaves of graded-commutative $\R$-algebras.
\end{lemma}

\begin{proof}
Suppose that $[\wt U,\alpha]=[\wt V,\beta]$ in the inverse-image presheaf.
Then there exists an open set $\wt W\subseteq \wt U\cap \wt V$ with $f(U)\subseteq \wt W$
such that $\alpha|_{\wt W}=\beta|_{\wt W}$.
By the definition of pullback, $F^*(\alpha|_{\wt W})=F^*(\beta|_{\wt W})$,
hence $F^*\alpha=F^*\beta$ on $U$.
Thus as presheaf morphism, this map is well-defined.
It is compatible with restrictions by construction.
Since $\Omega_A^k$ is a sheaf, this presheaf map induces a morphism of sheaves
$F^\sharp\colon f^{-1}\Omega^k_{\wt A}\to \Omega_A^k$.
It remains to check that the image is a smooth $A$-form.
This is local on $U$.
In local trivializations of $A$ and $\wt A$,
the form $F^*\alpha$ is obtained by composing the smooth local representative of $\alpha$
with the smooth local representative of $F$.
Hence $F^*\alpha$ has smooth local representatives.
Finally, let $[\wt U_1,\alpha]\in f^{-1}\Omega_{\wt A}^k(U)$
and $[\wt U_2,\beta]\in f^{-1}\Omega_{\wt A}^l(U)$.
After restricting to an open set $\wt W\subseteq \wt U_1\cap \wt U_2$ containing $f(U)$,
their product is represented by $[\wt W,\alpha\wedge\beta]$.
Pointwise, $F^*(\alpha\wedge\beta)=F^*\alpha\wedge F^*\beta$.
Thus $F^\sharp$ is multiplicative.
It also preserves constants, hence it is a morphism of sheaves of graded-commutative $\R$-algebras.
\end{proof}

By functoriality of inverse images,
the differential $d_{\wt A}$ induces a differential
$$f^{-1}d_{\wt A}\colon f^{-1}\Omega^k_{\wt A}\to f^{-1}\Omega^{k+1}_{\wt A},\quad
[\wt U,\alpha]\mapsto [\wt U,d_{\wt A}\alpha].$$
Thus $f^{-1}\Omega_{\wt A}$ is a sheaf of dg-$\R$-algebras.

\begin{definition}\label{def:liealgebroidmor}
Let $A\to M$ and $\wt A\to \wt M$ be first-order Lie algebroids.
A smooth vector bundle morphism $F\colon A\to \wt A$ over $f\colon M\to \wt M$
is called a \Emph{first-order Lie algebroid morphism}
if $F^\sharp\colon f^{-1}\Omega^\bullet_{\wt A}\to \Omega^\bullet_A$
is a morphism of sheaves of dg-$\R$-algebras, that is, if
\begin{equation}\label{eq:dgLie}
F^\sharp\circ f^{-1}d_{\wt A}=d_A\circ F^\sharp
\end{equation}
as morphisms of degree $+1$.
\end{definition}

Immediately, we see that identity $(\id_A,\id_M)$ is a Lie algebroid morphism.

\begin{lemma}
Let $A_i\to M_i$, $i=1,2,3$, be Lie algebroids,
and let $F_i\colon A_i\to A_{i+1}$ be vector bundle morphisms
over $f_i\colon M_i\to M_{i+1}$, $i=1,2$.
Using the canonical identification $(f_2\circ f_1)^{-1}\cong f_1^{-1}f_2^{-1}$, one has
\begin{equation}\label{eq:sharpfunct}
(F_2\circ F_1)^\sharp = F_1^\sharp\circ f_1^{-1}(F_2^\sharp).
\end{equation}
If, moreover, $F_1$ and $F_2$ are Lie algebroid morphisms,
then $F_2\circ F_1$ is again a Lie algebroid morphism.
\end{lemma}

\begin{proof}
For $[\wt U,\alpha]\in (f_2\circ f_1)^{-1}\Omega^k_{A_3}(U)$ one has
\[ (F_1^\sharp\circ f_1^{-1}(F_2^\sharp))([\wt U,\alpha]) = F_1^\sharp([ f_2^{-1}(\wt U),F_2^*\alpha]) = F_1^*(F_2^*\alpha) = (F_2\circ F_1)^*\alpha, \]
which proves~\eqref{eq:sharpfunct}.
If $F_1$ and $F_2$ are Lie algebroid morphisms, then
\[ \begin{split} 
(F_2\circ F_1)^\sharp\circ (f_2\circ f_1)^{-1}d_{A_3} &= F_1^\sharp\circ f_1^{-1}(F_2^\sharp\circ f_2^{-1}d_{A_3}) \\ &= F_1^\sharp\circ f_1^{-1}(d_{A_2}\circ F_2^\sharp) = d_{A_1}\circ (F_2\circ F_1)^\sharp, 
\end{split}
\]
so $F_2\circ F_1$ is a Lie algebroid morphism.
\end{proof}

\begin{corollary}
First-order Lie algebroids in the sense of Definition~\ref{def:liealgebroid},
together with Lie algebroid morphisms in the sense of Definition~\ref{def:liealgebroidmor} form a category.
\end{corollary}

\subsection*{The sheaf $\wgrisv$}

Let $\G\tto M$ be a Lie groupoid. Since
$\src\colon\G\to M$ is a submersion,
$\ker(T\src)\to\G$ is a smooth split vector bundle.
For an open set $V\subseteq\G$, we denote by
$\Omega^k_{\ker(T\src)}(V)$ the space of smooth
fiberwise continuous alternating $k$-forms on
$\ker(T\src)|_V$.
The differential $\ddRs$ is the de~Rham
differential along the source fibers.
Equivalently, it is given by the usual
Koszul formula using $\src$-vertical vector fields.
These constructions are local and remain valid
when $\G$ is non-Hausdorff.

For each open set $U\subseteq M$, define
\[
\wgrisvk(U):=
\{\omega\in\Omega^k_{\ker(T\src)}(\tgt^{-1}(U)):
\omega\text{ is right-invariant}\}.
\]
Here right-invariance means that, for every
$g\in\G$,
\[
R_g^*\bigl(
\omega|_{\tgt^{-1}(U)\cap\G_{\src(g)}}
\bigr)
=
\omega|_{\tgt^{-1}(U)\cap\G_{\tgt(g)}}.
\]
Equivalently, whenever the expressions are defined,
$(R_g^*\omega)_h(u_1,\ldots,u_k)
=\omega_{h}(u_1,\ldots,u_k).$
Right multiplication preserves the exterior
differential along source fibers. Hence
$\ddRs$ and the wedge product preserve
right-invariance, and
$\wgrisv:=\bigoplus_{k\geq0}\wgrisvk$
is a sheaf of dg-$\R$-algebras on $M$,
where the direct sum is taken in the
category of sheaves.

Define
$\Psi\colon\Omega^k_{\Lie(\G)}\to\wgrisvk$
by
\[
\Psi_U(\alpha)_h(u_1,\ldots,u_k)
:=
\alpha_{\tgt(h)}
\bigl(T_hR_{h^{-1}}(u_1),\ldots,
T_hR_{h^{-1}}(u_k)\bigr).
\]

\begin{lemma}\label{lemma:Omega-risv-isomorphism}
For each $k\geq 0$, the map $\Psi\colon \Omega^k_{\Lie(\mathcal{G})}\to \wgrisvk$ is an isomorphism of sheaves.
Moreover, the induced map $\Psi\colon \Omega_{\Lie(\mathcal{G})}^\bullet\to \wgrisv$ is an isomorphism of sheaves of dg-$\R$-algebras.
\end{lemma}

\begin{proof} 
Right multiplication and inversion are smooth and, 
for
$h\in \mathcal{G}$,
the map $T_hR_{h^{-1}} $ is an isomorphism.
Hence
$\Psi_U(\alpha)$ is a smooth $\src$-vertical $k$-form on $\tgt^{-1}(U)$.
It is right-invariant because for composable $h$ and $g$ one has 
$T_{hg}R_{(hg)^{-1}}\circ T_hR_g=T_hR_{h^{-1}}$.
Define $\Psi_U^{-1}\colon \wgrisvk(U)\to \Omega^k_{\Lie(\mathcal{G})}(U)$
by $\Psi_U^{-1}(\omega)_x(v_1,\ldots,v_k)=\omega_{\one(x)}(v_1,\ldots,v_k)$.
Smoothness follows because $\one\colon U\to \tgt^{-1}(U)$ is smooth.
Obviously, $\Psi_U^{-1}$ is the inverse of $\Psi_U$.
Hence $\Psi$ is an isomorphism of sheaves.
The compatibility with wedge products is immediate from the definition.
It remains to compare the differentials.
Let $\xi_0,\ldots,\xi_k$ be local sections of $\Lie(\mathcal{G})|_U$,
and let $\ora{\xi_i}$ be their right-invariant extensions to $\tgt^{-1}(U)$.
Then $\ora{\xi_i}$ are $\src$-vertical vector fields,
one has $[\ora{\xi_i},\ora{\xi_j}]=\ora{[\xi_i,\xi_j]}$,
and $T\tgt\circ \ora{\xi_i}=\anch(\xi_i)\circ \tgt$.
Applying the formula~\eqref{eq:dA} for $\ddRs$ to the vector fields $\ora{\xi_0},\ldots,\ora{\xi_k}$
gives exactly the formula for $d_{\Lie(\mathcal{G})}$ applied to $\xi_0,\ldots,\xi_k$.
Thus $\Psi(d_{\Lie(\mathcal{G})}\alpha)=\ddRs(\Psi(\alpha))$.
\end{proof}

\begin{theorem}
Let $(F,f)\colon \mathcal{G}\tto M\to \wt{\mathcal{G}}\tto \wt M$ be a morphism of \nnh Lie groupoids.
Then $(\Lie(F),f)$ is a first-order Lie algebroid morphism.
\end{theorem}

\begin{proof}
Since $\wt\src\circ F=f\circ \src$,
the tangent map sends $\ker(T\src)$ into $\ker(T\wt\src)$.
Hence $F$ induces pullback morphisms $F^*\colon f^{-1}\wt \tgt_*\Omega^k_{\ker(T\wt\src)}\to \tgt_*\Omega^k_{\ker(T\src)}$.
For $[V,\beta]\in f^{-1}\Omega^k_{\wt\G,\mathrm{risv} }(U)$, $g\in \tgt^{-1}(U), u_i\in \ker(T_g\src)$:
	$$(F^*[V,\beta])_g(u_1,\ldots, u_k)
	=
	\beta_{F(g)}(T_gF(u_1),\ldots, T_gF(u_k))
	$$ is well-defined since $F(\tgt^{-1}(U))
	\subseteq \wt{\tgt}^{-1}(f(U))$. 
    Moreover, by differentiating $R_{F(g)} \circ F = F\circ R_{F(g)}$, we see that $F$ preserves right-invariance.
Therefore $F^*$ restricts to a morphism $F^*\colon f^{-1}\Omega^k_{\wt\G,\mathrm{risv}}
\to \wgrisvk$.
Since $F$ maps each $\src$-fiber of $\mathcal{G}$ into the corresponding $\wt\src$-fiber of $\wt{\mathcal{G}}$,
pullback commutes with the
differential.
Indeed, for each $x\in M$, the restriction
$F_x\colon \mathcal{G}|_x\to \wt{\mathcal{G}}|_{f(x)}$
is a smooth map of source fibers. Since $\ddRs$ is the de Rham differential
along the source fibers, pullback commutes with it fiberwise. Therefore
$F^*\circ f^{-1}\ddRs=\ddRs\circ F^*$.

On the other hand, the definitions give a commutative diagram:
\[ \begin{tikzcd}
f^{-1}\Omega^k_{\Lie(\wt{\mathcal{G}})} \arrow[r, "(\Lie(F))^\sharp"] \arrow[d, swap, "f^{-1}\Psi"] & \Omega^k_{\Lie(\mathcal{G})} \arrow[d, "\Psi"] \\
f^{-1}\Omega^k_{\wt\G,\mathrm{risv}} \arrow[r, "F^*"] & \wgrisvk.
\end{tikzcd} \]
Indeed, for $[\,\wt U,\alpha\,]\in f^{-1}\Omega^k_{\Lie(\wt{\mathcal{G}})}(U)$,
$g\in \tgt^{-1}(U)$ and $u_i\in \ker(T_g\src)$,
both paths evaluate to
\[ \alpha_{f(\tgt(g))}\bigl( T_{F(g)}R_{F(g)^{-1}}(T_gF(u_1)),\ldots, T_{F(g)}R_{F(g)^{-1}}(T_gF(u_k)) \bigr). \]
Using Lemma~\ref{lemma:Omega-risv-isomorphism}, we obtain
\[
\begin{split}
  \Psi\bigl((\Lie(F))^\sharp(f^{-1}
  d_{\Lie(\wt{\mathcal{G}})}\alpha)\bigr)
  &= F^*\bigl(f^{-1}\Psi(d_{\Lie
  (\wt{\mathcal{G}})}\alpha)\bigr) 
  = F^*\bigl(f^{-1}\ddRs(\Psi(\alpha))\bigr) \\
  &= \ddRs\bigl(F^*(f^{-1}\Psi(\alpha))\bigr) 
  = \Psi\bigl(d_{\Lie(\mathcal{G})}
  ((\Lie(F))^\sharp\alpha)\bigr).
\end{split}
\]
Since $\Psi$ is an isomorphism, $(\Lie F)^\sharp\circ f^{-1}d_{\Lie(\wt{\mathcal{G}})} = d_{\Lie(\mathcal{G})}\circ (\Lie F)^\sharp$.
\end{proof}

\begin{corollary}\label{cor:Liefunctor}
$\Lie$ is a functor from the category of  \nnh  Lie groupoids
to the category of first-order Lie algebroids.
\end{corollary}

\subsection*{Anchor and bracket compatibility}
The next result shows that the dg-definition of morphism
implies the usual anchor and bracket compatibility.

\begin{theorem}\label{thm:braanchcomp}
Let $A\to M$ and $\wt A\to \wt M$ be first-order Lie algebroids,
and let $F\colon A\to \wt A$ be a first-order Lie algebroid morphism over $f\colon M\to \wt M$.
Then:
\begin{enumerate}\itemsep0em

\item \label{item:anchorcomp}
$\wt\anch\circ F = Tf\circ \anch$;

\item \label{item:bracketcomp}
if $\xi,\eta\in \shA(U)$ and $\wt\xi,\wt\eta\in \wt{\shA}(\wt U)$
satisfy $f(U)\subseteq \wt U$, $F\circ \xi=\wt\xi\circ f$, and
$F\circ \eta=\wt\eta\circ f$ on $U$,
then $F\circ [\xi,\eta]=[\wt\xi,\wt\eta]\circ f$ on $U$.
\end{enumerate}
Furthermore, if $f$ is a local diffeomorphism, then the converse is true.
\end{theorem}
	
\begin{proof}
We first prove~\ref{item:anchorcomp}.
Let $x\in M$ and $a\in A_x$.
Take an open neighborhood $U$ of $x$ and let $[\,\wt U,g\,]\in f^{-1}\shcinf_{\wt M}(U)$.
Since $F$ is a Lie algebroid morphism,
the degree-$0$ part of~\eqref{eq:dgLie} gives $F^\sharp(f^{-1}d_{\wt A}[\,\wt U,g\,])=d_A(f^\sharp[\,\wt U,g\,])$.
Evaluating at $x$ on $a$ gives
\[ (dg)_{f(x)}(\wt\anch(F(a))) = (d(g\circ f))_x(\anch(a)) = (dg)_{f(x)}(T_xf(\anch(a))). \]
This holds for every local smooth function $g$ near $f(x)$.
Choose a chart $\phi$ about $f(x)$ with values in a locally convex space $\sF$,
and take $g=\lambda\circ \phi$ for $\lambda\in \sF'$.
By the Hahn-Banach theorem,
it follows that $\wt\anch(F(a))=T_xf(\anch(a))$.
This proves~\ref{item:anchorcomp}.
Now assume that $\xi,\eta,\wt\xi,\wt\eta$ are as in~\ref{item:bracketcomp}.
Let $\alpha\in \Omega^1_{\wt A}(\wt U)$.
Applying~\eqref{eq:dgLie} to $\alpha$ and evaluating on $\xi,\eta$ gives
\[ (F^\sharp(d_{\wt A}\alpha))(\xi,\eta) = d_A(F^\sharp\alpha)(\xi,\eta). \]
Expanding both sides by \eqref{eq:dA}
and using the relations $F\circ \xi=\wt\xi\circ f$, $F\circ \eta=\wt\eta\circ f$,
together with anchor compatibility,
the anchor terms cancel and we obtain
\[ \alpha([\wt\xi,\wt\eta])\circ f = \alpha(F([\xi,\eta])) \quad\text{on }U. \]
Fix $x\in U$.
Choose a local trivialization $\wt\phi\colon \wt A|_W\to W\times \sF_{\wt\phi}$ about $f(x)$.
For each $\lambda\in \sF_{\wt\phi}'$,
the fiberwise linear form $\alpha_\lambda(y)(b):=\lambda(\pr_2(\wt\phi(b)))$
defines a local section of $\Omega^1_{\wt A}(W)$.
Applying the previous identity to $\alpha_\lambda$
and using the Hahn-Banach theorem,
we get $F([\xi,\eta](x))=[\wt\xi,\wt\eta](f(x))$.
Since $x$ was arbitrary,~\ref{item:bracketcomp} follows.

Assume now that $f$ is a local diffeomorphism, and that \ref{item:anchorcomp} and \ref{item:bracketcomp} hold.
We show that $F$ is a first-order Lie algebroid morphism. The assertion is local on $M$.
Let $U\subseteq M$ be open such that $f|_U\colon U\to \widetilde U:=f(U)$ is a
diffeomorphism. Let $[\widetilde U,\alpha]\in f^{-1}\Omega_{\widetilde A}^k(U)$.
Then $F^\sharp[\widetilde U,\alpha]=F^*\alpha$ on $U$. We show that
$F^*(d_{\widetilde A}\alpha)=d_A(F^*\alpha)$ on $U$.

Let $\xi_0,\ldots,\xi_k\in\Gamma(A|_U)$, and define
$\widetilde\xi_i\in\Gamma(\widetilde A|_{\widetilde U})$ by
$\widetilde\xi_i:=F\circ\xi_i\circ (f|_U)^{-1}$. Then
$F\circ\xi_i=\widetilde\xi_i\circ f$ on $U$. Using~\eqref{eq:dA}, we get
\[
\begin{aligned}
(F^*d_{\widetilde A}\alpha)(\xi_0,\ldots,\xi_k)
&=(d_{\widetilde A}\alpha)(\widetilde\xi_0,\ldots,\widetilde\xi_k)\circ f\\
&=\sum_{j=0}^k(-1)^j\,
\widetilde\anch(\widetilde\xi_j).\bigl(\alpha(\widetilde\xi_0,\ldots,
\widehat{\widetilde\xi_j},\ldots,\widetilde\xi_k)\bigr)\circ f\\
&\quad+\sum_{i<j}(-1)^{i+j}\,
\alpha([\widetilde\xi_i,\widetilde\xi_j],\widetilde\xi_0,\ldots,
\widehat{\widetilde\xi_i},\ldots,\widehat{\widetilde\xi_j},\ldots,
\widetilde\xi_k)\circ f,
\end{aligned}
\]
while
\[
\begin{aligned}
d_A(F^*\alpha)(\xi_0,\ldots,\xi_k)
&=\sum_{j=0}^k(-1)^j\,
\anch(\xi_j).\bigl((F^*\alpha)(\xi_0,\ldots,\widehat{\xi_j},\ldots,\xi_k)\bigr)\\
&\quad+\sum_{i<j}(-1)^{i+j}\,
(F^*\alpha)([\xi_i,\xi_j],\xi_0,\ldots,\widehat{\xi_i},\ldots,
\widehat{\xi_j},\ldots,\xi_k).
\end{aligned}
\]
For each $j$,
\[
(F^*\alpha)(\xi_0,\ldots,\widehat{\xi_j},\ldots,\xi_k)
=
\alpha(\widetilde\xi_0,\ldots,\widehat{\widetilde\xi_j},\ldots,
\widetilde\xi_k)\circ f.
\]
Thus anchor compatibility gives equality of the first sums. Bracket compatibility
gives $F\circ[\xi_i,\xi_j]=[\widetilde\xi_i,\widetilde\xi_j]\circ f$, hence equality
of the second sums. Therefore $F^*(d_{\widetilde A}\alpha)=d_A(F^*\alpha)$ on
$U$. Thus, $F^\sharp$ satisfies~\eqref{eq:dgLie}, and $F$ is a first-order Lie algebroid
morphism.
\end{proof}


The following lemma provides a local characterization of Lie algebroid morphisms. Below~\eqref{eq:local-form-morphism}
is stronger than~\ref{item:bracketcomp} for arbitrary base maps.
\begin{lemma}\label{lem:local-form-morphism}
Let $F\colon A\to \wt A$ be a smooth vector bundle morphism over
$f\colon M\to \wt M$.
Let $\phi\colon A|_U\to U\times \sF_\phi$ and
$\wt\phi\colon \wt A|_{\wt U}\to \wt U\times \sF_{\wt\phi}$ be local
trivializations with $f(U)\subseteq \wt U$.
Let $C_\phi$ and $C_{\wt\phi}$ be the corresponding local bracket terms, and
let $\wh F_{\phi\wt\phi}\colon U\times \sF_\phi\to \sF_{\wt\phi}$ be the local
representative of $F$.
Then $F$ is a first-order Lie algebroid morphism if and only if, for all $x\in U$ and $u,v\in \sF_\phi$, the following identities hold in every such pair of local trivializations:
\[
\wt\anch^{\wt\phi}\bigl(f(x),\wh F_{\phi\wt\phi}(x,u)\bigr)
=
df(x;\anch^\phi(x,u))
\]
and
\begin{equation}\label{eq:local-form-morphism}
\begin{aligned}
\wh F_{\phi\wt\phi}\bigl(x,C_\phi(x,u,v)\bigr)
&=
d_1\wh F_{\phi\wt\phi}\bigl(x,v;\anch^\phi(x,u)\bigr)
-
d_1\wh F_{\phi\wt\phi}\bigl(x,u;\anch^\phi(x,v)\bigr) \\
&\quad +
C_{\wt\phi}\bigl(f(x),
\wh F_{\phi\wt\phi}(x,u),\wh F_{\phi\wt\phi}(x,v)\bigr).
\end{aligned}
\end{equation}
\end{lemma}

\begin{proof}
Suppose first that $F$ is a first-order Lie algebroid morphism.
The anchor identity follows from
Theorem~\ref{thm:braanchcomp}\,\ref{item:anchorcomp}.
Fix $x\in U$, $u,v\in\sF_\phi$, and
$\lambda\in\sF_{\wt\phi}'$.
Let $\alpha_\lambda\in\Omega^1_{\wt A}(\wt U)$
be the constant local one-form defined by
$\alpha_\lambda(y)(b)
:=\lambda(\pr_2(\wt\phi(b)))$.
Evaluating
$d_A(F^\sharp\alpha_\lambda)
=F^\sharp(d_{\wt A}\alpha_\lambda)$
at $x$ on the constant local sections associated with
$u$ and $v$, we obtain
$\lambda\bigl(
d_1\wh F_{\phi\wt\phi}(x,v;\anch^\phi(x,u))
-d_1\wh F_{\phi\wt\phi}(x,u;\anch^\phi(x,v))
-\wh F_{\phi\wt\phi}(x,C_\phi(x,u,v))
+C_{\wt\phi}(f(x),
\wh F_{\phi\wt\phi}(x,u),
\wh F_{\phi\wt\phi}(x,v))
\bigr)=0$.
Since $\lambda$ is arbitrary, the Hahn-Banach theorem
gives \eqref{eq:local-form-morphism}.

Conversely, suppose that both local identities hold. The degree-zero part of \eqref{eq:dgLie} follows from anchor compatibility. For an arbitrary local $k$-form $\alpha$ with $k\geq1$, apply the local formula for $d_A$ from \eqref{eq:dAnablalocal} (with the trivial coefficient connection) to the pullback $F^*\alpha$. The chain rule produces derivatives of $f$ and derivatives $d_1\wh F_{\phi\wt\phi}$ in the fiber arguments. The anchor identity identifies the derivatives of $\alpha$ in the base variable, while \eqref{eq:local-form-morphism} combines each pair of fiber-derivative terms with the corresponding $C_\phi$ term into $C_{\wt\phi}$. The result is exactly $F^*(d_{\wt A}\alpha)$. Thus \eqref{eq:dgLie} holds in every degree. No finite generation of the algebra of forms is used.
\end{proof}

\subsection*{Pullback of representations}

We now show that Lie algebroid morphisms are exactly the maps along which
representations pull back.

Let $F\colon A\to \wt A$ be a smooth vector bundle morphism over
$f\colon M\to \wt M$, and let $(\wt E,\wt\nabla)$ be an $\wt A$-connection on
a vector bundle $\wt q\colon \wt E\to \wt M$.
For each $k\ge 0$, define
\[
F^\sharp_{\wt E,k}\colon f^{-1}\Omega^k_{\wt A}(-,\wt E)\to
\Omega_A^k(-,f^*\wt E)
\]
as follows.
Let $U\subseteq M$ be open and let
$[\,\wt U,\alpha\,]\in f^{-1}\Omega^k_{\wt A}(-,\wt E)(U)$.
Set
\[
F^\sharp_{\wt E,k,U}([\,\wt U,\alpha\,])_x(a_1,\ldots,a_k)
:=
\bigl(x,\alpha_{f(x)}(F(a_1),\ldots,F(a_k))\bigr)
\in (f^*\wt E)_x.
\]
For $k=0$, this is the usual pullback of local sections of $\wt E$.
We write $F^\sharp_{\wt E}:=\bigoplus_{k\ge 0}F^\sharp_{\wt E,k}$.

\begin{lemma}\label{lem:pullback-E-valued-forms}
For each $k\ge 0$, the maps $F^\sharp_{\wt E,k,U}$ are well-defined and define
a morphism of sheaves of $\R$-modules
\[
F^\sharp_{\wt E,k}\colon f^{-1}\Omega^k_{\wt A}(-,\wt E)\to
\Omega_A^k(-,f^*\wt E).
\]
Hence $F^\sharp_{\wt E}$ is a morphism of graded sheaves of $\R$-modules
\[
F^\sharp_{\wt E}\colon f^{-1}\Omega^\bullet_{\wt A}(-,\wt E)\to
\Omega^\bullet_A(-,f^*\wt E).
\]
Moreover, for every open set $U\subseteq M$, every
$\omega\in f^{-1}\Omega^k_{\wt A}(U)$, and every
$\beta\in f^{-1}\Omega^l_{\wt A}(U,\wt E)$, one has
\[
F^\sharp_{\wt E}(\omega\wedge \beta)
=
F^\sharp(\omega)\wedge F^\sharp_{\wt E}(\beta).
\]
\end{lemma}

\begin{proof}
The proof that $F^\sharp_{\wt E}$ is a well-defined morphism 
of graded sheaves is similar to the scalar case in Lemma~\ref{lemma:graded-commutative-morphism}.
Choose local trivializations
$\phi\colon A|_U\to U\times \sF_\phi$,
$\wt\phi\colon \wt A|_{\wt U}\to \wt U\times \sF_{\wt\phi}$, and
$\wt\psi\colon \wt E|_{\wt U}\to \wt U\times \sF_{\wt\psi}$ with
$f(U)\subseteq \wt U$.
Let $\wh F_{\phi\wt\phi}\colon U\times \sF_\phi\to \sF_{\wt\phi}$ be the local
representative of $F$, so that
$\wt\phi(F(\phi^{-1}(x,u)))=(f(x),\wh F_{\phi\wt\phi}(x,u))$.
If $\alpha_{\wt\phi\wt\psi}$ is the local representative of $\alpha$, then the
local representative of $F^\sharp_{\wt E,k}(\alpha)$ is
\[
(F^\sharp_{\wt E,k}\alpha)_{\phi,f^*\wt\psi}(x,u_1,\ldots,u_k)
=
\alpha_{\wt\phi\wt\psi}\bigl(f(x),
\wh F_{\phi\wt\phi}(x,u_1),\ldots,\wh F_{\phi\wt\phi}(x,u_k)\bigr).
\]
Hence $F^\sharp_{\wt E,k}(\alpha)$ is smooth.
The wedge identity is checked pointwise.
We may restrict to an open set $V\subseteq U$ on which $\omega$
and $\beta$ admit representative
 over some open
$\wt W\supseteq f(V)$, write
$\omega|_V=[\,\wt W,\wt\omega\,]$ and $\beta|_V=[\,\wt W,\wt\beta\,]$.
Then
\[
F^\sharp_{\wt E}(\omega\wedge \beta)
=
F^*(\wt\omega\wedge \wt\beta)
=
F^*\wt\omega\wedge F^*\wt\beta
=
F^\sharp(\omega)\wedge F^\sharp_{\wt E}(\beta).
\]
\end{proof}

Let $\phi\colon A|_U\to U\times \sF_\phi$,
$\wt\phi\colon \wt A|_{\wt U}\to \wt U\times \sF_{\wt\phi}$, and
$\wt\psi\colon \wt E|_{\wt U}\to \wt U\times \sF_{\wt\psi}$ be local
trivializations with $f(U)\subseteq \wt U$.
Let $\wh F_{\phi\wt\phi}\colon U\times \sF_\phi\to \sF_{\wt\phi}$ be the local
representative of $F$.
If $\wt\nabla$ has local connection term
$\wt B_{\wt\phi\wt\psi}\colon \wt U\times \sF_{\wt\phi}\times \sF_{\wt\psi}\to \sF_{\wt\psi}$,
set
\[
B^F_{\phi,f^*\wt\psi}(x,u,e)
:=
\wt B_{\wt\phi\wt\psi}\bigl(f(x),\wh F_{\phi\wt\phi}(x,u),e\bigr).
\]

\begin{lemma}\label{lem:pullback-connection}
Assume that $F\colon A\to \wt A$ is a smooth vector bundle morphism over
$f\colon M\to \wt M$ and that $\wt\anch\circ F=Tf\circ \anch$.
Let $(\wt E,\wt\nabla)$ be an $\wt A$-connection on $\wt E$.
Then there exists a unique morphism of sheaves of $\R$-modules
\[
\nabla^F\colon \shA\otimes_\R \Gamma(f^*\wt E|-)\to \Gamma(f^*\wt E|-)
\]
such that, for every open set $U\subseteq M$, every $\xi\in \shA(U)$, every
$\sigma\in \Gamma(f^*\wt E|_U)$, and every choice of local trivializations as
above, one has
\begin{equation}\label{eq:pullback-connection-local}
(\nabla^F_\xi\sigma)^{f^*\wt\psi}(x)
=
\anch(\xi).\sigma^{f^*\wt\psi}(x)
+
B^F_{\phi,f^*\wt\psi}(x,\xi^\phi(x),\sigma^{f^*\wt\psi}(x)).
\end{equation}
Moreover, $\nabla^F$ is an $A$-connection on $f^*\wt E$.
\end{lemma}

\begin{proof}
We first check that the local formula is compatible with change of
trivializations.
Let $\phi'$ be another local trivialization of $A$ over $U$, and let
$\wt\phi'$ and $\wt\psi'$ be local trivializations of $\wt A$ and $\wt E$
over $\wt U$.
Write
$
(\phi'\circ \phi^{-1})(x,u)=(x,\wh K(x,u)),
$ $
(\wt\phi'\circ \wt\phi^{-1})(y,\wt u)=(y,\wh H(y,\wt u)),
$ and $
(\wt\psi'\circ \wt\psi^{-1})(y,e)=(y,\wh G(y,e)),
$
so that $\wh K$, $\wh H$, and $\wh G$ are smooth and fiberwise linear, and
\[
\xi^{\phi'}(x)=\wh K(x,\xi^\phi(x)),
\qquad
\sigma^{f^*\wt\psi'}(x)=\wh G(f(x),\sigma^{f^*\wt\psi}(x)).
\]
Also,
$
\wh F_{\phi'\wt\phi'}(x,\wh K(x,u))
=
\wh H(f(x),\wh F_{\phi\wt\phi}(x,u)).
$
By the coordinate change formula for connection terms, we have
\[
\begin{aligned}
\wt B_{\wt\phi'\wt\psi'}\bigl(y,\wh H(y,\wt u),\wh G(y,e)\bigr) 
 =
\wh G\bigl(y,\wt B_{\wt\phi\wt\psi}(y,\wt u,e)\bigr)
-
d_1\wh G\bigl(y,e;\wt\anch^{\wt\phi}(y,\wt u)\bigr).
\end{aligned}
\]
Now let $\sigma^\psi:=\sigma^{f^*\wt\psi}$ and
$\sigma^{\psi'}:=\sigma^{f^*\wt\psi'}$.
Since $\sigma^{\psi'}(x)=\wh G(f(x),\sigma^\psi(x))$, the chain rule gives
\[
\anch(\xi).(\sigma^{\psi'})(x)
=
d_1\wh G\bigl(f(x),\sigma^\psi(x);df(x;\anch^\phi(x,\xi^\phi(x)))\bigr)
+
\wh G\bigl(f(x),\anch(\xi).\sigma^\psi(x)\bigr).
\]
Because $\wt\anch\circ F=Tf\circ \anch$, we have
\[
df(x;\anch^\phi(x,\xi^\phi(x)))
=
\wt\anch^{\wt\phi}\bigl(f(x),\wh F_{\phi\wt\phi}(x,\xi^\phi(x))\bigr).
\]
Using these identities, we obtain
\[
\begin{aligned}
&\anch(\xi).(\sigma^{\psi'})(x)
+
B^F_{\phi',f^*\wt\psi'}(x,\xi^{\phi'}(x),\sigma^{\psi'}(x)) 
\\ &\qquad 
=
\wh G\Bigl(f(x),
\anch(\xi).\sigma^\psi(x)
+
B^F_{\phi,f^*\wt\psi}(x,\xi^\phi(x),\sigma^\psi(x))\Bigr).
\end{aligned}
\]
This is exactly the correct transition law for a section of $f^*\wt E$.
Hence the local formula defines, for each open set $U$ and each pair
$(\xi,\sigma)$, a unique section
$\nabla^F_U(\xi,\sigma)\in \Gamma(f^*\wt E|_U)$.

Compatibility with restrictions is immediate from the same formula.
Thus the family $\nabla^F_U$ defines a morphism of sheaves
$
\nabla^F\colon \shA\otimes_\R \Gamma(f^*\wt E|-)\to \Gamma(f^*\wt E|-).
$
Uniqueness is clear, since the local representative is prescribed by
\eqref{eq:pullback-connection-local}.

We now check the axioms.
The formula is $\R$-linear in $\sigma$ and $\shcinf_M$-linear in $\xi$.
Let $h\in \shcinf_M(U)$.
Since $B^F_{\phi,f^*\wt\psi}(x,u,\cdot)$ is linear, we get
\[
\begin{aligned}
(\nabla^F_\xi(h\sigma))^{f^*\wt\psi}(x)
&=
\anch(\xi).(h\,\sigma^{f^*\wt\psi})(x)
+
B^F_{\phi,f^*\wt\psi}(x,\xi^\phi(x),h(x)\sigma^{f^*\wt\psi}(x)) \\
&=
h(x)(\nabla^F_\xi\sigma)^{f^*\wt\psi}(x)
+
(\anch(\xi).h)(x)\,\sigma^{f^*\wt\psi}(x).
\end{aligned}
\]
Hence $\nabla^F_\xi(h\sigma)=h\nabla^F_\xi\sigma+(\anch(\xi).h)\sigma$.
Thus $\nabla^F$ is an $A$-connection on $f^*\wt E$.
\end{proof}

The next lemma is the local form of the morphism condition.

\begin{lemma}\label{lem:pullback-curvature}
Let $F\colon A\to \wt A$ be a first-order Lie algebroid morphism over $f\colon M\to \wt M$,
and let $(\wt E,\wt\nabla)$ be an $\wt A$-connection.
Let $\nabla^F$ be the induced $A$-connection on $f^*\wt E$.
Choose local trivializations
$\phi\colon A|_U\to U\times \sF_\phi$,
$\wt\phi\colon \wt A|_{\wt U}\to \wt U\times \sF_{\wt\phi}$, and
$\wt\psi\colon \wt E|_{\wt U}\to \wt U\times \sF_{\wt\psi}$ with
$f(U)\subseteq \wt U$.
Then, for all $x\in U$, $u,v\in \sF_\phi$, and $e\in \sF_{\wt\psi}$,
\begin{equation}\label{eq:curv-pullback}
R_{\nabla^F,\phi,f^*\wt\psi}(x,u,v,e)
=
R_{\wt\nabla,\wt\phi\wt\psi}\bigl(f(x),
\wh F_{\phi\wt\phi}(x,u),\wh F_{\phi\wt\phi}(x,v),e\bigr).
\end{equation}
In particular, if $\wt\nabla$ is flat, then $\nabla^F$ is flat.
\end{lemma}

\begin{proof}
Write $\wt B:=\wt B_{\wt\phi\wt\psi}$ and
$B^F:=B^F_{\phi,f^*\wt\psi}$.
Then
$B^F(x,u,e)=\wt B(f(x),\wh F_{\phi\wt\phi}(x,u),e)$.
Let $C_\phi$ and $C_{\wt\phi}$ be the local bracket terms of $A$ and $\wt A$.
Fix $u,v\in \sF_\phi$.
Let $\xi_u,\xi_v$ be the constant local sections of $A|_U$ 
defined by $\xi_u(y) = \phi^{-1}(y,u)$
and $\xi_v(y)=\phi^{-1}(y,v)$.
If $\sigma$ is a local section of $f^*\wt E|_U$ and
$\tau:=\sigma^{f^*\wt\psi}$, then
$(\nabla^F_{\xi_u}\sigma)^{f^*\wt\psi}(x)=
d\tau(x;\anch^\phi(x,u))+B^F(x,u,\tau(x))$.
A direct expansion of
$R_{\nabla^F}(\xi_u,\xi_v)\sigma=
\nabla^F_{\xi_u}\nabla^F_{\xi_v}\sigma-
\nabla^F_{\xi_v}\nabla^F_{\xi_u}\sigma-
\nabla^F_{[\xi_u,\xi_v]}\sigma$
gives
\[
\begin{aligned}
R_{\nabla^F,\phi,f^*\wt\psi}(x,u,v,e)
&=
d_1B^F(x,v,e;\anch^\phi(x,u))
-
d_1B^F(x,u,e;\anch^\phi(x,v)) \\
&\quad +
B^F(x,u,B^F(x,v,e))
-
B^F(x,v,B^F(x,u,e)) \\
&\quad -
B^F(x,C_\phi(x,u,v),e).
\end{aligned}
\]
Set
$\wt u:=\wh F_{\phi\wt\phi}(x,u)$ and
$\wt v:=\wh F_{\phi\wt\phi}(x,v)$.
By the chain rule,
\[
\begin{aligned}
d_1B^F(x,u,e;\anch^\phi(x,v))
&=
d_1\wt B\bigl(f(x),\wt u,e;df(x;\anch^\phi(x,v))\bigr) \\
&\quad +
\wt B\bigl(f(x),d_1\wh F_{\phi\wt\phi}(x,u;\anch^\phi(x,v)),e\bigr)
\end{aligned}
\]
and similarly for $d_1B^F(x,v,e;\anch^\phi(x,u))$.
Using anchor compatibility and \eqref{eq:local-form-morphism}, we obtain
$
df(x;\anch^\phi(x,u))=\wt\anch^{\wt\phi}(f(x),\wt u)$ (similarly for $v$ instead of $u$)
and
\[
\begin{aligned}
\wh F_{\phi\wt\phi}\bigl(x,C_\phi(x,u,v)\bigr)
&=
d_1\wh F_{\phi\wt\phi}\bigl(x,v;\anch^\phi(x,u)\bigr)
-
d_1\wh F_{\phi\wt\phi}\bigl(x,u;\anch^\phi(x,v)\bigr) \\
&\quad +
C_{\wt\phi}(f(x),\wt u,\wt v).
\end{aligned}
\]
Substituting these identities into the formula for
$R_{\nabla^F,\phi,f^*\wt\psi}(x,u,v,e)$, the terms containing
$d_1\wh F_{\phi\wt\phi}$ cancel.
The remaining terms are 
$
R_{\wt\nabla,\wt\phi\wt\psi}\bigl(f(x),\wt u,\wt v,e\bigr)
$ exactly.
This proves~\eqref{eq:curv-pullback}.
If $\wt\nabla$ is flat, then the right-hand side vanishes in every
trivialization.
Hence $\nabla^F$ is flat.
\end{proof}

\begin{theorem}\label{thm:pullback-representation}
Let $F\colon A\to \wt A$ be a first-order Lie algebroid morphism over
$f\colon M\to \wt M$, and let $(\wt E,\wt\nabla)$ be an $\wt A$-connection.
Let $\nabla^F$ be the induced $A$-connection on $f^*\wt E$.
Then
\[
F^\sharp_{\wt E}\circ f^{-1}d_{\wt A,\wt\nabla}
=
d_{A,\nabla^F}\circ F^\sharp_{\wt E}
\]
as morphisms of degree $+1$ from
$f^{-1}\Omega^\bullet_{\wt A}(-,\wt E)$ to $\Omega^\bullet_A(-,f^*\wt E)$.
If, moreover, $(\wt E,\wt\nabla)$ is a representation of $\wt A$, then
$(f^*\wt E,\nabla^F)$ is a representation of $A$.
\end{theorem}

\begin{proof}
Fix $k\ge 0$, choose local trivializations
$\phi\colon A|_U\to U\times \sF_\phi$,
$\wt\phi\colon \wt A|_{\wt U}\to \wt U\times \sF_{\wt\phi}$, and
$\wt\psi\colon \wt E|_{\wt U}\to \wt U\times \sF_{\wt\psi}$ with
$f(U)\subseteq \wt U$, and 
let $[\,\tilde U,\alpha\,]\in f^{-1}\Omega^k_{\tilde A}(-,\tilde E)(U)$.
Set $\beta:=F^\sharp_{\wt E}([\wt U,\alpha])\in \Omega_A^k(U,f^*\wt E)$.
The local representative of $\beta$ is
\[
\beta_{\phi,f^*\wt\psi}(x,u_1,\ldots,u_k)
=
\alpha_{\wt\phi\wt\psi}\bigl(f(x),
\wh F_{\phi\wt\phi}(x,u_1),\ldots,\wh F_{\phi\wt\phi}(x,u_k)\bigr).
\]
Fix $x\in U$ and $u_0,\ldots,u_k\in \sF_\phi$, and set
$\wt u_i:=\wh F_{\phi\wt\phi}(x,u_i)$.
Using the local formula for $d_{A,\nabla^F}$ with constant local sections, we get
\[
\begin{aligned}
&(d_{A,\nabla^F}\beta)_{\phi,f^*\wt\psi}(x,u_0,\ldots,u_k) \\
&=
\sum_{i=0}^k (-1)^i
\Bigl(
d_1\beta_{\phi,f^*\wt\psi}(x,u_0,\ldots,\wh u_i,\ldots,u_k;
\anch^\phi(x,u_i))\Bigr) \\ &\quad
+   \sum_{i=0}^k (-1)^i
\Bigl(
B^F_{\phi,f^*\wt\psi}\bigl(x,u_i,
\beta_{\phi,f^*\wt\psi}(x,u_0,\ldots,\wh u_i,\ldots,u_k)\bigr)
\Bigr) \\
&\quad +
\sum_{0\le i<j\le k} (-1)^{i+j}
\beta_{\phi,f^*\wt\psi}\bigl(x,C_\phi(x,u_i,u_j),
u_0,\ldots,\wh u_i,\ldots,\wh u_j,\ldots,u_k\bigr).
\end{aligned}
\]
The chain rule gives
\[
\begin{aligned}
&d_1\beta_{\phi,f^*\wt\psi}(x,u_0,\ldots,\wh u_i,\ldots,u_k;
\anch^\phi(x,u_i)) \\
&\qquad =
d_1\alpha_{\wt\phi\wt\psi}(f(x),
\wt u_0,\ldots,\wh{\wt u_i},\ldots,\wt u_k;
df(x;\anch^\phi(x,u_i))) \\
&\qquad\quad +
\sum_{j\neq i}
\alpha_{\wt\phi\wt\psi}\bigl(f(x),
\wt u_0,\ldots,d_1\wh F_{\phi\wt\phi}(x,u_j;\anch^\phi(x,u_i)),
\ldots,\wh{\wt u_i},\ldots,\wt u_k\bigr).
\end{aligned}
\]
Since
$B^F_{\phi,f^*\wt\psi}(x,u_i,e)=
\wt B_{\wt\phi\wt\psi}(f(x),\wt u_i,e)$
and
$df(x;\anch^\phi(x,u_i))=\wt\anch^{\wt\phi}(f(x),\wt u_i)$,
the first sum becomes the corresponding first sum for
$d_{\wt A,\wt\nabla}\alpha$, plus the extra terms containing derivatives of
$\wh F_{\phi\wt\phi}$.

We now combine these extra terms with the bracket term.
For each pair $i<j$, the contribution of the derivatives of
$\wh F_{\phi\wt\phi}$ in the slots $i$ and $j$, together with the bracket term,
is
\[
\begin{aligned}
(-1)^{i+j}\,
\alpha_{\wt\phi\wt\psi}\Bigl(
f(x),&
\wh F_{\phi\wt\phi}(x,C_\phi(x,u_i,u_j)) 
-
d_1\wh F_{\phi\wt\phi}(x,u_j;\anch^\phi(x,u_i))
\\ &
+
d_1\wh F_{\phi\wt\phi}(x,u_i;\anch^\phi(x,u_j)),
\wt u_0,\ldots,\wh{\wt u_i},\ldots,\wh{\wt u_j},\ldots,\wt u_k
\Bigr).
\end{aligned}
\]
By \eqref{eq:local-form-morphism}, this equals
\[
(-1)^{i+j}\,
\alpha_{\wt\phi\wt\psi}\Bigl(
f(x),
C_{\wt\phi}(f(x),\wt u_i,\wt u_j),
\wt u_0,\ldots,\wh{\wt u_i},\ldots,\wh{\wt u_j},\ldots,\wt u_k
\Bigr).
\]
Hence the whole expression is exactly
$
(d_{\wt A,\wt\nabla}\alpha)_{\wt\phi\wt\psi}(f(x),\wt u_0,\ldots,\wt u_k),
$
that is,
$
d_{A,\nabla^F}\beta
=
F^\sharp_{\wt E}(d_{\wt A,\wt\nabla}\alpha)
$
in local coordinates.
Therefore
$F^\sharp_{\wt E}\circ f^{-1}d_{\wt A,\wt\nabla}
=
d_{A,\nabla^F}\circ F^\sharp_{\wt E}$.

If $(\wt E,\wt\nabla)$ is flat, then $\nabla^F$ is flat by
Lemma~\ref{lem:pullback-curvature}.
Hence $(f^*\wt E,\nabla^F)$ is a representation of $A$.
\end{proof}

\begin{corollary}\label{cor:pullback-representations-characterization}
For a smooth vector bundle morphism $F\colon A\to \wt A$ over
$f\colon M\to \wt M$, the following are equivalent:
\begin{enumerate}\itemsep0em
\item \label{item:repres-pullback-cor1} $F$ is a first-order Lie algebroid morphism;
\item \label{item:repres-pullback-cor2} for every representation
$(\wt E,\wt\nabla)$ of $\wt A$, the local formula
\eqref{eq:pullback-connection-local} defines a flat $A$-connection
$\nabla^F$ on $f^*\wt E$, and
\[
F^\sharp_{\wt E}\circ f^{-1}d_{\wt A,\wt\nabla}
=
d_{A,\nabla^F}\circ F^\sharp_{\wt E}.
\]
\end{enumerate}
\end{corollary}

\begin{proof}
The implication
\ref{item:repres-pullback-cor1} $\Rightarrow$
\ref{item:repres-pullback-cor2}
is Theorem~\ref{thm:pullback-representation}.

For
\ref{item:repres-pullback-cor2} $\Rightarrow$
\ref{item:repres-pullback-cor1},
take $\wt E=\wt M\times \R$ with the trivial flat $\wt A$-connection.
Then
$\Omega^\bullet_{\wt A}(-,\wt E)=\Omega^\bullet_{\wt A}$,
$\Omega^\bullet_A(-,f^*\wt E)=\Omega^\bullet_A$,
$d_{\wt A,\wt\nabla}=d_{\wt A}$, and
$d_{A,\nabla^F}=d_A$.
Hence
\[
F^\sharp\circ f^{-1}d_{\wt A}=d_A\circ F^\sharp,
\]
which is exactly Definition~\ref{def:liealgebroidmor}.
\end{proof}

\begin{proposition}
\label{prop:cohomology_functoriality}
Let $F\colon A\to \wt A$ be a first-order Lie algebroid morphism over $f\colon M\to \wt M$, and let $(\wt E,\wt\nabla)$ be a representation of $\wt A$.
Then the induced representation $(f^*\wt E,\nabla^F)$ of $A$ gives a morphism of complexes of sheaves
\[
F^\sharp_{\wt E}\colon f^{-1}\Omega^\bullet_{\wt A}(-,\wt E)\to \Omega^\bullet_A(-,f^*\wt E).
\]
Consequently, it induces a morphism of cohomology sheaves
\[
F^*\colon f^{-1}\shH^k_{\wt A}(-,\wt E)\to \shH^k_A(-,f^*\wt E).
\]
Moreover, these morphisms are functorial. More precisely, if $F\colon A_1\to A_2$ and $G\colon A_2\to A_3$
are first-order Lie algebroid morphisms over $f\colon M_1\to M_2$
and $g\colon M_2\to M_3$, respectively,
and if $(E,\nabla)$ is a representation of $A_3$, then, under the canonical identifications
$f^{-1}g^{-1}\cong (g\circ f)^{-1}$ and $f^*g^*E\cong (g\circ f)^*E$, one has
\[
(G\circ F)^*=F^*\circ f^{-1}(G^*).
\]
The identity morphism induces the identity on cohomology sheaves. In addition, pullback of global forms induces a map $F^*\colon H^k_{\wt A}(\wt M,\wt E)\to H^k_A(M,f^*\wt E)$.
\end{proposition}

\begin{proof}
By Theorem~\ref{thm:pullback-representation}, $F^\sharp_{\wt E}$ is a cochain map. Since the inverse-image functor for sheaves of $\R$-modules is exact, $\shH^k(f^{-1}\Omega^\bullet_{\wt A}(-,\wt E))\cong f^{-1}\shH^k_{\wt A}(-,\wt E)$, giving the claimed morphism. For composable morphisms $F$ and $G$, fiberwise evaluation yields $(G\circ F)^\sharp_E=F^\sharp_{g^*E}\circ f^{-1}G^\sharp_E$ under the canonical identification $f^*g^*E\cong(g\circ f)^*E$. Their induced connections also agree: with local representatives $\wh F,\wh G$, both have coefficient $B(g(f(x)),\wh G(f(x),\wh F(x,u)),e)$. Therefore their cochain maps and the induced maps on cohomology sheaves agree. The identity case is immediate.
\end{proof}

\begin{example}\label{ex:action_functoriality}
Let $A=M\times\fg$ and $\wt A=\wt M\times\wt\fg$ be action Lie algebroids, and let $F(x,u)=(f(x),\phi(u))$, where $\phi\colon\fg\to\wt\fg$ is a smooth Lie algebra homomorphism and $f\colon M\to\wt M$ is smooth. Assume that $F$ is a Lie algebroid morphism, so $df(x;\anch(x,u))=\wt\anch(f(x),\phi(u))$. Let $\wt E=\wt M\times\sE$ carry the representation of $\wt A$ induced by a smooth representation $\pi\colon\wt\fg\times\sE\to\sE$. Under $f^*\wt E\cong M\times\sE$, the induced connection is
\[
(\nabla^F_\xi\sigma)(x)=d\sigma(x;\anch(x,\xi(x)))+\pi(\phi(\xi(x)))\sigma(x).
\]
It is flat by Lemma~\ref{lem:pullback-curvature}. If $\alpha\in\Omega^k_{\wt A}(\wt M,\wt E)$, its pullback is
\[
(F^\sharp_{\wt E}\alpha)(x;u_1,\ldots,u_k)=\alpha(f(x);\phi(u_1),\ldots,\phi(u_k)).
\]
The cochain identity follows from the chain rule, anchor compatibility, and $\phi([u,v]_\fg)=[\phi(u),\phi(v)]_{\wt\fg}$, as in Theorem~\ref{thm:pullback-representation}. Proposition~\ref{prop:cohomology_functoriality} gives the induced functorial morphism of cohomology sheaves.
\end{example}

\begin{corollary}[Current algebroid morphisms]\label{cor:current-algebroid-morphisms}
Let $K$ be a compact smooth manifold and let $(F,f)\colon A\to\wt A$
be a first-order Lie algebroid morphism between first-order Banach Lie
algebroids. Assume that the bases of both algebroids admit local additions.
Then $F_*:=C^\infty(K,F)\colon C^\infty(K,A)\to C^\infty(K,\wt A)$
is a first-order Lie algebroid morphism over $f_*:=C^\infty(K,f)$.
In turn, for each compact smooth manifold $K$, the assignment
$A\mapsto C^\infty(K,A)$ defines a functor from the category of
first-order Banach Lie algebroids whose bases admit local additions
to the category of first-order Fr\'echet-Lie algebroids.
\end{corollary}

\begin{proof}
By Theorem~\ref{lemma:current-first-order}, the current bundles are
first-order Fr\'echet-Lie algebroids. The map $f_*$ is smooth. Fix $\gamma\in C^\infty(K,M)$
and use the current-bundle trivializations of
Theorem~\ref{lemma:current-first-order} near $\gamma$ and $f\circ\gamma$.
Denote their fiberwise comparison maps by $\Theta^A_{k,x}$ and
$\Theta^{\wt A}_{k,y}$. In these trivializations the fiber part $Q$
of $F_*$ satisfies
\[
Q(\eta,u)(k)=\Theta^{\wt A}_{k,f(\eta(k))}\circ F_{\eta(k)}\circ
(\Theta^A_{k,\eta(k)})^{-1}\bigl(u(k)\bigr).
\]
The expression on the right is jointly smooth in the local variables
$(k,x,w)$. Hence $Q$ is smooth by the exponential law for sections,
and $F_*$ is a smooth vector bundle morphism.

For fixed $k$, write $Q_k(x,w)$ for the expression on the right
with $\eta(k)=x$ and $u(k)=w$. This is the local representative of $F$
in the $k$-dependent trivializations of $A$ and $\wt A$. The two identities of
Lemma~\ref{lem:local-form-morphism} therefore hold for these
representatives. Their derivatives in the base variable are evaluated
pointwise: if $h\in T_\eta C^\infty(K,M)$, then, in local coordinates,
$d_1Q(\eta,v;h)(k)=d_1Q_k(\eta(k),v(k);h(k))$.
The current anchors and structure coefficients are likewise given
pointwise in these trivializations. Evaluating at each $k\in K$ thus
gives both identities of Lemma~\ref{lem:local-form-morphism} for $F_*$.
By the converse implication of that lemma, $F_*$ is a first-order Lie
algebroid morphism. Preservation of identities and compositions is
immediate from the corresponding identities for mapping-space maps.
\end{proof}

\section{Lie II theorem with Banach target}\label{sec:Lie-2nd-Banach}
Throughout this section, smooth maps defined on $[0,1]$ or $[0,1]^2$ mean restrictions of smooth maps defined on open neighborhoods in $\R$ or $\R^2$. One may also work with compact manifolds with corners.
Let $\G\tto M$ be a locally convex Lie groupoid. For $g\in \G$, right multiplication
$R_g\colon \G_{\tgt(g)}\to \G_{\src(g)}$, $h\mapsto hg$, is a smooth
diffeomorphism. We define the right Maurer-Cartan form by
$\omega^R(v_g):=T_gR_{g^{-1}}(v_g)$ for $v_g\in \ker(T\src)_g$. Thus
$\omega^R(v_g)\in \LG_{\tgt(g)}$.
Fix $x\in M$, and write $\tgt_x:=\tgt|_{\G_x}\colon \G_x\to M$. We denote by
$\omega_x^R\colon T(\G_x)\to \LG$ the bundle map over $\tgt_x$ obtained by
restricting $\omega^R$ to $T(\G_x)$. If $g\colon I\to \G_x$ is smooth, its
right logarithmic derivative is
$\rld g(t):=\omega^R_x(\dot g(t))\in \LG_{\tgt(g(t))}$.
Given a smooth map $\Gamma\colon[0,1]^2\to \G_x$ with coordinates $(s,t)$, we set
$\drs\Gamma:=\omega^R_x(\pds\Gamma)$ and
$\drt\Gamma:=\omega^R_x(\pdt\Gamma)$.

\begin{lemma}\label{lem:right-MC-source-fiber}
For each $x\in M$, the map $\omega_x^R\colon T(\G_x)\to \LG$ is a first-order Lie
algebroid morphism over $\tgt_x$. Moreover, the induced map
$T(\G_x)\to \tgt_x^*\LG$ is a smooth vector bundle isomorphism over $\G_x$.
\end{lemma}

\begin{proof}
Obviously for  $g\in \G_x$, 
$(\omega_x^R)_g=T_gR_{g^{-1}}\colon T_g(\G_x)\to \LG_{\tgt(g)}$ is a
topological linear isomorphism. Since $\mult$ and $\inv$ are
smooth, the induced map $T(\G_x)\to \tgt_x^*\LG$ is a smooth vector bundle
isomorphism.
Now let $U\subseteq M$ be open and let $\alpha\in \Omega^k_{\Lie(\G)}(U)$. By the definition of $\Psi$ in Lemma~\ref{lemma:Omega-risv-isomorphism}, $\Psi_U(\alpha)|_{\G_x\cap\tgt^{-1}(U)}=(\omega_x^R)^*\alpha$. Since $\Psi$ is an isomorphism of sheaves of dg algebras and $\ddRs$ restricts on $\G_x$ to the ordinary de Rham 
differential,$$
(\omega_x^R)^*(d_{\LG}\alpha)
= \Psi_U(d_{\LG}\alpha)|_{\G_x\cap\tgt^{-1}(U)}
= \ddR((\omega_x^R)^*\alpha).
$$
Hence $\omega_x^R$ is a first order Lie algebroid morphism over $\tgt_x$.
\end{proof}

\begin{lemma}\label{lem:MC-source-fiber}
Let $A\to M$ be a first-order locally convex Lie algebroid. Let $\Psi\colon T([0,1]^2)\to A$ be a
first-order Lie algebroid morphism over $\gamma\colon [0,1]^2\to M$. Put $a:=\Psi(\pdt)$ and $b:=\Psi(\pds)$. Then
$\anch(a)=\pdt\gamma$ and $\anch(b)=\pds\gamma$. If $\psi\colon A|_U\to U\times\sF$ is a local trivialization and $\gamma$ takes values in $U$, then $\pds a_\psi-\pdt b_\psi=C_\psi(\gamma,a_\psi,b_\psi)$.

In particular, if $\Gamma\colon[0,1]^2\to\G_x$ is smooth for a locally convex Lie groupoid $\G\tto
M$, these identities hold with $\gamma=\tgt\circ\Gamma$, $a=\drt\Gamma$ and $b=\drs\Gamma$. If
$\Gamma(s,0)=\one_x$, then $b(s,0)=0$; if $\Gamma(s,1)$ is independent of $s$, then $b(s,1)=0$.
\end{lemma}

\begin{proof}
The anchor identities follow from anchor compatibility for $\Psi$.
Let $\lambda\in\sF'$ and treat $\lambda$ as a constant local $A$-form on $U$ in
the trivialization $\psi$. Since $\Psi$ is a Lie algebroid morphism, we have
$d(\Psi^*\lambda)=\Psi^*(d_A\lambda)$. Moreover
$(\Psi^*\lambda)(\pdt)=\lambda(a_\psi)$ and
$(\Psi^*\lambda)(\pds)=\lambda(b_\psi)$. Hence
$
d(\Psi^*\lambda)(\pds,\pdt)
= \lambda(\pds a_\psi-\pdt b_\psi).$
We compute $d_A\lambda$ using $\psi$. Let $y\in U$ and let
$u,v\in\sF$. Let $\xi_u$ and $\xi_v$ be the constant local sections with
values $u$ and $v$. Since $\lambda$ is constant in the trivialization,
$\anch(\xi_u).\lambda(\xi_v)=0$ and $\anch(\xi_v).\lambda(\xi_u)=0$. Therefore
$$
(d_A\lambda)_y(\xi_u(y),\xi_v(y))
= -\lambda([\xi_u,\xi_v]_A(y))
= -\lambda(C_\psi(y,u,v)).
$$
We then have
$
\Psi^*(d_A\lambda)(\pds,\pdt)
= \lambda(C_\psi(\gamma,a_\psi,b_\psi)).$
Thus
$$
\lambda(\pds a_\psi-\pdt b_\psi-C_\psi(\gamma,a_\psi,b_\psi))=0
$$
for all $\lambda\in\sF'$. By the Hahn-Banach theorem, the result follows.

The final statement follows by noting that 
the map $T\Gamma\colon T([0,1]^2)\to T(\G_x)$ is a Lie algebroid morphism.
By Lemma~\ref{lem:right-MC-source-fiber}, $\omega_x^R\colon T(\G_x)\to\LG$ is
a Lie algebroid morphism. Hence $\omega_x^R\circ T\Gamma$ is a Lie algebroid
morphism over $m=\tgt\circ\Gamma$. 
If $\Gamma(s,0)=\one_x$, then $\pds\Gamma(s,0)=0$, so $b(s,0)=0$. If
$\Gamma(s,1)$ is independent of $s$, then $\pds\Gamma(s,1)=0$, so
$b(s,1)=0$.
\end{proof}

\begin{remark}\label{remark:source-adapted-charts-exp-map}
Let $\G\tto M$ be a Banach-Lie groupoid. Since $\src$ is a split submersion
and $\one\colon M\to\G$ is a section of $\src$, for every $x\in M$ there are
open neighborhoods $U\subseteq M$ of $x$, $V\subseteq\sF$ of $0$, and
$W\subseteq\G$ of $\one_x$, and a diffeomorphism
$\chi\colon U\times V\to W$ such that $\src(\chi(y,v))=y$ and
$\chi(y,0)=\one_y$. The induced local trivialization
$\psi\colon \LG|_U\to U\times\sF$ is given by
$\psi(T_{(y,0)}\chi(0,u))=(y,u)$.
\end{remark}

\begin{lemma}\label{lem:param-reconstruction}
Let $\G\tto M$ be a Banach-Lie groupoid with Hausdorff source fibers. Let $P$ be a locally convex manifold, and let
$\gamma\colon P\times[0,1]\to M$ and $\eta\colon P\times[0,1]\to \LG$ be
smooth maps such that $\eta(p,t)\in \LG_{\gamma(p,t)}$ and
$\anch(\eta(p,t))=\partial_t\gamma(p,t)$. Then for every $p_0\in P$, there
exist an open neighborhood $P_0$ of $p_0$ and a unique smooth map
$g\colon P_0\times[0,1]\to \G$ such that $g(p,0)=\one_{\gamma(p,0)}$,
$\drt g(p,t)=\eta(p,t)$, $\src(g(p,t))=\gamma(p,0)$, and
$\tgt(g(p,t))=\gamma(p,t)$.
\end{lemma}

\begin{proof}
Fix $p_0\in P$. We work locally so we shrink $P$
 to a chart domain about $p_0$.
For each $\tau\in[0,1]$, choose a  chart
$\chi\colon U\times V\to W$ about $\one_{\gamma(p_0,\tau)}$ inducing a local
trivialization $\psi\colon \LG|_U\to U\times\sF$ as in
Remark~\ref{remark:source-adapted-charts-exp-map}. After shrinking, choose an
open neighborhood $S\subseteq U$ of $\gamma(p_0,\tau)$ and an open ball
$B(0,r)\subseteq V$ such that $\tgt(\chi(S\times B(0,r)))\subseteq U$. Define
$B\colon S\times B(0,r)\to\mathcal L(\sF)$ by
$$
\psi\left(\omega^R(T_{(x,v)}\chi(0,w))\right)
= (\tgt(\chi(x,v)),B(x,v)w).
$$
The map $(x,v,w)\mapsto B(x,v)w$ is smooth and linear in $w$. By
Lemma~\ref{lem:smooth-linear-expo}, $B$ is smooth as a map to
$\mathcal L(\sF)$. Also $B(x,0)=\id$. After shrinking $S$ and $r$, we may assume
$\|B(x,v)-\id\|<1/2$ on $S\times B(0,r)$. Hence $B(x,v)$ is invertible and
$\|B(x,v)^{-1}\|\le 2$.
By compactness of $\gamma(\{p_0\}\times[0,1])$, we can choose a subdivision
$0=t_0<t_1<\cdots<t_N=1$ and, for each $I_j=[t_{j-1},t_j]$, a chart
$\chi_j\colon U_j\times V_j\to W_j$, an open set $S_j\subseteq U_j$, a radius
$r_j>0$, a trivialization $\psi_j\colon \LG|_{U_j}\to U_j\times\sF_j$, and a map
$B_j$ as above, such that $\gamma(p_0,I_j)\subseteq S_j$ and
$\|B_j(x,v)^{-1}\|\le 2$ on $S_j\times B(0,r_j)$. By
the tube lemma, there exists an open neighborhood $P_0\subseteq P$ of $p_0$ such that
$\gamma(P_0\times I_j)\subseteq S_j$ for every $j$.
On $P_0\times I_j$, write $\psi_j(\eta(p,t))=(\gamma(p,t),\alpha_j(p,t))$.
After shrinking $P_0$, choose $A_j\ge0$ such that
$\|\alpha_j(p,t)\|\le A_j$ on $P_0\times I_j$. Refining the subdivision  and reindexing
if necessary, we may
assume $2A_j(t_j-t_{j-1})<r_j$ for every $j$.
Fix $j$ and put $x_j(p)=\gamma(p,t_{j-1})$. Consider the ODE
$$
\partial_t c_j(p,t)=B_j(x_j(p),c_j(p,t))^{-1}\alpha_j(p,t),
\qquad
c_j(p,t_{j-1})=0.
$$
The right-hand side is smooth in $(p,t,c_j)$ on
$P_0\times I_j\times B(0,r_j)$. 
The Banach ODE theorem with parameters gives a
unique local smooth solution $c_j$ depending smoothly on $p$.
On its interval of
existence, $\|\partial_t c_j(p,t)\|\le 2A_j$, hence
$\|c_j(p,t)\|\le 2A_j(t-t_{j-1})<r_j$. Thus the solution exists on all of $I_j$.

Define $\widetilde g_j(p,t)=\chi_j(x_j(p),c_j(p,t))$. Then
$\src(\widetilde g_j(p,t))=\gamma(p,t_{j-1})$ and
$\widetilde g_j(p,t_{j-1})=\one_{\gamma(p,t_{j-1})}$. Put
$\widetilde\gamma_j=\tgt\circ\widetilde g_j$. By construction,
$\widetilde\gamma_j(P_0\times I_j)\subseteq U_j$ and
$$\psi_j(\drt\widetilde g_j(p,t))=(\widetilde\gamma_j(p,t),\alpha_j(p,t)).$$
Applying the anchor gives
$\partial_t\widetilde\gamma_j=\anch_j(\widetilde\gamma_j,\alpha_j)$, where
$\anch_j$ is the local representative of the anchor in $\psi_j$. Since
$\anch(\eta)=\partial_t\gamma$ and $\psi_j(\eta)=(\gamma,\alpha_j)$, the curve
$\gamma$ solves the same ODE with the same initial value at $t_{j-1}$. Uniqueness
gives $\widetilde\gamma_j=\gamma$ on $P_0\times I_j$. Hence
$\drt\widetilde g_j=\eta$ on $P_0\times I_j$.

We now concatenate the pieces. Put $h_0(p)=\one_{\gamma(p,0)}$. For
$j=1,\ldots,N$, define inductively $g_j(p,t)=\widetilde g_j(p,t)h_{j-1}(p)$ on
$I_j$, and set $h_j(p)=g_j(p,t_j)$. The product is defined because
$\src(\widetilde g_j(p,t))=\gamma(p,t_{j-1})=\tgt(h_{j-1}(p))$. 
Hence $\src(g_j(p,t))=\gamma(p,0)$,
$\tgt(g_j(p,t))=\gamma(p,t)$, and $\drt g_j=\eta$. Also
$g_j(p,t_j)=h_j(p)=g_{j+1}(p,t_j)$. Thus the $g_j$ define a continuous
piecewise smooth map $g\colon P_0\times[0,1]\to\G$.
It remains to check smoothness at the breakpoints. 
Fix $j$ and put
$h(p)=g(p,t_j)$. Near $(p_0,t_j)$, set $q(p,t)=g(p,t)h(p)^{-1}$. Then
$q(p,t_j)=\one_{\gamma(p,t_j)}$, $\src(q(p,t))=\gamma(p,t_j)$, and
$\drt q=\eta$, since $h(p)^{-1}$ is independent of $t$. In a source-adapted chart
around $\one_{\gamma(p_0,t_j)}$, the left and right pieces of $q$ solve the same
Banach-space ODE with initial value $0$ at $t_j$. By local uniqueness, they agree
near $t_j$. Hence $g$ is smooth across each breakpoint. Since the subdivision is
finite, $g$ is smooth on $P_0\times[0,1]$.

For uniqueness, let $\bar g$ be another solution. Fix $p\in P_0$. The curves
$g(p,\cdot)$ and $\bar g(p,\cdot)$ lie in the Hausdorff source fiber
$\G_{\gamma(p,0)}$. Let $E=\{t\in[0,1]:g(p,t)=\bar g(p,t)\}$. Then $E$ is
nonempty. It is closed because $\G_{\gamma(p,0)}$ is Hausdorff. If $t_*\in E$,
right multiply both curves by $g(p,t_*)^{-1}$ and use a source-adapted chart at
$\one_{\gamma(p,t_*)}$. Local ODE uniqueness shows that $E$ is open. Hence
$E=[0,1]$. Since $p$ was arbitrary, $g=\bar g$.
\end{proof}
The curve $\eta$ is called admissible over $\gamma$, and the curve $g$ is
called the reconstruction of $\eta$.

\begin{proposition}\label{prop:homotopy-invariance}
Let $\G\tto M$ be a locally convex Lie groupoid and let $\cH\tto N$ be a Banach-Lie groupoid with Hausdorff source fibers. Let
$(\Phi,f)\colon \LG\to\Lie(\cH)$ be a first-order Lie algebroid morphism. Let
$x\in M$, and let $\Gamma\colon[0,1]^2\to \G_x$ be smooth with
$\Gamma(s,0)=\one_x$ and $\Gamma(s,1)$ independent of $s$. For each $s$, let
$\Delta_s\colon[0,1]\to \cH_{f(x)}$ be the reconstruction starting at
$\one_{f(x)}$ of the $\Lie(\cH)$-path $\Phi(\drt\Gamma(s,t))$. Then
$\Delta_s(1)$ is independent of $s$.
\end{proposition}

\begin{proof}
Put $m=\tgt_\G\circ\Gamma$, $a=\drt\Gamma$, and $b=\drs\Gamma$. By
Lemma~\ref{lem:MC-source-fiber}, $\anch_\G(a)=\partial_t m$,
$\anch_\G(b)=\partial_s m$, and $b(s,0)=b(s,1)=0$. By anchor compatibility,
$\anch_\cH(\Phi(a))=Tf(\anch_\G(a))=\partial_t(f\circ m)$. Hence $\Phi(a)$ is
admissible over $f\circ m$.
By Lemma~\ref{lem:param-reconstruction}, applied locally with parameter $s$,
the reconstructions form a smooth map $\Delta\colon[0,1]^2\to\cH$. On
overlaps the locally constructed maps agree by uniqueness. Hence
$\drt\Delta=\Phi(a)$, $\src_\cH(\Delta)=f(x)$, and
$\tgt_\cH(\Delta)=f\circ m$.
Set $d=\drs\Delta$. We prove $d=\Phi(b)$. Fix $s_0\in[0,1]$. The compact set
$f(m(s_0,[0,1]))$ is covered by finitely many local trivializations of
$\Lie(\cH)$. Choose a subdivision $0=t_0<\cdots<t_N=1$ such that each subcurve
over $I_j=[t_{j-1},t_j]$ lies in one trivialization domain $U_j$. After
shrinking to a neighborhood $S$ of $s_0$, we may assume
$f(m(S\times I_j))\subseteq U_j$ for all $j$.
On $S\times I_j$, write the local representatives of $\Phi(a)$, $\Phi(b)$,
and $d$ as $A_j$, $B_j$, and $D_j$. Let $C_j^\cH$ be the local structure map of
$\Lie(\cH)$ in the chosen local trivialization. The morphism
$\Phi\circ\omega_{\G,x}^R\circ T\Gamma$ gives, by
Lemma~\ref{lem:MC-source-fiber},
$$
\partial_sA_j-\partial_tB_j=C_j^\cH(f\circ m,A_j,B_j).
$$
The morphism $\omega_{\cH,f(x)}^R\circ T\Delta$ gives
$$
\partial_sA_j-\partial_tD_j=C_j^\cH(f\circ m,A_j,D_j).
$$
Subtracting gives
$$
\partial_t(D_j-B_j)=-C_j^\cH(f\circ m,A_j,D_j-B_j).
$$
For fixed $s$, this is a linear Banach-space ODE in $D_j-B_j$.
At $t=0$, $d(s,0)=0$ because $\Delta(s,0)=\one_{f(x)}$ is independent of $s$,
and $\Phi(b(s,0))=0$. Hence $D_1(s,0)-B_1(s,0)=0$.  Uniqueness gives
$D_1=B_1$ on $S\times I_1$. At each next subdivision point this equality is an
equality of algebroid elements, so it remains true in the next trivialization.
Induction gives $d=\Phi(b)$ on $S\times[0,1]$. Since $s_0$ was arbitrary, this
holds globally.
At $t=1$, $b(s,1)=0$, so $d(s,1)=0$. Since
$d(s,1)=\drs(\Delta(s,1))$, and since the right Maurer-Cartan map is a
pointwise linear isomorphism on $\cH_{f(x)}$, we get
$\partial_s\Delta(s,1)=0$. Thus $s\mapsto\Delta_s(1)$ is locally constant.
Since $[0,1]$ is connected, it is constant.
\end{proof}

\begin{lemma}\label{lem:smooth-local-source-paths}
Let $\G\tto M$ be a locally convex Lie groupoid. Let $g_0\in \G$ lie in the connected
component of $\one_{\src(g_0)}$ in $\G_{\src(g_0)}$. Then
there are an open neighborhood $U_g\subseteq \G$ of $g_0$ and a smooth map
$P\colon U_g\times[0,1]\to \G$ such that, for every $h\in U_g$, one has
$\src(P(h,t))=\src(h)$, $P(h,0)=\one_{\src(h)}$, and $P(h,1)=h$.
\end{lemma}

\begin{proof}
Let $x_0=\src(g_0)$. Since locally convex manifolds are locally path connected, the
connected component of $\one_{x_0}$ in $\G_{x_0}$ is its path component. Choose
a continuous path $c\colon[0,1]\to\G_{x_0}$ from $\one_{x_0}$ to $g_0$.
We can choose finitely many source-adapted charts
$\kappa_i\colon V_i\to O_i\times W_i$ so that $\src=\pr_1\circ\kappa_i$, where
each $W_i$ is open and convex, and choose a subdivision
$0=t_0<\cdots<t_n=1$ such that $c([t_{i-1},t_i])\subseteq V_i$ so that $c(t_i)\in V_i\cap V_{i+1}$ for
$1\le i\leq n-1$.
Choose smooth local sections $\sigma_i$ of $\src$,
defined near $x_0$, with $\sigma_i(x_0)=c(t_i)$ and image in
$V_i\cap V_{i+1}$. Set $\sigma_0(x)=\one_x$ near $x_0$, and choose a local
section $\sigma_n$ through $g_0$. After shrinking, all $\sigma_i$ are defined
on an open neighborhood $O$ of $x_0$, and
$\sigma_{i-1}(O)\cup\sigma_i(O)\subseteq V_i$ for $i=1,\ldots,n$.
Choose a smooth function $\beta\colon[0,1]\to[0,1]$ with $\beta=0$ near $0$
and $\beta=1$ near $1$. In the chart $\kappa_i$, write
$\kappa_i(\sigma_{i-1}(x))=(x,u_i^-(x))$ and
$\kappa_i(\sigma_i(x))=(x,u_i^+(x))$. Define
$\Sigma_i(x,\tau)=\kappa_i^{-1}(x,(1-\beta(\tau))u_i^-(x)+\beta(\tau)u_i^+(x))$.
After the preceding shrinkings, convexity of $W_i$ makes this well defined.
Then $\src(\Sigma_i(x,\tau))=x$, $\Sigma_i(x,0)=\sigma_{i-1}(x)$, and
$\Sigma_i(x,1)=\sigma_i(x)$. Moreover $\Sigma_i$ is constant in $\tau$ near
$0$ and $1$.
Shrink an open neighborhood $U_g\subseteq V_n$ of $g_0$ so that
$\src(U_g)\subseteq O$. Write $\kappa_n(h)=(\src(h),q(h))$, and write
$\kappa_n(\sigma_n(x))=(x,u_{n+1}^-(x))$. Define
$$
\Sigma_{n+1}(h,\tau)
= \kappa_n^{-1}(\src(h),(1-\beta(\tau))u_{n+1}^-(\src(h))+\beta(\tau)q(h)).
$$
After shrinking $U_g$, convexity keeps the segment in $W_n$. Thus
$\Sigma_{n+1}$ is a smooth path from $\sigma_n(\src(h))$ to $h$, inside the
source fiber over $\src(h)$, and it is constant near the endpoints.
Divide $[0,1]$ into $n+1$ equal sub intervals. On the first $n$ subintervals,
define $P(h,t)$ by $\Sigma_i(\src(h),\tau)$, where $\tau$ is the affine
reparametrization to $[0,1]$. On the final subinterval, define
$P(h,t)=\Sigma_{n+1}(h,\tau)$. Adjacent pieces agree on neighborhoods of the
breakpoints because $\beta$ is constant near $0$ and $1$. Hence $P$ is smooth.
The source and endpoint identities are immediate.
\end{proof}

We now prove the Banach analogue of Lie II theorem for groupoids.

\begin{theorem}[Lie II with Banach targets]\label{thm:Lie-II-Banach}
Let $\G\tto M$ be a locally convex Lie groupoid and $\cH\tto N$ a Banach-Lie groupoid, both
with Hausdorff source fibers. Assume that $\G$ is
source connected and source simply connected.
Let $(\Phi,f)\colon\LG\to\Lie(\cH)$ be a first-order Lie algebroid morphism over $f\colon M\to N$.
Then there exists a unique smooth groupoid morphism $F\colon\G\to\cH$ over $f$ such that
$\Lie(F)=\Phi$. 
\end{theorem}

\begin{proof}
Let $g\in \G$, and put $x=\src_\G(g)$. Since $\G_x$ is connected and locally
path connected, $g$ lies in the path component of $\one_x$. By
Lemma~\ref{lem:smooth-local-source-paths}, there is a smooth path
$\Gamma\colon[0,1]\to\G_x$ from $\one_x$ to $g$. Put
$a_\Gamma=\drt\Gamma$. Then
$\anch_\G(a_\Gamma)=\partial_t(\tgt_\G\circ\Gamma)$. Hence
$\Phi(a_\Gamma)$ is admissible over $f\circ\tgt_\G\circ\Gamma$.
Let $\Delta_\Gamma$ be the reconstruction in $\cH$, starting at
$\one_{f(x)}$, with $\drt\Delta_\Gamma=\Phi(a_\Gamma)$. Define
$F(g)=\Delta_\Gamma(1)$. Then
$\src_\cH(F(g))=f(\src_\G(g))$ and
$\tgt_\cH(F(g))=f(\tgt_\G(g))$.
This definition is independent of $\Gamma$. 

If $\Gamma_0$ and $\Gamma_1$ are two such paths, source-simply-connectedness gives a continuous homotopy $H\colon[0,1]^2\to\G_x$ with fixed endpoints.
Choose a smooth $\beta\colon[0,1]\to[0,1]$ which is zero near $0$ and one near $1$,
with $\beta(0)=0$ and $\beta(1)=1$. Then
$\widehat H(s,t):=H(\beta(s),\beta(t))$ is smooth near $\partial[0,1]^2$.
Since $\G_x$ is a Hausdorff locally convex manifold,
\cite[Theorem~11]{Wockel2009SteenrodApproximation}, applied to the trivial bundle over
$[0,1]^2$ with $L=[0,1]^2$ and $U=(0,1)^2$, gives a smooth homotopy equal to
$\widehat H$ on the boundary. The reparameterized paths are smoothly homotopic
to the original paths relative to their endpoints. 
For a smooth path $\Gamma$ with fixed endpoints and the chosen $\beta$,
the smooth homotopy $\Gamma((1-r)t+r\beta(t))$, $r\in [0,1]$, joins
$\Gamma(t)$ to $\Gamma(\beta(t))$ (with fixed endpoints).
Proposition~\ref{prop:homotopy-invariance} shows
that their reconstructions have the same endpoint. Thus $F$ is well defined and preserves source
and target.

We prove smoothness. Fix $g_0\in\G$. Lemma~\ref{lem:smooth-local-source-paths}
gives an open neighborhood $U_g$ of $g_0$ and a smooth family
$P\colon U_g\times[0,1]\to\G$ with $P(h,0)=\one_{\src(h)}$ and $P(h,1)=h$.
Set $\gamma(h,t)=f(\tgt_\G(P(h,t)))$ and
$\eta(h,t)=\Phi(\drt P(h,t))$. Then
$\anch_\cH(\eta)=\partial_t\gamma$. By Lemma~\ref{lem:param-reconstruction},
after shrinking $U_g$, the reconstructions form a smooth map
$\Delta\colon U_g\times[0,1]\to\cH$. By construction, $F(h)=\Delta(h,1)$ on
$U_g$. Hence $F$ is smooth.

We prove multiplicativity. Let $g_1,g_2\in\G$ be composable with
$\src_\G(g_2)=\tgt_\G(g_1)$. Choose smooth source-fiber paths $\Gamma_1$ from
$\one_{\src_\G(g_1)}$ to $g_1$, and $\Gamma_2$ from
$\one_{\src_\G(g_2)}$ to $g_2$. Choose a smooth map
$\theta\colon[0,1]\to[0,1]$ with $\theta=0$ near $0$ and $\theta=1$ near $1$.
Replacing $\Gamma_i$ by $\Gamma_i\circ\theta$ does not change the value of
$F(g_i)$, by the preceding well-definedness argument. Thus we may assume that
$\Gamma_1$ and $\Gamma_2$ are constant near their endpoints.
Define a smooth path $\Gamma$ from $\one_{\src_\G(g_1)}$ to $g_2g_1$ by
$\Gamma(t)=\Gamma_1(2t)$ for $t\le 1/2$, and
$\Gamma(t)=\Gamma_2(2t-1)g_1$ for $t\ge 1/2$. Let $\Delta_i$ reconstruct
$\Phi(\drt\Gamma_i)$, with initial points $\one_{f(\src_\G(g_i))}$. Then
$\Delta_i(1)=F(g_i)$. Define $\Delta$ by
$\Delta(t)=\Delta_1(2t)$ for $t\le 1/2$, and
$\Delta(t)=\Delta_2(2t-1)\Delta_1(1)$ for $t\ge 1/2$. The paths are smooth
because the pieces are constant near the joining point.
On the first half, $\drt\Delta(t)=2\Phi(\drt\Gamma_1(2t))$. On the second
half, right multiplication by the fixed arrow $\Delta_1(1)$ does not change
the right logarithmic derivative, so
$\drt\Delta(t)=2\Phi(\drt\Gamma_2(2t-1))$. The same statements hold for
$\Gamma$, with $g_1$ in place of $\Delta_1(1)$. Hence
$\drt\Delta=\Phi(\drt\Gamma)$. By uniqueness of reconstruction,
$F(g_2g_1)=\Delta(1)=F(g_2)F(g_1)$. The constant path gives
$F(\one_x)=\one_{f(x)}$. Thus $F$ is a groupoid morphism over $f$.

We show that $\Lie(F)=\Phi$. Let $a\in\LG_x$. Choose a source-adapted chart
$\chi\colon U\times V\to W$ around $\one_x$, with induced trivialization
$\psi\colon \LG|U\to U\times\sF$. Write $\psi(a)=(x,v)$. For small $r$, set
$\sigma(r)=\chi(x,rv)$ and $\Gamma_r(t)=\chi(x,trv)$. Then $\sigma'(0)=a$.
Let $\Delta(r,t)$ be the reconstruction in $\cH$ of
$\Phi(\drt\Gamma_r(t))$, starting at $\one_{f(x)}$. Then
$F(\sigma(r))=\Delta(r,1)$.
Set $\widetilde a=\drt\Delta$ and
$d=\omega_{\cH,f(x)}^R(\partial_r\Delta)$. Choose a local
trivialization $\Upsilon\colon\Lie(\cH)|_{U'}\to U'\times\sE$ around $f(x)$.
After shrinking $r$, write the local representatives as
$\widetilde a_\Upsilon$ and $d_\Upsilon$. Since $\Delta(0,t)=\one_{f(x)}$, we
have $\widetilde a_\Upsilon(0,t)=0$. Since $\Delta(r,0)=\one_{f(x)}$, we have
$d_\Upsilon(0,0)=0$.
Set
$B(x,z,w) = \pr_2(\psi(\omega^R(T_{(x,z)}\chi(0,w))))$. Hence
$B$ is  smooth and linear in $w$.
One obtains
$$
\psi(\drt\Gamma_r(t))
= (\tgt_\G(\Gamma_r(t)),B(x,trv,rv)),
$$
with $B(x,0,w)=w$. Therefore
$\left.\partial_r\right|_{r=0}B(x,trv,rv)=v$. Write
$\Upsilon(\Phi(a))=(f(x),w)$. Since the local representative of $\Phi$ is
smooth and fiberwise linear, its derivative in the base direction at the zero
fiber vector is zero. Differentiating at the zero fiber vector gives
$\partial_r\widetilde a_\Upsilon(0,t)=w$.
Apply Lemma~\ref{lem:MC-source-fiber} to the two-parameter family
$\Delta(r,t)$ in the source fiber $\cH_{f(x)}$. In the trivialization
$\Upsilon$,
$$
\partial_r\widetilde a_\Upsilon-\partial_t d_\Upsilon
= C_\Upsilon(\tgt_\cH\circ\Delta,\widetilde a_\Upsilon,d_\Upsilon).
$$
At $r=0$, the right hand side is zero because
$\widetilde a_\Upsilon(0,t)=0$. Hence
$\partial_t d_\Upsilon(0,t)=\partial_r\widetilde a_\Upsilon(0,t)=w$. Since
$d_\Upsilon(0,0)=0$, we get $d_\Upsilon(0,1)=w$. At the identity, the right
Maurer-Cartan map is the canonical identification
$T_{\one_{f(x)}}(\cH_{f(x)})=\Lie(\cH)_{f(x)}$. Therefore
$$
\Lie(F)(a)
= \left.\frac{d}{dr}\right|_{r=0}F(\sigma(r))
= \left.\frac{d}{dr}\right|_{r=0}\Delta(r,1)
= d(0,1)
= \Phi(a).
$$
Finally, let $F_1,F_2\colon\G\to\cH$ be smooth groupoid morphisms over $f$
with $\Lie(F_1)=\Lie(F_2)=\Phi$. Fix $g\in\G$, and choose a smooth
source-fiber path $\Gamma$ from $\one_{\src_\G(g)}$ to $g$. For a groupoid
morphism $F_i$, differentiating
$F_i\circ R_{\Gamma(t)^{-1}}=R_{F_i(\Gamma(t))^{-1}}\circ F_i$ at
$\Gamma(t)$ gives
$$\drt(F_i\circ\Gamma)=\Lie(F_i)(\drt\Gamma)=\Phi(\drt\Gamma).$$ Thus
$F_1\circ\Gamma$ and $F_2\circ\Gamma$ solve the same reconstruction problem in
the Hausdorff source fiber of $\cH$, with the same initial value. By
uniqueness, they agree. Evaluating at $t=1$ gives $F_1(g)=F_2(g)$. Hence
$F_1=F_2$.
\end{proof}

\begin{remark}
In particular, for a source-connected and source-simply connected locally convex Lie groupoid $\G$
 and a Banach-Lie groupoid $\cH$ (both with Hausdorff source fibers),
differentiation is a bijection $\Hom_{\mathrm{Grpd}}(\G,\cH)\cong\Hom_{\mathrm{Algd}}(\LG,\Lie(\cH))$.

The uniqueness argument also applies to a source-connected locally convex Lie groupoid $\G$ without
Hausdorff source fibers, whenever the smooth morphisms in question exist and the target is a
Banach-Lie groupoid with Hausdorff source fibers.
\end{remark}

 If $\G\tto M$ is a Banach-Lie groupoid, then by
\cite[Theorem~5.1]{BeltitaGolinskiJakimowiczPelletier2019}, the
source-connected component subgroupoid $\G^0\subseteq \G$ admits a
source-simply connected covering groupoid $\wt\G\to \G^0$.
We write $p_\G\colon\wt\G\to\G$ for the composition
$\wt\G\to\G^0\hookrightarrow\G$. Thus $\wt\G\tto M$ is source connected and
source simply connected, and
$\Lie(p_\G)\colon\Lie(\wt\G)\to\Lie(\G)$ is an isomorphism over $\id_M$.
We use the same notation for $\cH$.

\begin{corollary}\label{cor:functoriality-ssc-cover}
Let $\G\tto M$ and $\cH\tto N$ be Banach-Lie groupoids with Hausdorff source fibers. Let
$F\colon\G\to\cH$ be a Lie groupoid morphism over $f\colon M\to N$. Then
there exists a unique Lie groupoid morphism $\wt F\colon\wt\G\to\wt\cH$ over
$f$ such that $p_\cH\circ\wt F=F\circ p_\G$.
\end{corollary}
\begin{proof}
Since $F$ is a Lie groupoid morphism, $\Lie(F)\colon\Lie(\G)\to\Lie(\cH)$ is
a Lie algebroid morphism over $f$. Set
$$
\wt\Phi:=\Lie(p_\cH)^{-1}\circ\Lie(F)\circ\Lie(p_\G)
\colon
\Lie(\wt\G)\to\Lie(\wt\cH).
$$
Then $\wt\Phi$ is a Lie algebroid morphism over $f$. By
Theorem~\ref{thm:Lie-II-Banach}, there exists a unique Lie groupoid morphism
$\wt F\colon\wt\G\to\wt\cH$ over $f$ such that $\Lie(\wt F)=\wt\Phi$.
Both maps $p_\cH\circ\wt F$ and $F\circ p_\G$ are Lie groupoid morphisms
$\wt\G\to\cH$ over $f$. Moreover,
$$
\Lie(p_\cH\circ\wt F)
= \Lie(p_\cH)\circ\Lie(\wt F)
= \Lie(F)\circ\Lie(p_\G)
= \Lie(F\circ p_\G).
$$
Since $\wt\G$ is source connected and source simply connected, uniqueness in
Theorem~\ref{thm:Lie-II-Banach} gives $p_\cH\circ\wt F=F\circ p_\G$.
For uniqueness, let $\wt F_1,\wt F_2\colon\wt\G\to\wt\cH$ be Lie groupoid
morphisms over $f$ with $p_\cH\circ\wt F_1=F\circ p_\G=p_\cH\circ\wt F_2$.
Applying the Lie functor gives
$\Lie(p_\cH)\circ\Lie(\wt F_1)=\Lie(p_\cH)\circ\Lie(\wt F_2)$. Since
$\Lie(p_\cH)$ is an isomorphism, $\Lie(\wt F_1)=\Lie(\wt F_2)$. Again by
uniqueness in Theorem~\ref{thm:Lie-II-Banach}, $\wt F_1=\wt F_2$.
\end{proof}

\begin{remark}\label{rem:categorical-ssc-cover}
Let $\mathsf{LieGrpd}_{\mathrm{Ban}}$ be the category of Banach-Lie groupoids with Hausdorff source fibers, and let
$\mathsf{LieGrpd}_{\mathrm{Ban}}^{\mathrm{ssc}}$ be the full subcategory of
source connected and source simply connected Banach-Lie groupoids.
For each Banach-Lie groupoid $\G\tto M$, choose a covering morphism
$p_\G\colon\wt\G\to\G$ as above. By
Corollary~\ref{cor:functoriality-ssc-cover}, every Lie groupoid morphism
$F\colon\G\to\cH$ admits a unique lift $\wt F\colon\wt\G\to\wt\cH$ such that
$p_\cH\circ\wt F=F\circ p_\G$. The identity and composition laws follow from
the uniqueness of these lifts. Hence $\G\mapsto\wt\G$ and $F\mapsto\wt F$
define a functor
$$
\widetilde{(-)}\colon
\mathsf{LieGrpd}_{\mathrm{Ban}}\to
\mathsf{LieGrpd}_{\mathrm{Ban}}^{\mathrm{ssc}}.
$$
Moreover, the family $(p_\G)_\G$ defines a natural transformation
$p\colon\widetilde{(-)}\Rightarrow\id_{\mathsf{LieGrpd}_{\mathrm{Ban}}}$.
If $\G$ is source connected and source simply connected, we may choose
$p_\G=\id_\G$. With this choice, $\widetilde{(-)}$ is right adjoint to the
inclusion
$\iota\colon\mathsf{LieGrpd}_{\mathrm{Ban}}^{\mathrm{ssc}}\hookrightarrow
\mathsf{LieGrpd}_{\mathrm{Ban}}$. Indeed, if $\mathcal K\tto P$ is source
connected and source simply connected and $F\colon\mathcal K\to\G$ is a Lie
groupoid morphism, then Corollary~\ref{cor:functoriality-ssc-cover} gives a
unique morphism $\wt F\colon\mathcal K\to\wt\G$ such that
$p_\G\circ\wt F=F$. Thus composition with $p_\G$ gives a natural bijection
$$
\Hom_{\mathsf{LieGrpd}_{\mathrm{Ban}}^{\mathrm{ssc}}}(\mathcal K,\wt\G)
\cong
\Hom_{\mathsf{LieGrpd}_{\mathrm{Ban}}}(\mathcal K,\G).
$$
\end{remark}

\printbibliography

@article{AbadCrainic2012Representations,
  author  = {Arias Abad, Camilo and Crainic, Marius},
  title   = {Representations up to Homotopy of {Lie} Algebroids},
  journal = {Journal f{\"u}r die reine und angewandte Mathematik},
  volume  = {663},
  pages   = {91--126},
  date    = {2012},
  doi     = {10.1515/crelle.2011.095}
}

@article{Wockel2009SteenrodApproximation,
  author       = {Wockel, Christoph},
  title        = {A generalization of {Steenrod's} approximation theorem},
  journaltitle = {Archivum Mathematicum},
  year         = {2009},
  doi          = {10.48550/arXiv.math/0610252},
  volume       = {45},
  number       = {2},
  pages        = {95--104},
  eprint       = {math/0610252},
  eprinttype   = {arxiv},
  eprintclass  = {math.DG}
}

@article{AmiriGloecknerSchmeding2020CurrentGroupoids,
  author      = {Amiri, Habib and Gl{\"o}ckner, Helge and Schmeding, Alexander},
  title       = {Lie Groupoids of Mappings Taking Values in a {Lie} Groupoid},
  journal     = {Archivum Mathematicum},
  volume      = {56},
  number      = {5},
  pages       = {307--356},
  date        = {2020},
  doi         = {10.5817/AM2020-5-307},
  eprint      = {1811.02888},
  eprinttype  = {arxiv},
  eprintclass = {math.DG}
}

@article{Anastasiei2011BanachLieAlgebroids,
  author  = {Anastasiei, Mihai},
  title   = {Banach {Lie} Algebroids},
  journal = {Analele {\c{S}}tiin{\c{t}}ifice ale Universit{\u{a}}{\c{t}}ii {``Alexandru Ioan Cuza''} din Ia{\c{s}}i. Matematic{\u{a}}},
  volume  = {57},
  number  = {2},
  pages   = {409--416},
  date    = {2011},
  doi     = {10.2478/v10157-011-0035-y}
}

@article{BeltitaGolinskiJakimowiczPelletier2019,
  author  = {Belti{\c{t}}{\u{a}}, Daniel and Goli{\'n}ski, Tomasz and Jakimowicz, Grzegorz and Pelletier, Fernand},
  title   = {Banach--{Lie} Groupoids and Generalized Inversion},
  journal = {Journal of Functional Analysis},
  volume  = {276},
  number  = {5},
  pages   = {1528--1574},
  date    = {2019},
  doi     = {10.1016/j.jfa.2018.12.002}
}

@article{CabauPelletier2012AlmostLie,
  author  = {Cabau, Patrick and Pelletier, Fernand},
  title   = {Almost {Lie} Structures on an Anchored {Banach} Bundle},
  journal = {Journal of Geometry and Physics},
  volume  = {62},
  number  = {11},
  pages   = {2147--2169},
  date    = {2012},
  doi     = {10.1016/j.geomphys.2012.06.005}
}

@article{CattaneoDherinWeinstein2013Comorphisms,
  author  = {Cattaneo, Alberto S. and Dherin, Benoit and Weinstein, Alan},
  title   = {Integration of {Lie} Algebroid Comorphisms},
  journal = {Portugaliae Mathematica},
  volume  = {70},
  number  = {2},
  pages   = {113--144},
  date    = {2013},
  doi     = {10.4171/PM/1928}
}

@article{Crainic2003Cohomology,
  author  = {Crainic, Marius},
  title   = {Differentiable and Algebroid Cohomology, {Van Est} Isomorphisms, and Characteristic Classes},
  journal = {Commentarii Mathematici Helvetici},
  volume  = {78},
  number  = {4},
  pages   = {681--721},
  date    = {2003},
  doi     = {10.1007/s00014-001-0766-9}
}

@article{CrainicFernandes2003Integrability,
  author  = {Crainic, Marius and Fernandes, Rui Loja},
  title   = {Integrability of {Lie} Brackets},
  journal = {Annals of Mathematics},
  volume  = {157},
  number  = {2},
  pages   = {575--620},
  date    = {2003},
  doi     = {10.4007/annals.2003.157.575}
}

@article{EvensLuWeinstein1999Transverse,
  author  = {Evens, Sam and Lu, Jiang-Hua and Weinstein, Alan},
  title   = {Transverse Measures, the Modular Class and a Cohomology Pairing for {Lie} Algebroids},
  journal = {The Quarterly Journal of Mathematics},
  volume  = {50},
  number  = {200},
  pages   = {417--436},
  date    = {1999},
  doi     = {10.1093/qjmath/50.200.417}
}

@article{Fernandes2002Holonomy,
  author  = {Fernandes, Rui Loja},
  title   = {Lie Algebroids, Holonomy and Characteristic Classes},
  journal = {Advances in Mathematics},
  volume  = {170},
  number  = {1},
  pages   = {119--179},
  date    = {2002},
  doi     = {10.1006/aima.2001.2070}
}

@misc{GloecknerNeeb2026InfiniteDimensional,
  author      = {Gl{\"o}ckner, Helge and Neeb, Karl-Hermann},
  title       = {Infinite-Dimensional {Lie} Groups},
  date        = {2026},
  doi         = {10.48550/arXiv.2602.12362},
  eprint      = {2602.12362},
  eprinttype  = {arxiv},
  eprintclass = {math.FA},
  note        = {Preliminary book manuscript}
}

@article{HigginsMackenzie1990Algebraic,
  author  = {Higgins, Philip J. and Mackenzie, Kirill C. H.},
  title   = {Algebraic Constructions in the Category of {Lie} Algebroids},
  journal = {Journal of Algebra},
  volume  = {129},
  number  = {1},
  pages   = {194--230},
  date    = {1990},
  doi     = {10.1016/0021-8693(90)90246-K}
}

@article{Kubarski1991ChernWeil,
  author  = {Kubarski, Jan},
  title   = {The {Chern--Weil} Homomorphism of Regular {Lie} Algebroids},
  journal = {Publications du D{\'e}partement de math{\'e}matiques (Lyon)},
  pages   = {1--69},
  date    = {1991},
  url     = {https://www.numdam.org/item/PDML_1991____1_0/}
}

@article{Marle2008Calculus,
  author      = {Marle, Charles-Michel},
  title       = {Calculus on {Lie} Algebroids, {Lie} Groupoids and {Poisson} Manifolds},
  journal     = {Dissertationes Mathematicae},
  volume      = {457},
  pages       = {1--57},
  date        = {2008},
  doi         = {10.4064/dm457-0-1},
  eprint      = {0806.0919},
  eprinttype  = {arxiv},
  eprintclass = {math.DG}
}

@book{Mackenzie2005GeneralTheory,
  author    = {Mackenzie, Kirill C. H.},
  title     = {General Theory of {Lie} Groupoids and {Lie} Algebroids},
  series    = {London Mathematical Society Lecture Note Series},
  number    = {213},
  publisher = {Cambridge University Press},
  location  = {Cambridge},
  date      = {2005},
  doi       = {10.1017/CBO9781107325883}
}

@article{Pradines1967Lie,
  author  = {Pradines, Jean},
  title   = {Th{\'e}orie de {Lie} pour les groupo{\"i}des diff{\'e}rentiables. Calcul diff{\'e}rentiel dans la cat{\'e}gorie des groupo{\"i}des infinit{\'e}simaux},
  journal = {Comptes Rendus de l'Acad{\'e}mie des Sciences de Paris, S{\'e}rie A},
  volume  = {264},
  pages   = {245--248},
  date    = {1967},
  langid  = {french}
}

@article{Vaintrob1997Homological,
  author  = {Vaintrob, A. Yu.},
  title   = {Lie Algebroids and Homological Vector Fields},
  journal = {Russian Mathematical Surveys},
  volume  = {52},
  number  = {2},
  pages   = {428--429},
  date    = {1997},
  doi     = {10.1070/RM1997v052n02ABEH001802}
}

@article{BeltitaGolinskiTumpach2018QueerPoisson,
  author      = {Belti{\c{t}}{\u{a}}, Daniel and Goli{\'n}ski, Tomasz and Tumpach, Alice-Barbara},
  title       = {Queer {Poisson} Brackets},
  journal     = {Journal of Geometry and Physics},
  volume      = {132},
  pages       = {358--362},
  date        = {2018},
  doi         = {10.1016/j.geomphys.2018.06.013},
  eprint      = {1710.03057},
  eprinttype  = {arxiv},
  eprintclass = {math.FA}
}

@misc{GolinskiJakimowicz2025Predual,
  author      = {Goli{\'n}ski, Tomasz and Jakimowicz, Grzegorz},
  title       = {Poisson Structure on Predual of {Banach} {Lie} Algebroid},
  date        = {2025},
  doi         = {10.48550/arXiv.2505.13351},
  eprint      = {2505.13351},
  eprinttype  = {arxiv},
  eprintclass = {math.DG},
  note        = {Preprint}
}

@book{kriegl1997convenient,
  author    = {Kriegl, Andreas and Michor, Peter W.},
  title     = {The Convenient Setting of Global Analysis},
  series    = {Mathematical Surveys and Monographs},
  volume    = {53},
  publisher = {American Mathematical Society},
  location  = {Providence, RI},
  date      = {1997},
  doi       = {10.1090/surv/053}
}

\end{document}